\documentclass[12pt]{amsart}
\usepackage[active]{srcltx}
\usepackage{calc,amssymb,amsthm,amsmath,amscd, eucal,ulem,stmaryrd}
\usepackage{alltt}
\RequirePackage[dvipsnames,usenames]{color}

\input{kmacros3.sty}
\input{xy}
\xyoption{all}
\usepackage{tikz}
\usepackage{mabliautoref}
\renewcommand{\m}{\mathfrak{m}}
\renewcommand{\n}{\mathfrak{n}}
\renewcommand{\fram}{\mathfrak{m}}
\numberwithin{equation}{theorem}

\usepackage{fullpage}

\usepackage{setspace}
\usepackage{hyperref}

\usepackage{enumerate}
\usepackage{graphicx}
\usepackage[all,cmtip]{xy}

\usepackage{upgreek}
\newcommand{\mytau}{{\uptau}}

\usepackage{verbatim}

\renewcommand{\O}{\mathcal O}
\newcommand{\BCMReg}[1]{{$\textnormal{BCM}_{#1}$-regular}}
\newcommand{\BCMRat}[1]{{$\textnormal{BCM}_{#1}$-rational}}
\newcommand{\BCM}{\textnormal{BCM}{}}

\begin{document}

\title{Singularities in mixed characteristic via perfectoid big Cohen-Macaulay algebras}
\author{Linquan Ma and Karl Schwede}
\address{Department of Mathematics\\ Purdue University\\  West Lafayette\\ IN 47907}
\address{Department of Mathematics\\ University of Utah\\ Salt Lake City\\ UT 84112}
\email{ma326@purdue.edu}
\email{schwede@math.utah.edu}

\thanks{The first named author was supported in part by NSF Grant \#1836867/1600198, and was supported in part by NSF CAREER Grant DMS \#1252860/1501102 when preparing this article.}
\thanks{The second named author was supported in part by NSF CAREER Grant DMS \#1252860/1501102 and NSF grant \#1801849.}
\thanks{This material is partially based upon work supported by the National Science Foundation under Grant No. DMS-1440140 while the authors were in residence at the Mathematical Sciences Research Institute in Berkeley California during the Spring 2019 semester.}
\maketitle

\begin{abstract}
We utilize recent results of Andr\'e and Gabber on the existence of weakly functorial integral perfectoid big Cohen-Macaulay (\BCM{}) algebras to study singularities of local rings in mixed characteristic.  In particular, we introduce a mixed characteristic \BCM{}-variant of rational/$F$-rational singularities, of log terminal/$F$-regular singularities and of multiplier/test ideals of divisor pairs. We prove a number of results about these objects including a restriction theorem for perfectoid \BCM{} multiplier/test ideals and deformation statements for perfectoid \BCMReg{} and \BCMRat{} singularities. As an application, we obtain results on the behavior of $F$-regular and $F$-rational singularities in arithmetic families.
\end{abstract}

\setcounter{tocdepth}{1}
\tableofcontents


\section{Introduction}

For the past nearly 40 years, researchers have known about a connection between classes of singularities defined by resolution of singularities and connected to Kodaira-type vanishing theorems, with singularities defined by Frobenius in positive characteristic \cite{FedderFPureRat,MehtaRamanathanFrobeniusSplittingAndCohomologyVanishing,HochsterHunekeTC1,FedderWatanabe,SmithFRatImpliesRat,HaraRatImpliesFRat,MehtaSrinivasRatImpliesFRat,SmithMultiplierTestIdeals,HaraInterpretation,TakagiInterpretationOfMultiplierIdeals,HaraYoshidaGeneralizationOfTightClosure}.
For instance, a variety $X$ has rational singularities in characteristic zero if and only if its modulo $p$ reductions have $F$-rational singularities for $p \gg 0$.  Analogous statements hold for KLT singularities and $F$-regular singularities, and also with the multiplier ideals and test ideals.

In this paper, we begin to extend this story to mixed characteristic.  Using Andr\'e's recent breakthrough result on the existence of weakly functorial integral perfectoid big Cohen-Macaulay $R^+$-algebras  \cite{AndreWeaklyFunctorialBigCM}\footnote{The existence of weakly functorial big Cohen-Macaulay algebras was also obtained independently by Gabber \cite{GabberMSRINotes,GabberRameroFoundationsAlmostRingTheory}.} (also see \cite{AndrePerfectoidAbhyankarLemma,AndreDirectsummandconjecture,BhattDirectsummandandDerivedvariant,HeitmannMaBigCohenMacaulayAlgebraVanishingofTor}), we produce mixed characteristic analogs of
\begin{itemize}
\item{} log terminal / strongly $F$-regular singularities, which we call \emph{\BCMReg{B}}.
\item{} rational / $F$-rational singularities, which we call \emph{\BCMRat{B}}.
\item{} multiplier / test ideals, which we denote by $\mytau_B(R, \Delta)$.
\item{} Grauert-Riemenschneider / parameter test submodules, which we denote by $\mytau_B(\omega_R)$.
\end{itemize}
Here $B$ is any fixed (usually integral perfectoid) big Cohen-Macaulay $R^+$-algebra. In fact our first main theorem is that the choice of this $B$ is not so important.  From our perspective the role of $B$ is analogous to a resolution of singularities in characteristic zero or the perfection (up to a small perturbation) in characteristic $p > 0$.

\begin{theoremA*}[\autoref{prop.tau=tauB}, \autoref{prop.testidealsmallperturb}]
Let $(R,\m)$ be a complete local domain of mixed characteristic $(0,p)$. The objects $\mytau_B(\omega_R)$ and $\mytau_B(R, \Delta)$ (and hence the notions of \BCMReg{B} and \BCMRat{B} singularities) are independent of the choice of integral perfectoid big Cohen-Macaulay $R^+$-algebra $B$ as long as $B$ is chosen to be sufficiently large.
\end{theoremA*}

Note that similar definitions were produced by \cite{PerezRG} who also emphasized a more ideal/module closure theoretic approach.  On the other hand, the two authors of this paper previously defined in \cite{MaSchwedePerfectoidTestideal} a closely related multiplier/test ideal of pairs $(R, \fra^t)$ in the case that $R$ is a complete regular local ring of mixed characteristic.

We prove a number of results about these objects, which we now describe.
We first point out that \BCMReg{B} singularities are always \BCMRat{B} (and in particular they are Cohen-Macaulay), see \autoref{thm.BCMRegularImpliesCM}. Moreover, in characteristic $p>0$, \BCMRat{B}, \BCMReg{B} singularities and \BCM{} test ideals are \emph{exactly} the same as $F$-rational, $F$-regular and test ideals. See \autoref{prop: big CM rational=F-rational} and \autoref{cor.TauBCMVsTestIdeal}.

\begin{theoremB*}[Comparison with characteristic zero definitions]
Suppose that $(R, \fram)$ is a complete local domain of mixed characteristic $(0, p)$.
\begin{enumerate}
\item If $R$ is \BCMRat{B} for all (or even some sufficiently large) big Cohen-Macaulay algebras $B$, then $R$ is pseudo-rational.  See \autoref{prop: big CM rational implies pseudo-rational} and \autoref{prop.tau=tauB}.
    \label{itm.BCMRatImpliesRat}
\item For all sufficiently large integral perfectoid big Cohen-Macaulay $R^+$-algebras $B$, $\mytau_B(\omega_R) \subseteq \pi_* \omega_Y$ for all proper birational maps $\pi : Y \to X = \Spec R$.  See \autoref{prop: big CM test ideal birational}.
    \label{itm.BCMTauOmegaVsMultOmega}
\item If $R$ is normal and $\Delta \geq 0$ is an effective $\bQ$-divisor such that $K_R + \Delta$ is $\bQ$-Cartier, and if $(R, \Delta)$ is \BCMReg{B} for all (or even some sufficiently large) integral perfectoid big Cohen-Macaulay $R^+$-algebras $B$, then $(R, \Delta)$ is \emph{KLT}.  See \autoref{cor.BCMRegImpliesKLT} and \autoref{prop.testidealsmallperturb}.
    \label{itm.BCMRegImpliesKLT}
\item Suppose $R$ is normal and $\Delta \geq 0$ is an effective $\bQ$-divisor such that $K_R + \Delta$ is $\bQ$-Cartier.  Then for all sufficiently large integral perfectoid big Cohen-Macaulay $R^+$-algebras $B$, we have $\mytau_B(R, \Delta) \subseteq  \pi_* \O_Y(\lceil K_Y - \pi^* (K_X + \Delta)\rceil)$ for all proper birational maps $\pi : Y \to X = \Spec R$.  In other words:
    \[
        \mytau_B(R, \Delta) \subseteq \mJ(R, \Delta),
    \]
    the perfectoid \BCM{} test ideal is contained in the usual multiplier ideal (if resolutions of singularities exist or if one runs over all proper birational maps).  See \autoref{thm.TauBCMvsMult}.
    \label{itm.BCMTauDeltaVsMultDelta}
\end{enumerate}
\end{theoremB*}

Here \autoref{itm.BCMRatImpliesRat} and \autoref{itm.BCMTauOmegaVsMultOmega} should be viewed as the mixed characteristic analog of the main result of \cite{SmithFRatImpliesRat}.  Likewise \autoref{itm.BCMRegImpliesKLT} is analogous to the main result of \cite{HaraWatanabeFRegFPure}, whereas the characteristic $p > 0$ analog of \autoref{itm.BCMTauDeltaVsMultDelta} appeared essentially in \cite[Theorem 2.13]{TakagiInterpretationOfMultiplierIdeals}.  We should note that our strategy when proving these results is quite distinct from the characteristic $p > 0$ case; we are inspired by the ideas from \cite{MaThevanishingconjectureformapsofTorandderivedsplinters}.  A related mixed characteristic result was also proved in \cite[Lemma 5.6]{MaSchwedePerfectoidTestideal} but again using a different strategy.

We show that our perfectoid \BCM{} test ideals are stable under small perturbations as long as $B$ is large enough. This is analogous to multiplier ideals and test ideals.

\begin{theoremC*}[\autoref{prop.testidealsmallperturb}]
Suppose $R$ is a complete normal local domain of mixed characteristic $(0,p)$ and $\Delta \geq 0$ is an effective $\bQ$-divisor such that $K_R + \Delta$ is $\bQ$-Cartier. Then for all sufficiently large integral perfectoid big Cohen-Macaulay $R^+$-algebras $B$, we have
$$\mytau_B(R, \Delta)=\mytau_B(R, \Delta+\varepsilon\Div_R(g))$$
for any $0 \neq g\in R$ and any rational number $0 < \varepsilon\ll 1$.
\end{theoremC*}

We also obtain transformation rules for $\mytau_B(R, \Delta)$ under finite maps analogous to the main result of \cite{SchwedeTuckerTestIdealFiniteMaps}.

\begin{theoremD*}[Transformation under finite maps, \autoref{theorem: finite mapPairsCanonical}]
Suppose that $R \subseteq S$ is a finite extension of complete normal local domains of mixed characteristic $(0, p)$ with induced $\phi : \Spec S \to \Spec R$. Further assume that $\Delta \geq 0$ is a $\bQ$-divisor on $\Spec R$ such that $K_R + \Delta$ is $\bQ$-Cartier and such that $\phi^* \Delta \geq \Ram$, the ramification divisor.  Then:
\begin{equation}
\label{eq.TransRuleR}
\mytau_B(R, \Delta) = \Tr( \mytau_B(S, \phi^* \Delta - \Ram) ).
\end{equation}
\end{theoremD*}

One concern the reader may have at this point is that our $\mytau_B(\omega_R)$ and $\mytau_B(R, \Delta)$ are too small to be useful.  We prove the following result which should be viewed as an analog of \cite[Theorem 6.13]{HochsterHunekeTC1}.  

\begin{theoremE*}[\autoref{theorem: uniform annihilation}]
Let $(A,\m_A)\to (R,\m)$ be a module-finite extension such that $A$ is a complete regular local ring of mixed characteristic $(0,p)$ and $R$ is a complete local domain. Suppose $h\in A$ is such that $A_h\to R_h$ is finite \'{e}tale. Then there exists an integer $N$ such that $h^N \omega_R \subseteq \mytau_B(\omega_R)$ for every integral perfectoid big Cohen-Macaulay $R$-algebra $B$.
\end{theoremE*}

We also show that the ideal generated by all those $h$ such that $h^N\omega_R \subseteq \mytau_B(\omega_R)$ contains the defining ideal of the singular locus of $R/p$, see \autoref{thm.GenSingularlocus}. This gives us a weak analog of \cite[Theorem 3.4]{HochsterHunekeTightClosureAndStrongFRegularity}.

The previous result can be viewed as showing that certain elements are contained in $\mytau_B(R, \Delta)$.  Another way we can obtain similar results is to prove a restriction theorem. The following is one of the most useful results of this paper (in our opinion).  In particular, we are proving a form of inversion of adjunction in mixed characteristic.  Compare with for instance \cite[Proposition 2.12(1)]{TakagiInterpretationOfMultiplierIdeals}, \cite[Theorem 4.1]{HaraYoshidaGeneralizationOfTightClosure}, \cite{EsnaultViehwegLecturesOnVanishing}, \cite{KollarMori} and \cite[Theorem 9.5.1]{LazarsfeldPositivity2}.

\begin{theoremF*}[Restriction theorem, \autoref{thm.TestRestriction}, \autoref{cor.FregDeforms}]
Let $(R,\m)$ be a complete normal local domain of mixed characteristic $(0,p)$ and fix $\Delta \geq 0$ a $\bQ$-divisor on $R$ such that $K_R + \Delta$ is $\bQ$-Cartier with index not divisible by $p$.  Choose $0 \neq h \in R$ such that $R/hR$ is normal, and that $V(h)$ and $\Delta$ have no common components.  Then for any integral perfectoid big Cohen-Macaulay $R^+$-algebra $B$ and any $1 > \varepsilon > 0$, there exists a big Cohen-Macaulay $({R}/h{R})^+$-algebra $C$ (that can be chosen to be integral perfectoid if $R/hR$ has mixed characteristic) together with a compatible map $B \to C$ so that:
\[
\mytau_{C}(R/hR, \Delta|_{R/hR}) \subseteq \mytau_B(R, \Delta + (1-\varepsilon)\Div_R(h)) \cdot (R/hR).
\]

As a consequence, if $R/hR$ is of characteristic $p > 0$ and $(R/hR, \Delta|_{R/hR})$ is strongly $F$-regular, then $(R, \Delta+(1-\varepsilon)\Div_R(h))$ is \emph{KLT} and thus $(R,\Delta+\Div_R(h))$ is log canonical.
\end{theoremF*}
We also obtain a parameter test submodule version of the same result in \autoref{theorem: big CM restriction}; hence we obtain the analog of one of the main results of \cite{ElkikDeformationsOfRational} and of \cite{FedderWatanabe}.

Our key technical result which lurks behind many of the aformentioned theorems is that, in mixed characteristic, we prove any set of integral perfectoid big Cohen-Macaulay $R$-algebras can be mapped to another one. This is the mixed characteristic perfectoid analog of the main result of \cite{DietzBCMSeeds} and answers a question of Shimomoto \cite[Problem 1]{ShimomotoIntegralperfectoidbigCMviaAndre}. Our method of proving this result is largely inspired by the ideas from \cite{AndreWeaklyFunctorialBigCM}, and this is the only place in this paper that we use the machinery of perfectoid algebras and spaces.

\begin{theoremG*}[\autoref{thm.DominateAnySetOfBCM}]
Let $(R, \fram)$ be a complete local domain of mixed characteristic $(0, p)$. Let $\{ B_{\gamma} \}_{\gamma \in \Gamma}$ be any set of integral perfectoid big Cohen-Macaulay $R$-algebras (resp. $R^+$-algebras). Then there is an integral perfectoid big Cohen-Macaulay $R^+$-algebra $B$ such that the map $R\to B$ (resp. $R^+\to B$) factors through all the maps $R\to B_{\gamma}$ (resp. $R^+\to B_\gamma$).

\end{theoremG*}

As an application of our theory and the main result of \cite{HaraRatImpliesFRat,MehtaSrinivasRatImpliesFRat,TakagiInterpretationOfMultiplierIdeals} we obtain the following result, again compare with \cite{ElkikDeformationsOfRational} in characteristic zero and also \cite{HashimotoCMFinjectiveHoms,PatakfalviSchwedeZhangFsingularitiesFamilies} in characteristic $p > 0$.

\begin{theoremH*}[\autoref{thm.ProjectiveFamilyFratModPImpliesChar0}, \autoref{thm.ProjectiveFamilyOneFRegImpliesAlmostall}]
Let $X \to U \subseteq \Spec \bZ$ be a proper flat family (resp. let $(X,\Delta\geq 0) \to U\subseteq \Spec \bZ$ be a proper flat family of pairs such that $K_X+\Delta$ is $\bQ$-Cartier of index $N$). Suppose the fiber $X_p$ is $F$-rational for some point $p \in U$ (resp. $(X_p,\Delta_p)$ is strongly $F$-regular for some $p$ whose residual characteristic does not divide $N$). Then the general fiber $X_{\bQ}$ has rational singularities (resp. the general fiber $(X_{\bQ},\Delta_{\bQ})$ is \emph{KLT}). Furthermore $X_q$ is $F$-rational (resp. $(X_q, \Delta_q)$ is strongly $F$-regular) for a Zariski dense and open set of closed points $q \in U$.
\end{theoremH*}

It is not hard to see that the statement is false for arbitrary affine schemes (\autoref{ex.BadAffineExample}).  However, a properly formulated \emph{local} version of it is true for both rational singularities (\autoref{thm.LocalOpenFrationalOverZ}) and for KLT singularities (\autoref{thm.LocalOpenFregularOverZ}).  The point is we must consider a family of \emph{local} rings over $\Spec \bZ$ (or more general bases).  Finally, we also point out that we can replace $\bZ$ with more general mixed characteristic Dedekind domains.

In \autoref{sec.AlgorithmicConsequences}, as an application of these ideas we obtain a new effective way of proving that a singularity in characteristic zero is rational or KLT.  In particular if one starts in characteristic $0$ with a singularity $(R, \fram)$ and if one spreads it out to a mixed characteristic domain $R_A$ which restricts to a characteristic $p > 0$ local ring $R_p$ with $F$-rational singularities, then $R$ has rational singularities.  The point is that there is no requirement that $p \gg 0$.  Note that in \cite{SmithFRatImpliesRat}, Smith observed that one needs to check only a single $R_p$, but $p > 0$ had to be large enough so that $\omega_A / \pi_* \omega_{Y_A}$ is $p$-torsion free (which essentially means one must already compute whether the singularity is rational or not).  We can avoid this restriction.  This idea also generalizes to KLT and strongly $F$-regular pairs (as long as $p$ does not divide the index of $K_R + \Delta$).  This is important because $F$-rationality and $F$-regularity checking are implemented in Macaulay2 \cite{M2} in \cite{TestIdealsPackage} for \emph{more general} cases than rational or log terminal checking is done in characteristic zero \cite{DmodulesSource}.  Finally, also compare with \cite{ZhuLCThresholdsInPositiveChar} where related results were proved for the log canonical threshold via jet schemes.

\vskip 9pt

We describe the organization of the paper.  In \autoref{sec.Prelim} we briefly discuss preliminaries and notations.  In \autoref{sec.BCMRational} we prove many of the main results of this paper for \BCMRat{B} singularities, the proofs of which motivate and are simple cases of many of the other main results in the paper.  In \autoref{sec.AndreWeaklyFunctorial} we review Andr\'e and Gabber's recent result on weakly functorial big Cohen-Macualay algebras and we  prove Theorem G above.  In \autoref{sec.BCMParamTest} we introduce $\mytau_B(\omega_R)$ and prove many of the results above in that setting and we prove Theorem E.  In \autoref{sec.BCMTestIdeals} we introduce $\mytau_B(R, \Delta)$ and when a pair $(R, \Delta)$ is \BCMReg{B}, and we prove Theorem C, Theorem D, and Theorem F. In \autoref{sec.ApplicationToArithmeticFamilies} we use these ideas to study $F$-singularities in families where the characteristic varies and we prove Theorem H. Finally in \autoref{sec.AlgorithmicConsequences} we discuss algorithmic consequences mentioned above.

\vskip 9pt

\emph{Acknowledgements:}  The authors thank Yves Andr\'e, Bhargav Bhatt, Paolo Cascini, Ofer Gabber, Srikanth Iyengar, Tiankai Liu, Zsolt Patakfalvi, Rebecca R.G., Anurag K. Singh, Kazuma Shimomoto, Shunsuke Takagi, Joe Waldron, and Jakub Witaszek for valuable and inspiring discussions.  Important progress was made when the authors were visiting MSRI in March 2018.  Additionally, the authors would like to thank Paolo Cascini, Zsolt Patakfalvi, and MSRI for hospitality when the authors were working on this paper.  We are particularly grateful to Ofer Gabber for pointing out gaps in Section 4 of this paper and also to Bhargav Bhatt for extended discussions about some of the arguments in this paper.  The first author also thanks Bernd Ulrich for valuable discussion on Flenner's local Bertini theorems. We thank all of the referees for numerous helpful comments.  Finally, we thank one of the referees and also Ray Heitmann for pointing out a mistake in our original version of what is now \autoref{lemma: Cohen-Gabber}.

\section{Preliminaries}
\label{sec.Prelim}

Throughout this paper, all rings will be commutative with unity.  {\it Local rings} $(R, \fram)$ are always assumed to be Noetherian, although we frequently will consider non-Noetherian rings (including those with unique maximal ideals).



\subsection{Perfectoid algebras and big Cohen-Macaulay algebras} Throughout this paper we will use the language of integral perfectoid algebras and almost mathematics as in \cite{AndreWeaklyFunctorialBigCM,BhattDirectsummandandDerivedvariant,ScholzePerfectoidspaces,GabberRameroAlmostringtheory}. We will work over a {\it fixed} perfectoid field $K=\widehat{\mathbb{Q}_p(p^{1/p^\infty})}$ and its ring of integers $K^\circ=\widehat{\mathbb{Z}_p[p^{1/p^\infty}]}$. We collect some definitions.

A {\it perfectoid $K$-algebra} is a Banach $K$-algebra $R$ such that the set of powerbounded elements $R^\circ\subseteq R$ is bounded and the Frobenius is surjective on $R^\circ/p$. A $K^\circ$-algebra $S$ is called {\it integral perfectoid} if it is $p$-adically complete, $p$-torsion free, and the Frobenius induces an isomorphism $S/p^{1/p}\to S/p$. If $R$ is a perfectoid $K$-algebra, then the ring of powerbounded elements $R^\circ$ is integral perfectoid, and if $S$ is integral perfectoid, then $S[1/p]$ is perfectoid (\cite[Theorem 5.2]{ScholzePerfectoidspaces}).

\begin{remark}
In \cite{BhattDirectsummandandDerivedvariant}, there is an extra condition in the definition of integral perfectoid algebra: one requires that $S=S_*=\{x\in S[1/p] \hspace{0.5em}  | \hspace{0.5em}  p^{1/p^n} x\in S \text{ for all } n \}$. However, in practice one can safely ignore the difference between $S$ and $S_*$ (or simply pass from $S$ to $S_*$) because they are $p^{1/p^\infty}$-almost isomorphic to each other. Our definitions of perfectoid and integral perfectoid algebras are the same as in \cite[Section 2.2]{AndreWeaklyFunctorialBigCM}.
\end{remark}

Let $(R,\m)$ be a local ring and let $\underline{x}=x_1,\dots,x_d$ be a system of parameters of $R$. Recall that an $R$-algebra $B$ is {\it big Cohen-Macaulay with respect to $\underline{x}$} if $\underline{x}$ is a regular sequence on $B$ (this means $(x_1,\dots,x_i):_Bx_{i+1}=(x_1,\dots,x_i)B$ and $B/(x_1,\dots,x_d)B\neq 0$), and $B$ is called a {\it (balanced) big Cohen-Macaulay algebra} if it is big Cohen-Macaulay with respect to $\underline{x}$ for every system of parameters $\underline{x}$. It is well known that if $B$ is big Cohen-Macaulay with respect to $\underline{x}$, then $\widehat{B}^\m$ is (balanced) big Cohen-Macaulay \cite[Corollary 8.5.3]{BrunsHerzog}. Big Cohen-Macaulay algebras always exist: in equal characteristic, this follows from \cite{HochsterHunekeInfiniteIntegralExtensionsAndBigCM} \cite{HochsterHunekeApplicationsofBigCM}, and in mixed characteristic, this is settled by Andr\'{e} in \cite{AndreDirectsummandconjecture} (see also \cite{HeitmannMaBigCohenMacaulayAlgebraVanishingofTor}).

Let $(R,\m)$ be a local ring of mixed characteristic $(0,p)$ or equal characteristic $p>0$, and let $\underline{x}=x_1,\dots, x_d$ be a system of parameters of $R$. Let $B$ be an $R$-algebra and $g$ be an element of $R$ such that $g$ has a compatible system of $p$-power roots $\{g^{1/p^e}\}_{e=1}^{\infty}$ in $B$. For example, this holds if $R$ is a domain and $B$ is an $R^+$-algebra, where $R^+$ denotes the absolute integral closure of $R$: this is the integral closure of $R$ inside an algebraic closure of its fraction field. We say $B$ is {\it $g^{1/p^\infty}$-almost big Cohen-Macaulay with respect to $\underline{x}$} if $\underline{x}$ is a $g^{1/p^\infty}$-almost regular sequence on $B$, that is, $\frac{(x_1,\dots,x_i):_Bx_{i+1}}{(x_1,\dots,x_i)B}$ is $g^{1/p^\infty}$-almost zero and $B/(x_1,\dots,x_d)B$ is {\it not} $g^{1/p^\infty}$-almost zero. This terminology is slightly misleading since ``big Cohen-Macaulay" does not formally imply ``almost big Cohen-Macaulay" because the last condition that $B/(x_1,\dots,x_d)B$ is not $g^{1/p^\infty}$-almost zero is stronger than $B/(x_1,\dots,x_d)B\neq 0$. However in most cases this is not an issue: see for example \cite[Proposition 2.5.1]{AndreWeaklyFunctorialBigCM} and \autoref{clm.AlmostCM}.

The importance of introducing this almost big Cohen-Macaulay notion is that in characteristic $p>0$, it is not hard to show that under mild assumptions, $R^{1/p^\infty}$ is $g^{1/p^\infty}$-almost big Cohen-Macaulay for suitable choice of $g$ (e.g., $R_g$ is Cohen-Macaulay) with respect to any $\underline{x}$. Hochster essentially proved that every almost big Cohen-Macaulay algebra can be mapped to a big Cohen-Macaulay algebra (see \cite{HochsterTopicsInTheHomologicalTheory} \cite{HochsterSolidClosure} \cite{HochsterBigCohen-Macaulayalgebrasindimensionthree} \cite{DietzBCMSeeds}).

Throughout this paper, we will often work with (integral perfectoid) big Cohen-Macaulay $R^+$-algebras for a complete local domain $(R,\m)$ of mixed characteristic $(0,p)$.  The existence of such algebras follows from recent work of Andr\'{e} \cite{AndreWeaklyFunctorialBigCM} and Shimomoto \cite{ShimomotoIntegralperfectoidbigCMviaAndre}. For the convenience of the reader we re-define integral perfectoid in this context (note that, since any $p$-adically complete $R^+$-algebra $B$ is automatically an algebra over $K^\circ$, the definition of integral perfectoid algebra below is compatible with our general definition above).

\begin{definition}
Let $(R,\m)$ be a complete local domain of mixed characteristic $(0,p)$. Fix an algebraic closure of its fraction field and so fix an $R^+$. Let $g$ be a nonzero element in $R^+$. Then an \emph{integral perfectoid $R^+$-algebra} (resp., a $(pg)^{1/p^\infty}$-almost integral perfectoid $R^+$-algebra) is an $R^+$-algebra $S$ such that
\begin{enumerate}
  \item $S$ is $p$-adically complete and $p$-torsion free;
  \item The Frobenius map $S/p^{1/p}\to S/p$ is an isomorphism (resp., the Frobenius map $S/p^{1/p}\to S/p$ is an injection and a $(pg)^{1/p^\infty}$-almost surjection).
\end{enumerate}
\end{definition}

Given a $(pg)^{1/p^\infty}$-almost integral perfectoid algebra $B$, we can tilt and then untilt to obtain an honest integral perfectoid algebra $B^\natural=(B^\flat)^\sharp$. Moreover, there always exists a natural map $B^\natural\to B$ that is injective and is $(pg)^{1/p^\infty}$-almost surjective, see \cite[2.3.1]{AndreWeaklyFunctorialBigCM} or \cite[Section 2]{AndrePerfectoidAbhyankarLemma}.

\subsection{Singularities coming from characteristic 0 and singularities in characteristic $p$} 
In characteristic zero, for higher dimensional varieties, we frequently measure singularities by considering a resolution of singularities $\pi : Y \to X$ and comparing top differential forms on $Y$ with those on $X$.  We begin with rational and log terminal singularities, noting some alternate definitions which hopefully are suggestive for some of our definitions later. We will not need all of these definitions but we include them for motivation. We refer the reader to \cite{HartshorneResidues} for standard notations and facts on local and Grothendieck duality. For a local ring $(R,\m)$, we use $E$ to denote the injective hull of $R/\fram$.

\begin{notation}
  Given a Weil divisor $D$ on an affine scheme $X = \Spec R$, we write
  \[
    R(D) = \Gamma(X, \cO_X(D)).
  \]
  In other words, it is the fractional ideal corresponding to $\cO_X(D)$.
\end{notation}

\begin{defprop}[Singularities in characteristic zero]
\label{defProp.Char0}
Suppose that $(R, \fram)$ is a $d$-dimensional local ring essentially of finite type over a field of characteristic zero.  Suppose that $\pi : Y \to X = \Spec R$ is a log resolution of singularities (all the definitions below are independent of the choice of resolution).
\begin{description}
\item[Rational singularities] We say that $R$ has \emph{rational singularities} if $R \qis \myR \Gamma(Y, \O_Y)$.  This is equivalent to
\begin{itemize}
\item[(a)]  $R$ is Cohen-Macaulay and one of the following equivalent conditions holds:
\item[(b)]  $\Gamma(Y, \omega_Y) \to \omega_R$ is surjective.
\item[(b$'$)]  $H^d_{\fram}(R) \to \bH^d_{\fram}(\myR \Gamma(Y, \O_Y))$ is injective. 
\end{itemize}
\item[Grauert-Riemenschneider (multiplier) submodules]  We next define $\mJ(\omega_R)$, called the \emph{Grauert Riemenschneider (multiplier) submodule of $R$}, to be $\Gamma(Y, \omega_Y) \subseteq \omega_R$. If further $\Lambda \geq 0$ is a $\bQ$-Cartier divisor on $X$ (see \autoref{rem.QDivisors}), then we define $\mJ(\omega_R, \Lambda)$ to be $\Gamma(Y, \omega_Y(-\lfloor \pi^* \Lambda\rfloor)) \subseteq \omega_R$. 
\item[Kawamata-log terminal singularities]  Suppose further that $R$ is normal and there exists a $\bQ$-divisor $\Delta \geq 0$ such that $K_R + \Delta$ is $\bQ$-Cartier.  We now require $Y$ to be a log resolution for $(X, \Delta)$.  Then we say that the pair $(R, \Delta)$ is \emph{Kawamata-log terminal} (or KLT) if the following equivalent conditions hold.
    \begin{itemize}
\item[(c)]  $\lceil K_Y - \pi^* (K_X + \Delta) \rceil \geq 0$. 
\item[(c$'$)]  $H^d_{\fram}(\omega_R) \to \bH^d_{\fram}(\myR\Gamma(Y,\O_Y(\lfloor \pi^*(K_X + \Delta)\rfloor)))$ is injective. 
\end{itemize}
\item[Multiplier ideals]  With the same assumptions as we had for KLT singularities, we define the \emph{multiplier ideal} $\mJ(X, \Delta)$ to be:
\begin{itemize}
\item[(d)] $\Gamma(Y, \O_Y(\lceil K_Y - \pi^* (K_X + \Delta)\rceil)) \subseteq R$.
\item[(d$'$)]  $\Ann_R \ker\Big( H^d_{\fram}(\omega_R) \to \bH^d_{\fram}\big(\myR\Gamma(Y,\O_Y(\lfloor \pi^*(K_X + \Delta)\rfloor))\big)\Big)$. 
\end{itemize}
Note that $(X, \Delta)$ is KLT if and only if $\mJ(X, \Delta) = R$.
\end{description}
\end{defprop}

\begin{remark}
  \label{rem.QDivisors}
  Note that when deal with $\bQ$-divisors that are not ($\bQ$-)Cartier, we \emph{always} assume that the ambient scheme is normal.   We also only round (coefficients of) divisors on normal schemes.   A $\bQ$-Cartier divisor on any affine scheme $\Spec R$ is nothing more than the data of a rational number (the coefficient), and a ideal that is locally free.
\end{remark}

\begin{proof}
Using local and Grothendieck duality, we prove the equivalence of (c) and (c$'$) below.  The equivalence of (d) and (d$'$) is essentially the same and the equivalence of (b) and (b$'$) also follows from local and Grothendieck duality, but it is less involved.  Therefore we omit those proofs.  For (c) and (c$'$), notice that the trace map
\begin{equation}
\label{eqn.TraceMapOfLogCanonical}
\myR \Gamma(Y, \O_Y(\lceil K_Y - \pi^* (K_X + \Delta)\rceil)) \to R
\end{equation}
is surjective on 0th-cohomology if and only if the Matlis dual
\[
\Hom_R\big(\myR \Gamma(Y, \O_Y(\lceil K_Y - \pi^* (K_X + \Delta)\rceil)), E\big) \leftarrow E = H^d_{\fram}(\omega_R).
\]
is injective.  By local and Grothendieck duality and using the fact that $Y$ is regular so that $\omega_Y = \omega_Y^{\mydot}[-d]$, we see that
\[
\begin{array}{rl}
& \Hom_R\big(\myR \Gamma(Y, \O_Y(\lceil K_Y - \pi^* (K_X + \Delta)\rceil)), E\big) \\
\qis & \Hom_R\big(\myR \Gamma(Y, \myR \sHom_{\O_Y}(\O_Y(\lfloor \pi^* (K_X + \Delta)\rfloor), \omega_Y)), E\big) \\
\qis & \Hom_R\big(\myR \Hom_{R}(\myR \Gamma(Y, \O_Y(\lfloor \pi^* (K_X + \Delta)\rfloor)), \omega_X^{\mydot}[-d])), E\big) \\
\qis & \myR\Gamma_{\fram}\big( \myR \Gamma(Y, \O_Y(\lfloor \pi^*(K_X + \Delta) \rfloor)) \big)[d].
\end{array}
\]
Hence $\Gamma(Y, \O_Y(\lceil K_Y - \pi^*(K_X + \Delta)\rceil))$ is Matlis dual to $\bH^d_{\fram}(\myR \Gamma(Y, \O_Y(\lfloor \pi^*(K_X + \Delta) \rfloor)))$.  The statement follows.  Notice that the Grothendieck dual to \autoref{eqn.TraceMapOfLogCanonical} is
\[
\omega_R[d] \to \omega_R^{\mydot} \to \myR \Gamma(Y, \O_Y(\lfloor \pi^* (K_X + \Delta) \rfloor))[d]
\]
which induces the inclusion in (c$'$).
\end{proof}

\begin{definition}[Pseudo-rational and KLT singularities in all characteristics]
If $(R, \fram)$ is a local ring with a dualizing complex then we say that $R$ has \emph{pseudo-rational} singularities if it satisfies the conditions (a) and (b) (or (b$'$)) in \autoref{defProp.Char0} for \emph{all} proper birational maps $\pi : Y \to X = \Spec R$ with $Y$ normal.

Additionally, if $R$ is normal and $K_R + \Delta$ is $\bQ$-Cartier for an effective $\bQ$-divisor $\Delta\geq 0$, then we say that $(R, \Delta)$ is KLT if it satisfies the equivalent conditions (c) or (c$'$) above in \autoref{defProp.Char0} for \emph{all} proper birational maps $\pi : Y \to X = \Spec R$ with $Y$ normal.

One can likewise define $\mJ(R, \Delta)$ to be the intersection of $\Gamma(Y, \O_Y(\lceil K_Y - \pi^* (K_X + \Delta)\rceil)) \subseteq R$ where $Y$ runs over all proper birational maps and define $\mJ(\omega_R)$ to be the intersection of $\Gamma(Y, \omega_Y) \subseteq \omega_R$ where $Y$ again runs over all proper birational maps.  Note that by Chow's Lemma, in any of these definitions, one can restrict attention to projective birational maps.
\end{definition}

We next move to positive characteristic and we record the definitions of $F$-rational and $F$-regular singularities. Recall that in characteristic $p>0$, the Frobenius map $R\to F^e_*R$ induces a natural Frobenius action $F^e$ on the local cohomology module $H_\m^d(R)$ (which can be identified with $H_\m^d(R)\to H_\m^d(F_*^eR)$).

\begin{defprop}[Singularities in characteristic $p > 0$]
\label{defprop.CharpSings}
Suppose $(R, \fram)$ is an excellent $d$-dimensional local domain of characteristic $p > 0$ with normalized dualizing complex $\omega_R^{\mydot}$ and canonical module $\omega_R = h^{-d}(\omega_R^{\mydot})$.
\begin{description}
\item[$F$-rational singularities] We say that $R$ has $F$-rational singularities if
\begin{itemize}
\item[(a)]  $R$ is Cohen-Macaulay and one of the following equivalent conditions holds:
\item[(b)]  If $N \subseteq H^d_{\fram}(R)$ is such that $F(N) \subseteq N$, then $N = 0$ or $N = H^d_{\fram}(R)$.
\item[(b$'$)]  For any $0 \neq c \in R$, there exists an integer $e > 0$ so that $H^d_{\fram}(R) \xrightarrow{1 \mapsto F^e_* c} H^d_{\fram}(F^e_* R)$ is injective.
\item[(b$''$)]  If $R^+$ is the absolute integral closure of $R$ and $B$ is any big Cohen-Macaulay $R^+$-algebra, then $H^d_{\fram}(R) \to H^d_{\fram}(B)$ is injective.
\end{itemize}
\item[Parameter test submodules]  The \emph{parameter test submodule} $\tau(\omega_R)$ is the $\omega_R$-annihilator of the kernel of $H^d_{\fram}(R) \to H^d_{\fram}(R^+)$. Further, if $\Lambda \geq 0$ is a $\bQ$-Cartier divisor with $n \Lambda = \Div(f)$, then we define $\tau(\omega_R, \Lambda)$ to be the $\omega_R$-annihilator of the kernel of $H^d_{\fram}(R) \xrightarrow{1 \mapsto f^{1/n}} H^d_{\fram}(R^+)$.
\item[Strongly $F$-regular singularities]
Further assume that $R$ is normal and $\Delta \geq 0$ is an effective $\bQ$-divisor on $\Spec R$.  
Fix a choice of $K_R \geq 0$ and suppose $K_R + \Delta$ is $\bQ$-Cartier with $n(K_R + \Delta) = \Div_R(f)$ for some $f \in R$ and $n > 0$. Let $B$ be a big Cohen-Macaulay $R^+$-algebra (and hence contains a copy of $f^{1/n}$ coming from $R^+$). Notice that the map $H^d_{\fram}(R) \xrightarrow{f^{1/n}} H^d_{\fram}(B)$ factors through $H^d_{\fram}(R(K_R)) \cong E$.
\[
\xymatrix{
H^d_{\fram}(R) \ar[dr] \ar[rr]^{\cdot f^{1/n}} & & H^d_{\fram}(B)\\
& H^d_{\fram}(R(K_R)) \ar[ur]_{\psi}
}
\]
In this case, we say that $(R, \Delta)$ has \emph{strongly $F$-regular singularities} if the induced map \[
\psi : H^d_{\fram}(R(K_R)) \to H^d_{\fram}(B)
\]
is injective (we will prove that this is independent of the choice of $B$).
\item[Test ideals]  With the same assumptions as we had for strongly $F$-regular singularities, we define the \emph{test ideal} $\tau(R,\Delta)$ to be
    \[
    \Ann_R\big( \ker (H^d_{\fram}(R(K_R)) \to H^d_{\fram}(B)) \big).
    \]
    This is also independent of the choice of $B$.
\end{description}
\end{defprop}
\begin{proof}
The equivalence of (b) and (b$'$) is \cite[Theorem 2.6]{SmithFRatImpliesRat}.  The equivalence of (b$''$) essentially follows from \cite[Section 5]{SmithTightClosureParameter}.  See the proof of \autoref{prop: big CM rational=F-rational}.


We first establish the factorization in the diagram before the definition of strongly $F$-regularity.  Notice that for any finite extension $R \subseteq S \subseteq R^+$ with $S$ normal and containing $f^{1/n}$ and with induced map $\rho : \Spec S \to \Spec R$, we have that $S \xrightarrow{\cdot f^{1/n}} S$ factors as $S \hookrightarrow S(\rho^* K_R) \xrightarrow{\nu} S$.  This is because $(1/n) \Div(f) \geq \rho^* K_R$ by construction.  But $R \hookrightarrow S(\rho^* K_R)$ factors through $R \hookrightarrow R(K_R)$.  Applying $H^d_{\fram}(-)$ then gives us our desired factorization.

It remains to prove that strong $F$-regularity of pairs and the test ideal are independent of the choice of $B$.  With $S$ as above, we have the factorization
\[
\xymatrix{
H^d_{\fram}(R(K_R))\ar[r]& H^d_{\fram}(S(\rho^* K_R)) \ar[r]^-{\nu} & H^d_{\fram}(S) \ar[r] & H^d_{\fram}(B).\\
}
\]
Note the map $H^d_{\fram}(S) \to H^d_{\fram}(B)$ is induced by $S \to B$.
Now, $B$ is a big Cohen-Macaulay $R^+$-algebra, hence the kernel of $H^d_{\fram}(S) \to H^d_{\fram}(B)$ is independent of $B$ by \cite[Section 5]{SmithTightClosureParameter} (see also \autoref{prop: big CM test ideal char p}). Thus so is the kernel of the composition.  This proves the independence of the choice of $B$ for both strong $F$-regularity and for the test ideal.
\end{proof}

Our definition of strong $F$-regularity and test ideal for pairs $(R, \Delta)$ is \emph{not} the usual definition (we made our choice for motivational purposes as an analogous definition will be made in the mixed characteristic setting). However, it is equivalent to the usual definitions.  Note that $F$-singularities of pairs in the literature have typically (perhaps nearly always) been defined in the $F$-finite setting, see for instance \cite{HaraWatanabeFRegFPure,HaraTakagiOnAGeneralizationOfTestIdeals,TakagiInterpretationOfMultiplierIdeals}.  In that case, $(R, \Delta \geq 0)$ is called \emph{strongly $F$-regular} if for every $0 \neq c \in R$, there exists an $e > 0$ so that the canonical map $R \to F^e_* R(\lceil (p^e - 1) \Delta + \Div_R(c) \rceil)$ splits (this definition holds without the condition that $K_R + \Delta$ is $\bQ$-Cartier).

\begin{proposition}
\label{prop.DefFregularSameClassical}
Our definition of strongly $F$-regular pairs and the test ideal from \autoref{defprop.CharpSings} agrees with the classical definition of \cite{HaraWatanabeFRegFPure,BlickleSchwedeTuckerTestAlterations} if $R$ is $F$-finite and $K_R + \Delta$ is $\bQ$-Cartier.
\end{proposition}
\begin{proof}[Proof sketch]
We prove the statement for the test ideal since strong $F$-regularity is just the condition that $\tau(R, \Delta) = R$. From now on till the end of the proof $\tau(R,\Delta)$ refers to test ideal as in \cite[Definition 2.37]{BlickleSchwedeTuckerTestAlterations}.  Again write that $n(K_R + \Delta) = \Div_R(f)$.  By \cite[Theorem 4.6]{BlickleSchwedeTuckerTestAlterations}, we know that $\tau(R, \Delta) = \Image(\Tr : S(\lceil K_S - \rho^* (K_R + \Delta)\rceil) \to R)$ for any sufficiently larger finite extension $R \subseteq S$ with induced $\rho : \Spec S \to \Spec R$ (here $\Tr$ is the Grothendieck trace).  We may assume that $f^{1/n} \in S$ so that $\Div_S(f^{1/n}) = \rho^*(K_R + \Delta)$.  Thus $\tau(R, \Delta) = \Image(\Tr : f^{1/n} \cdot S(K_S) \to R)$.  We observe that $\Hom_R(S(K_S), E) = \Hom_R(\Hom_R(S, \omega_R), E) = H^d_{\fram}(S)$
Applying Matlis duality, we then have
\[
\tau(R, \Delta) = \Ann_R(\ker(E \to H^d_{\fram}(S)))
\]
and it is easy to see that the map $E \to H^d_{\fram}(S)$ is a factor of $H^d_{\fram}(R) \xrightarrow{f^{1/n}} H^d_{\fram}(S)$ as in \autoref{defprop.CharpSings}.  Because we chose $S$ large enough, the kernel is independent of the choice of $S$ and hence taking a limit we see that
\[
\tau(R, \Delta) = \Ann_R(\ker(E \to H^d_{\fram}(R^+)))
\]
where the map $E\to H_\m^d(R^+)$ is the map in the diagram of \autoref{defprop.CharpSings} (one can take $B=R^+$ since $R^+$ itself is big Cohen-Macaulay by \cite{HochsterHunekeInfiniteIntegralExtensionsAndBigCM}).
\end{proof}

Another advantage of our definition of test ideals is that it makes transformation rules under finite maps quite transparent (see \autoref{sec.TransformationRulesUnderFiniteMaps}).

\begin{remark}
It is natural to expect that, even without the $F$-finite hypothesis, we have that our definition of $\tau(R, \Delta)$ is equal to the annihilator of $0^{* \Delta}_{E}$, which can be defined to be
\[
\big\{ \eta \in E \;|\; \text{ there exists nonzero } c \in R \text{ such that } \forall e > 0,  0 = \eta \otimes c^{1/p^e} \in E \otimes_R (R(\lfloor p^e \Delta \rfloor))^{1/p^e} \big\}.
\]
We believe this is straightforward but we do not work it out here since the literature on $F$-singularities of pairs includes the $F$-finite hypothesis.
\end{remark}

\section{Big Cohen-Macaulay rational singularities}
\label{sec.BCMRational}
In this section we prove our result on rational and $F$-rational singularities. Most results of this section, at least when $R$ is a complete normal local domain, will also follow from our more general results in \autoref{sec.BCMParamTest}. However, our treatment here is less technical and it only requires the version of weakly functorial big Cohen-Macaulay algebras established by Heitmann and the first author in \cite{HeitmannMaBigCohenMacaulayAlgebraVanishingofTor}\footnote{The version of weakly functorial big Cohen-Macaulay algebras established in \cite{HeitmannMaBigCohenMacaulayAlgebraVanishingofTor} is weaker, but the method is substantially shorter than \cite{AndreWeaklyFunctorialBigCM}, which requires \cite{AndrePerfectoidAbhyankarLemma}.}.  On the other hand, those later arguments in many cases are simply jazzed up versions of what is done in this section.

\begin{definition}
\label{def.BCMrat}
Let $(R,\m)$ be an excellent local ring of dimension $d$ and let $B$ be a big Cohen-Macaulay $R$-algebra. We say $R$ is {\it big Cohen-Macaulay-rational with respect to $B$} (or simply \emph{\BCMRat{B}}) if
\begin{itemize}
\item{} $R$ is Cohen-Macaulay and
\item{} $H_\m^d(R)\to H_\m^d(B)$ is injective.
\end{itemize}
We say $R$ is \emph{\BCMRat{}} if $R$ is \emph{\BCMRat{B}} for all big Cohen-Macaulay $R$-algebras $B$.
\end{definition}

\begin{remark}
\label{rem.BCMrationalcompletion}
$R$ is \emph{\BCMRat{B}} if and only if $\widehat{R}$ is \emph{\BCMRat{\widehat{B}}}. This is because $\widehat{B}$ is a big Cohen-Macaulay $\widehat{R}$-algebra \cite[Corollary 8.5.3]{BrunsHerzog} and $H^d_{\fram}(M) \cong H^d_{\fram}(\widehat{M})$ for any $R$-module $M$.  As a consequence, $R$ is \emph{\BCMRat{}} if and only if $\widehat{R}$ is \emph{\BCMRat{}} since every big Cohen-Macaulay $\widehat{R}$-algebra is certainly a big Cohen-Macaulay $R$-algebra.
\end{remark}

\begin{lemma}
\label{lem.BCMratNormal}
If $R$ is \emph{\BCMRat{}}, then $\widehat{R}$ is a normal domain.
\end{lemma}
\begin{proof}
By \autoref{rem.BCMrationalcompletion}, we may assume $R$ is complete. Let $P$ be a minimal prime of $R$ such that $\dim R/P=\dim R$ and let $B$ be a big Cohen-Macaulay $R/P$-algebra (and hence a big Cohen-Macaulay $R$-algebra). Since $R$ is \BCMRat{}, the composition $H_\m^d(R)\to H_\m^d(R/P)\to H_m^d(B)$ is injective, thus $H_\m^d(R)\to H_\m^d(R/P)$ is injective. For every system of parameters $x_1,\dots,x_d$ of $R$, consider the following diagram:
\[\xymatrix{
H_\m^d(R) \ar@{^{(}->}[r] &  H_\m^d(R/P) \\
R/(x_1,\dots,x_d) \ar@{^{(}->}[u] \ar[r] & (R/P)/(x_1,\dots,x_d)(R/P) \ar[u]
}
\]
where the injectivity of the left vertical map is because $R$ is Cohen-Macaulay (by the definition of {\BCMRat{}}) and thus $x_1,\dots,x_d$ is a regular sequence on $R$. Chasing the diagram, we know that $R/(x_1,\dots,x_d)\to (R/P)/(x_1,\dots,x_d)(R/P)$ is injective. Thus $P\subseteq (x_1,\dots,x_d)$ for every system of parameters $x_1,\dots,x_d$ of $R$. Hence $P\subseteq \cap_t(x_1^t,\dots,x_d^t)=0$. This proves that $R$ is a domain.

Now let $R^{\textnormal{N}}$ denote the normalization of $R$, which is also a complete normal domain.  A big Cohen-Macaulay $R^{\textnormal{N}}$-algebra $C$ is still a big Cohen-Macaulay $R$-algebra, and hence it follows that $H^d_{\fram}(R) \to H^d_{\fram}(R^{\textnormal{N}})$ is injective.  Suppose $\frac{y}{x_1}\in R^{\textnormal{N}}$ and $x_1,\dots,x_d$ is a system of parameters of $R$. The injectivity of $H^d_{\fram}(R) \to H^d_{\fram}(R^{\textnormal{N}})$ yields $y\in \cap_t(x_1,x_2^t,\dots,x_d^t)=(x_1)$, thus $\frac{y}{x_1}\in R$. Thus $R$ is normal. 
\end{proof}

Our first observation is that, with the above definition, {\BCMRat{}} singularities deform in the following sense.
\begin{proposition}
\label{prop: big CM rational deform}
Let $(R,\m)$ be an excellent local ring of dimension $d$ and $x\in R$ be a nonzerodivisor. Let $B$ be a big Cohen-Macaulay $R$-algebra. If $R/xR$ is {\BCMRat{B/xB}}, then $R$ is {\BCMRat{B}}. In particular, if $R/xR$ is {\BCMRat{}}, then so is $R$.
\end{proposition}
Compare our proof with the main argument of \cite{ElkikDeformationsOfRational}.
\begin{proof}
Since $R/xR$ is Cohen-Macaulay and $x$ is a nonzerodivisor, $R$ is also Cohen-Macaulay. It suffices to prove $H_\m^d(R)\to H_\m^d(B)$ is injective. The commutative diagram
\[
\xymatrix{
0 \ar[r] & R \ar[r]^{\cdot x} \ar[d] & R \ar[d] \ar[r] & R/xR \ar[r]\ar[d] & 0\\
0\ar[r] & B\ar[r]^{\cdot x} & B\ar[r] & B/xB \ar[r] & 0
}
\]
induces
\[
\xymatrix{
0 \ar[r] & H_\m^{d-1}(R/xR) \ar[r] \ar[d] & H_\m^d(R) \ar[d]^\alpha \ar[r]^{\cdot x} & H_\m^d(R) \ar[r]\ar[d]^\alpha & 0\\
0\ar[r] & H_\m^{d-1}(B/xB) \ar[r] & H_\m^d(B) \ar[r]^{\cdot x} & H_\m^d(B) \ar[r] & 0
}
\]
Since each element of $H_\m^d(R)$ is annihilated by a power of $x$, if $\ker\alpha\neq 0$, then there exists $0\neq\eta\in\ker\alpha$ such that $x\eta=0$. Then a diagram chase shows that $\alpha$ is injective provided $H_\m^{d-1}(R/xR) \to H_\m^{d-1}(B/xB)$ is injective. But this follows because $R/xR$ is {\BCMRat{B/xB}}.
\end{proof}

Next we prove that in characteristic $p>0$, {\BCMRat{}} singularities are the same as $F$-rational singularities. This should be well-known to experts in tight closure theory.

\begin{proposition}
\label{prop: big CM rational=F-rational}
Let $(R,\m)$ be an excellent local ring of characteristic $p>0$. Then $R$ is {\BCMRat{}} if and only if $R$ is $F$-rational.
\end{proposition}
\begin{proof}
Suppose $R$ is {\BCMRat{}}. It is enough to prove $\widehat{R}$ is $F$-rational. Thus by \autoref{rem.BCMrationalcompletion} and \autoref{lem.BCMratNormal}, we may assume $R$ is a complete local domain. By \cite{HochsterHunekeInfiniteIntegralExtensionsAndBigCM}, $R^+$ is a big Cohen-Macaulay algebra so $H_\m^d(R)\to H_\m^d(R^+)$ is injective. This implies $R$ is $F$-rational by \cite[Section 5]{SmithTightClosureParameter}.

Conversely, suppose $R$ is $F$-rational. Then by \autoref{defprop.CharpSings}, $H_\m^d(R)$ is simple in the category of $R$-modules with Frobenius action. Therefore the kernel of $H_\m^d(R)\to H_\m^d(B)$, being $F$-stable, is either $0$ or $H_\m^d(R)$. But the kernel cannot be $H_\m^d(R)$ because $B$ is big Cohen-Macaulay: pick any system of parameters $x_1,\dots,x_d$ of $R$, the class $[\frac{1}{x_1x_2\cdots x_d}]$ is not zero in $H_\m^d(B)$ since $x_1,\dots,x_d$ is a regular sequence on $B$.
\end{proof}

\begin{remark}
Let $(R,\m)$ be a complete local domain of characteristic $p>0$. The proof of \autoref{prop: big CM rational=F-rational} implies that $R$ is \BCMRat{} if and only if $R$ is \BCMRat{B} for one single  big Cohen-Macaulay $B$ (in fact, one can take $B=R^+$ or any big Cohen-Macaulay $R^+$-algebra $B$). We do not know whether this is true in characteristic $0$. However, we will prove in \autoref{prop.tau=tauB} a perfectoid version in mixed characteristic: $R$ is \BCMRat{B} for all integral perfectoid big Cohen-Macaulay $R$-algebras $B$ if and only if $R$ is \BCMRat{B} for one (large enough) integral perfectoid big Cohen-Macaulay $R^+$-algebra $B$.
\end{remark}

Next we show that {\BCMRat{}} singularities are pseudo-rational.

\begin{proposition}
\label{prop: big CM rational implies pseudo-rational}
Let $(R,\m)$ be an excellent local ring that is {\BCMRat{}}. Then $R$ is pesudo-rational. In particular, if $R$ is essentially of finite type over a field of characteristic $0$, then $R$ has rational singularities.
\end{proposition}
\begin{proof}
By \autoref{rem.BCMrationalcompletion}, we may assume $R$ is complete. We know $R$ is normal by \autoref{lem.BCMratNormal} and Cohen-Macaulay (by the definition of {\BCMRat{}}). Therefore to show $R$ is pseudo-rational, it suffices to prove that $H_\m^d(R)\to H_E^d(X, \cO_X)$ is injective for every projective birational map $\pi$: $X\to \Spec R$ with $E$ the pre-image of $\{\m\}$. Let $X=\Proj R[Jt]$ for some ideal $J\subseteq R$. By the Sancho-de-Salas exact sequence \cite{SanchodeSalasBlowingupmorphismswithCohenMacaulayassociatedgradedrings}, we have
$$
[H_{\m+Jt}^d(R[Jt])]_0\to H_\m^d(R)\to H_E^d(X, \cO_X).
$$
Thus in order to prove $H_\m^d(R)\to H_E^d(X, \cO_X)$ is injective, it is enough to show the natural map $H_{\m+Jt}^d(R[Jt])\to H_\m^d(R)$ vanishes.

We let $S=R[Jt]$ and let $\n\subseteq S$ denote the maximal ideal $\m+Jt$. We consider the surjective ring homomorphism $\widehat{S_\n}\to R$. Since $R$ is a domain this map factors through $\widehat{S_\n}/P$ for some minimal prime $P$ of $\widehat{S_\n}$. We note that $\dim(\widehat{S_\n}/P)=\dim R+1$ and hence $R$ is obtained from $\widehat{S_\n}/P$ by killing a height one prime $Q$. Since $R$ and $\widehat{S_\n}$ clearly have the same characteristic, by \cite[Theorem 3.9]{HochsterHunekeApplicationsofBigCM} in equal characteristic and \cite[Theorem 3.1]{HeitmannMaBigCohenMacaulayAlgebraVanishingofTor} in mixed characteristic, we have a commutative diagram:
\[
\xymatrix{
\widehat{S_\n}/P \ar[r] \ar[d] & R\ar[d] \\
B\ar[r] & C
}
\]
where $B$, $C$ are balanced big Cohen-Macaulay algebras for $\widehat{S_\n}/P$ and $R$ respectively. This induces a commutative diagram:
\[
\xymatrix{
H_{\m+Jt}^d(R[Jt])=H_{\n}^d(\widehat{S_\n}) \ar[r] & H_{\n}^d(\widehat{S_\n}/P) \ar[r] \ar[d] & H_\m^d(R)\ar[d]^\alpha \\
& 0=H_{\n}^d(B)\ar[r] & H_\m^d(C)
}
\]
where the bottom left $0$ is because $B$ is a balanced big Cohen-Macaulay $\widehat{S_\n}/P$-algebra and $\dim(\widehat{S_\n}/P)>d$. Chasing this diagram, to show the map $H_{\m+Jt}^d(R[Jt])\to H_\m^d(R)$ vanishes, it suffices to prove $\alpha$ is injective, which follows because $R$ is {\BCMRat{}}.
\end{proof}

Combining the above three Propositions, we have the following theorem:

\begin{theorem}
\label{theorem: F-ratPseudorat}
Let $(R,\m)$ be an excellent local ring of mixed characteristic $(0,p)$ and $x\in R$ be a nonzerodivisor such that $R/xR$ has equal characteristic $p>0$ (e.g. $R$ is an excellent local domain and $x=p$). If $R/xR$ is $F$-rational, then $R$ is pseudo-rational.
\end{theorem}
\begin{proof}
Since $R/xR$ is $F$-rational, by \autoref{prop: big CM rational=F-rational}, $R/xR$ is {\BCMRat{}}. By \autoref{prop: big CM rational deform}, we know $R$ is {\BCMRat{}}. This implies $R$ is pseudo-rational by \autoref{prop: big CM rational implies pseudo-rational}.
\end{proof}

Motivated by \autoref{prop: big CM rational=F-rational} and \autoref{prop: big CM rational implies pseudo-rational}, we conjecture the following:
\begin{conjecture}
Let $(R,\m)$ be a local ring essentially of finite type over a field of characteristic $0$. Then $R$ is {\BCMRat{}} if and only if $R$ has rational singularities.
\end{conjecture}

\section{Dominating integral perfectoid big Cohen-Macaulay $R^+$-algebras}
\label{sec.AndreWeaklyFunctorial}

Our main goal in this section is to prove that, in mixed characteristic, given any set of integral perfectoid big Cohen-Macaulay $R$-algebras (resp. $R^+$-algebras), one can find an integral perfectoid big Cohen-Macaulay $R^+$-algebra that dominates all of them. See \autoref{thm.DominateAnySetOfBCM}. This can be viewed as a mixed characteristic perfectoid analog of the main result in \cite{DietzBCMSeeds}. Our method is largely inspired by Andr\'e's method in \cite{AndreWeaklyFunctorialBigCM}. We start by recalling the following important result of Andr\'{e} \cite[Theorem 3.1.1]{AndreWeaklyFunctorialBigCM} and Shimomoto \cite[Main Theorem 2]{ShimomotoIntegralperfectoidbigCMviaAndre}.

\begin{theorem}[Andr\'{e} \cite{AndreWeaklyFunctorialBigCM}, Shimomoto \cite{ShimomotoIntegralperfectoidbigCMviaAndre}]
\label{theorem: AndreShimomoto}
Let $(R,\m)$ be a complete local domain of mixed characteristic $(0,p)$ with $\underline{x}=p,x_1,\dots,x_{d-1}$ a system of parameters. Let $B$ be a $(pg)^{1/p^\infty}$-almost integral perfectoid $R$-algebra that is $(pg)^{1/p^\infty}$-almost big Cohen-Macaulay with respect to $\underline{x}$. Then
\begin{enumerate}
  \item There exists an integral perfectoid (balanced) big Cohen-Macaulay $R$-algebra $C$ and a morphism $B^{\natural}=(B^\flat)^\sharp\to C$.
  \item We may further map $C$ to an integral perfectoid (balanced) big Cohen-Macaulay $R^+$-algebra $C'$.
\end{enumerate}
\end{theorem}

\begin{remark}
\label{rem.R^+structure}
We caution the reader that, in the above theorem, if $B$ is already an $R^+$-algebra, then the construction of \autoref{theorem: AndreShimomoto} (a) yields a map $R^+\to B^{\natural}\to C$ (i.e., $B^{\natural}\to C$ is $R^+$-linear). However, if we follow the construction of \autoref{theorem: AndreShimomoto} (b), we may result in a {\it different} $R^+$-structure on $C'$ (i.e., the map $B^\natural\to C'$ may not be $R^+$-linear).
\end{remark}

Using \autoref{theorem: AndreShimomoto}, Andr\'{e} proved the following theorem on the weak functoriality of big Cohen-Macaulay algebras \cite[Theorem 1.2.1, 4.1.1]{AndreWeaklyFunctorialBigCM}. Similar results were also obtained independently by Gabber \cite{GabberMSRINotes,GabberRameroFoundationsAlmostRingTheory}.

\begin{theorem}
\label{theorem: Andre}
Any local homomorphism $R\to R'$ of complete local domains, with $R$ of mixed characteristic $(0,p)$, fits into a commutative diagram:
\[\xymatrix{
R \ar[r] \ar[d] & R'\ar[d]\\
R^+ \ar[r] \ar[d] & R'^+\ar[d]\\
B \ar[r] & B'
}\]
where $B$ and $B'$ are integral perfectoid (balanced) big Cohen-Macaulay $R^+$-algebra and $R'^+$-algebra respectively (when $R'$ is of characteristic $p>0$, this means $B'$ is perfect and $p^\flat$-adically complete). Moreover, if $R\to R'$ is surjective, then $B$ can be given in advance.
\end{theorem}

\begin{remark}
\label{rem.functorialityOfNonReducedAndre}
With assumptions as otherwise in \autoref{theorem: Andre}, suppose that $S = R/I$ is not necessarily a domain but instead is reduced and equidimensional (we still assume $R$ is complete).  Let $Q_i \subseteq R$ be the minimal primes of $I$ so that $I = \bigcap_{i = 1}^n Q_i$ and set $S_i = R/Q_i$.  It follows that ${S^{+}} = (R/I)^+ = \prod_{i = 1}^n (R/Q_i)^+ = \prod_{i = 1}^n S^{+}_i$.  We can form $C_i$ for each $S^+_i$ as in \autoref{theorem: Andre} and set $C = \prod_{i = 1}^n C_i$.  It follows immediately that we still have a commutative diagram
\[
\xymatrix{
R \ar[r] \ar[d] & S\ar[d]\\
R^+ \ar[r] \ar[d] & S^+\ar[d]\\
B \ar[r] & C.
}
\]
Notice that $C$ is still an integral perfectoid big Cohen-Macaulay $S^+$-algebra and in particular $H^j_{\fram}(C)= 0$ for every $j < \dim S$.
\end{remark}

\autoref{theorem: Andre} and \autoref{rem.functorialityOfNonReducedAndre} will be crucial in later sections: the weak functoriality property, the fact that $B$, $B'$ are integral perfectoid algebras over $R^+$, $R'^+$ respectively, and the fact that $B$ can be chosen in advance for surjective $R\to R'$. We next begin to prove our domination result. The following is our key lemma. Throughout the rest of this section, $\widehat{\bullet}$ always denotes $p$-adic completion unless otherwise stated.

\begin{theorem}
\label{lemma: seed}
Let $(R,\m)$ be a complete local domain of mixed characteristic $(0,p)$ and let $B_1$, $B_2$ be two integral perfectoid big Cohen-Macaulay $R^+$-algebras. Then $B_1\otimes_{\widehat{{R^+}}}B_2$ maps $\widehat{R^+}$-linearly to another integral perfectoid big Cohen-Macaulay $R^+$-algebra $B$, i.e., we have a commutative diagram of $\widehat{R^+}$-algebras and $\widehat{R^+}$-linear maps
\[\xymatrix{
B_1 \ar[r] & B\\
\widehat{{R^+}}\ar[r] \ar[u] & B_2\ar[u]
}.
\]
\end{theorem}
\begin{proof}
By Cohen's structure theorem we have $(A,\m_A)\to (R,\m)$ that is a module-finite extension such that $A\cong C_k\llbracket x_1,\dots,x_{d-1}\rrbracket$ is a complete and unramified regular local ring, where we are abusing notation a bit and use $C_k$ to denote the unique complete DVR with uniformizer $p$ and residue field $k$. Then $C_{\overline{k}}$ (i.e., the unique complete DVR with uniformizer $p$ and residue field $\overline{k}$) is integral over $C_k$ up to $p$-adic completion. More precisely, we have $C_{\overline{k}} \cong \widehat{V_{\overline{k}}}$ such that $V_{\overline{k}} \subseteq R^+$ is a DVR with uniformizer $p$ and is integral over $C_k$. 
Thus we have natural maps $$C_{\overline{k}}\to \widehat{A^+}\to \widehat{A^+}^{\m_A}=\widehat{R^+}^{\m}.$$ This induces a map $$C_{\overline{k}}\otimes_{C_k}A\to \widehat{R^+}^{\m}.$$ We next $\m_A$-adically complete the above map and we obtain $$A_0:=C_{\overline{k}}\llbracket x_1,\dots,x_{d-1}\rrbracket\to \widehat{R^+}^{\m}.$$
Note that we can replace $B_1$, $B_2$ by $\widehat{B_1}^{\m}$, $\widehat{B_2}^{\m}$, which are integral perfectoid big Cohen-Macaulay algebras by \cite[Corollary 8.5.3]{BrunsHerzog} and \cite[Proposition 2.2.1]{AndreWeaklyFunctorialBigCM} (we are taking completion with respect to $\m$, which has the same effect as taking completion with respect to $(p,x_1,\dots,x_{d-1})$ so \cite[Proposition 2.2.1]{AndreWeaklyFunctorialBigCM} can be applied). Therefore without loss of generality, we may assume $B_1,B_2$ are $A_0$-algebras.

Let $A_{\infty,0}$ be the $p$-adic completion of $A_0[p^{1/p^\infty},x_1^{1/p^\infty},\dots,x_{d-1}^{1/p^\infty}]\subseteq A_0\otimes_AR^+\to
\widehat{R^+}^{\m}$ (here we are fixing once and for all choices of compatible roots of $p, x_1, \dots, x_{d-1}$ inside $\widehat{R^+}^{\m}$). This is an integral perfectoid algebra faithfully flat over $A_0$ (and hence over $A$). Because $B_1$ and $B_2$ are integral perfectoid $\widehat{R^+}^{\m}$-algebras, they are also $A_{\infty, 0}$-algebras.

We consider $B_1\widehat{\otimes}_{A_{\infty, 0}}B_2$. Since $B_i$ is a (balanced) big Cohen-Macaulay $A_0$-algebra and $A_0$ is regular, $B_i$ is faithfully flat over $A_0$ \cite[page 77]{HochsterHunekeInfiniteIntegralExtensionsAndBigCM}. The same reasoning shows that $B_i$ is faithfully flat over $A_0[p^{1/p^l},x_1^{1/p^l},\dots,x_{d-1}^{1/p^l}]$ for every $l$ and thus $B_i$ is faithfully flat over $A_0[p^{1/p^\infty},x_1^{1/p^\infty},\dots,x_{d-1}^{1/p^\infty}]$. This implies $B_i/p$ is faithfully flat over $A_{\infty, 0}/p$ and thus $(B_1\widehat{\otimes}_{A_{\infty, 0}}B_2)/p\cong B_1/p\otimes_{A_{\infty, 0}/p}B_2/p$ is faithfully flat over $A_{\infty, 0}/p$. Therefore $B_1\widehat{\otimes}_{A_{\infty, 0}}B_2$ is big Cohen-Macaulay with respect to $\underline{x}=p, x_1,\dots, x_{d-1}$. We consider
\[
B_1\widehat{\otimes}_{A_{\infty, 0}}B_2 \to \widetilde{B}=(B_1\widehat{\otimes}_{A_{\infty, 0}}B_2)[\frac{1}{p}]^\circ.
\]
Here $\widetilde{B}$ is an integral perfectoid $A_{\infty,0}$-algebra that is $p^{1/p^\infty}$-almost isomorphic to $B_1\widehat{\otimes}_{A_{\infty, 0}}B_2$ by \cite[Proposition 6.18]{ScholzePerfectoidspaces}. In particular, $\widetilde{B}$ is reduced, has no $p$-torsion, and is $p^{1/p^\infty}$-almost big Cohen-Macaulay with respect to $\underline{x}$.

We next consider the ring $R\otimes_AR$. This is a commutative ring and is a finite extension of $R$, and if we tensor with the fraction field $K$ of $A$, then this ring becomes
\[
(R\otimes_AR)\otimes_AK\cong L\otimes_KL,
\]
where $L$ is the fraction field of $R$. Since $K\to L$ is finite separable (as they have characteristic $0$), we know that $L\otimes_KL$ is reduced. Therefore the kernel of $R\otimes_AR\to (R\otimes_AR)\otimes_AK$ is a radical ideal of $R\otimes_AR$, which we call $J$. We now let $S=(R\otimes_AR)/J$, then $S$ is a reduced ring. Because $B_1$ and $B_2$ are faithfully flat over $A$, so is $B_1\otimes_AB_2$, in particular, $B_1\otimes_AB_2$ is torsion-free over $A$ thus $J$ maps to $0$ in $B_1\otimes_AB_2$. It follows that $J$ maps to $0$ in $\widetilde{B}$ as $R\otimes_AR\to \widetilde{B}$ factors through $B_1\otimes_AB_2$. Hence we have an induced map $S\to \widetilde{B}$. The multiplication map $S=(R\otimes_AR)/J\to R$ is surjective (the map is well-defined since $J$ maps to $0$ under the multiplication because $R$ is torsion-free over $A$), and since $R\otimes_AR$ is module-finite over $R$, $S$ has the same dimension as $R$. Therefore $R=S/P$ for a minimal prime $P$ of $S$ which defines the diagonal.  In summary:
\begin{itemize}
\item[(i)]  $S = (R \otimes_A R)/J$ is reduced.
\item[(ii)]  We have an induced map $S \to \widetilde{B}$.
\item[(iii)]  The multiplication map induces $S \twoheadrightarrow R$, the kernel of which is $P$.
\end{itemize}
Since $S$ is reduced and $P$ is a minimal prime of $S$, there exists an element $g \in S$ such that
\begin{itemize}
\item[(iv)] $g \notin P$ and
\item[(v)] $gP=0$.
\end{itemize}
We now set
\[
T=\widetilde{B}\langle g^{1/p^\infty}\rangle[\frac{1}{p}]^\circ,
\]
i.e., this is Andre's construction of adjoining $p^{1/p^\infty}$-root of $g$ to $\widetilde{B}$. By \cite{AndreDirectsummandconjecture} or \cite[Theorem 2.3]{BhattDirectsummandandDerivedvariant}, $T$ is an integral perfectoid $\widetilde{B}$-algebra (so in particular it is reduced) and is $p^{1/p^\infty}$-almost faithfully flat over $\widetilde{B}$ mod $p$. In particular,
 \begin{itemize}
 \item{} $T$ is $p^{1/p^\infty}$-almost big Cohen-Macaulay with respect to $\underline{x}$.
\end{itemize}

We next consider $g^{-1/p^\infty}T=\cap_e (g^{-1/p^e}T)$. By \cite[2.3.2]{AndreWeaklyFunctorialBigCM}, $g^{-1/p^\infty}T$ is a $(pg)^{1/p^\infty}$-almost integral perfectoid $R$-algebra. The key point here is that, since $T$ is reduced, any $g$-torsion in $T$ is automatically $g^{1/p^\infty}$-torsion and thus its image vanishes in $g^{-1/p^\infty}T$, and thus also in $(g^{-1/p^\infty}T)^\natural$. Therefore the map $S\to (g^{-1/p^\infty}T)^\natural$ factors through $R=S/P$ because $gP=0$ in $S$. In sum, we have a commutative diagram:
\[\xymatrix{
A \ar[d] \ar[r] & S=(R\otimes_AR)/J \ar[r] \ar[d] & S/P=R \ar[rd] \\
B_1\widehat{\otimes}_{A_{\infty, 0}}B_2 \ar[r] &\widetilde{B} \ar[r] & T \ar[r] & (g^{-1/p^\infty}T)^\natural \ar[r] & g^{-1/p^\infty}T
}
\]
We now come to the heart of the argument.
\begin{claim}
\label{clm.AlmostCM}
$g^{-1/p^\infty}T$ is a $(pg)^{1/p^\infty}$-almost big Cohen-Macaulay $R$-algebra with respect to $\underline{x}$.
\end{claim}
\begin{proof}[Proof of Claim]
Since $g^{-1/p^\infty}T$ is $g^{1/p^\infty}$-almost isomorphic to $T$ and $T$ is $p^{1/p^\infty}$-almost big Cohen-Macaulay with respect to $\underline{x}$, the colon ideal $$\frac{(p, x_1,\dots,x_{i-1}):_{g^{-1/p^\infty}T}x_i}{(p, x_1,\dots,x_{i-1})(g^{-1/p^\infty}T)}$$ is annihilated by $(pg)^{1/p^\infty}$ for every $i$. Therefore it is enough to prove that
$$\frac{g^{-1/p^\infty}T}{(p, x_1,\dots,x_{d-1})(g^{-1/p^\infty}T)}$$ is not $(pg)^{1/p^\infty}$-almost zero. If on the contrary, we have
$$(pg)^{1/p^e}\in (p, x_1,\dots,x_{d-1})(g^{-1/p^\infty}T)$$ for every $e$, it follows that $$pg\in (p,x_1,\dots,x_{d-1})^{p^e}(g^{-1/p^\infty}T)$$ for every $e$. Since $g^{-1/p^\infty}T$ is $g^{1/p^\infty}$-almost isomorphic to $T$, $T$ is $p^{1/p^\infty}$-almost faithfully flat over $\widetilde{B}$ mod $p^m$ for every $m$, and $\widetilde{B}$ is $p^{1/p^\infty}$-almost isomorphic to $B_1\widehat{\otimes}_{A_{\infty, 0}}B_2$, we can deduce that
\[
(pg)^2\in (p,x_1,\dots,x_{d-1})^{p^e}B_1\widehat{\otimes}_{A_{\infty, 0}}B_2
\]
for every $e$, where we are abusing notation and interpreting $g \in S$ as a lift to an element of $R\otimes_A R$. Next we note that modulo $p^m$, we have
$$(B_1\widehat{\otimes}_{A_{\infty, 0}}B_2)/p^m\cong (B_1\otimes_{A_0[p^{1/p^\infty},x_1^{1/p^\infty},\dots,x_{d-1}^{1/p^\infty}]}B_2)/p^m.$$ Therefore $(pg)^2\in (p,x_1,\dots,x_{d-1})^{p^e}(B_1\otimes_{A_0[p^{1/p^\infty},x_1^{1/p^\infty},\dots,x_{d-1}^{1/p^\infty}]}B_2)$ for every $e$. But then we know that there exists $A_n=A_0[p^{1/p^n},x_1^{1/p^n},\dots,x_{d-1}^{1/p^n}]$ for some $n\gg0$ depending on $e$ such that
$$(pg)^2\in (p,x_1,\dots,x_{d-1})^{p^e}(B_1\otimes_{A_n}B_2).$$
But since $A_0$ is the $(p,x_1,\dots,x_{d-1})$-adic completion of $C_{\overline{k}}\otimes_{C_k}A$ and $C_{\overline{k}}\cong \widehat{V_{\overline{k}}}$, we have
$$A_n/(p,x_1,\dots,x_{d-1})^{p^e}\cong (V_{\overline{k}}\otimes_{C_k}A[p^{1/p^n},x_1^{1/p^n},\dots,x_{d-1}^{1/p^n}])/(p,x_1,\dots,x_{d-1})^{p^e}.$$
Since $V_{\overline{k}}$ is integral over $C_k$, we have induced map $$V_{\overline{k}}\otimes_{C_k}A[p^{1/p^n},x_1^{1/p^n},\dots,x_{d-1}^{1/p^n}]\to A^+.$$
It follows that there exists a finite domain extension $A'$ of $A$ such that
$$(pg)^2\in (p,x_1,\dots,x_{d-1})^{p^e}(B_1\otimes_{A'}B_2).$$
At this point, we mod out by the prime ideal $P$ (again, abusing notation and identifying $P$ with the diagonal minimal prime ideal of $R\otimes_AR$ so that $(R\otimes_AR)/P=R$).
Since we have $$\frac{B_1\otimes_{A'}B_2}{P(B_1\otimes_{A'}B_2)}\to B_1\otimes_{A'[R]}B_2,$$ it follows that, after we mod out by $P$, we obtain\footnote{Here we are using the fact that both $A'$ and $R$ are inside $R^+$ and $B_1$, $B_2$ are $R^+$-algebras.}
\[
(pg)^2 \in \left((p,x_1,\dots,x_{d-1})^{p^e}(B_1\otimes_{A'[R]}B_2) \right)\cap R.
\]
Now $A'[R]$ is a module-finite domain extension of $R$, so it is a complete local domain. Since $B_1$, $B_2$ are big Cohen-Macaulay algebras over $A'[R]$, they are solid $A'[R]$-algebras by \cite[Corollary 2.4]{HochsterSolidClosure}. Thus $B_1\otimes_{A'[R]}B_2$ is a solid $A'[R]$-algebra by \cite[Proposition 2.1 (a)]{HochsterSolidClosure}, and hence $B_1\otimes_{A'[R]}B_2$ is a solid $R$-algebra by \cite[Corollary 2.3]{HochsterSolidClosure}. Therefore in $R$, we have $$(pg)^2\in ((p,x_1,\dots,x_{d-1})^{p^e})^\bigstar \subseteq \overline{(x_1,\dots,x_d)^{p^e}}$$ for every $e$, where $I^\bigstar$ denotes the solid closure of $I$, which is always contained inside the integral closure by \cite[Theorem 5.10]{HochsterSolidClosure}. However, since $g\notin P$, the element $(pg)^2$ is nonzero in $R$ and hence $$0\neq (pg)^2\in \cap_e\overline{(x_1,\dots,x_d)^{p^e}}=0$$ which is a contradiction.  This completes the proof of \autoref{clm.AlmostCM}.\footnote{There is an alternative approach to \autoref{clm.AlmostCM} suggested by the referee: we can first map $\widetilde{B}$ to an integral perfectoid big Cohen--Macaulay algebra $\widetilde{B}'$ (in order to show this, one needs to check the non-triviality condition for $\widetilde{B}$, which is essentially the second half of the current proof but with $p^{1/p^\infty}$ in place of $(pg)^{1/p^\infty}$), and then we use \cite[2.5.1]{AndreWeaklyFunctorialBigCM} to handle the $g^{1/p^\infty}$-almost issue.}
\end{proof}

By \autoref{clm.AlmostCM}, we can apply \autoref{theorem: AndreShimomoto} to the $(pg)^{1/p^\infty}$-almost integral perfectoid, $(pg)^{1/p^\infty}$-almost big Cohen-Macaulay $R$-algebra $g^{-1/p^\infty}T$, we know that there exists an integral perfectoid (balanced) big Cohen-Macaulay $(g^{-1/p^\infty}T)^{\natural}$-algebra $B'$ such that $B_1\otimes_RB_2$ maps to $B'$. This is because the kernel of $B_1\otimes_AB_2\to B_1\otimes_RB_2$ is precisely the extended ideal $P(B_1\otimes_AB_2)$, and we have shown that the image of $P$ vanishes in $(g^{-1/p^\infty}T)^{\natural}$. Next we need to refine $B'$ to an $R^+$-algebra $B$ such that we have an $R^+$-linear map $B_1\otimes_{R^+}B_2\to B$. This does not follow immediately from \autoref{theorem: AndreShimomoto} (for example see \autoref{rem.R^+structure}) so we proceed carefully below.

\subsection*{Making $B'$ a big Cohen-Macaulay $R^+$ algebra}
The argument below is completely parallel to \cite[4.3.5]{AndreWeaklyFunctorialBigCM}.  However, we include it for the convenience of the reader and since we are dealing with diagrams not exactly dealt with in {\it loc. cit.}
The rough idea is to run the above construction for all module-finite domain extensions of $R$ and take the direct limit. Below we give the details. We write $R^+=\varinjlim_\beta R_\beta$ for a directed system of normal module-finite domain extensions $\{R_\beta\}_\beta$ of $R$.

We let $S_\beta=(R_\beta\otimes_AR_\beta)/J_\beta$ and let $R_\beta=S_\beta/P_\beta$ for the minimal prime $P_\beta$ of $S_\beta$ defining the diagonal. By the same reasoning as above, there exists $g_{\beta} \in S_{\beta}$ such that $g_\beta\notin P_\beta$ and $g_\beta P_\beta=0$. Moreover, if $\alpha<\beta$, then it is straightforward to check that the image of $g_\alpha$ in $S_\beta/P_\beta$ is nonzero. Let $\underline{g}_\beta$ be a finite sequence of elements of $S_\beta$ such that each element in this sequence is the image of $g_\alpha$ for some $\alpha\leq \beta$ (so each element is not contained in $P_\beta$) and that $(\prod \underline{g}_\beta)P_\beta =0$ in $S_\beta$.
We recall that we have an almost isomorphism
\[
B_1\widehat{\otimes}_{A_{\infty, 0}}B_2 \to \widetilde{B}=(B_1\widehat{\otimes}_{A_{\infty, 0}}B_2)[\frac{1}{p}]^\circ
\]
where $\widetilde{B}$ is an integral perfectoid $A_{\infty,0}$-algebra. 

We next consider the directed system formed by tuples $\underline{\beta}=(\beta, \underline{g}_\beta)$, with order $(\beta, \underline{g}_\beta)\leq (\beta', \underline{g}_{\beta'})$ if $\beta\leq \beta'$ and the image of $\underline{g}_\beta$ forms part of the sequence $\underline{g}_{\beta'}$. Also note that $\{S_\beta\}_{\underline{\beta}}$ also forms a direct limit system, and all the $S_\beta$ map to $\widetilde{B}$ compatibly because $B_1$ and $B_2$ are $R^+$-algebras. For each $\underline{\beta}$, we set
\[
T_{\underline{\beta}}=\widetilde{B}\langle \underline{g}_\beta^{1/p^\infty} \rangle[\frac{1}{p}]^\circ.
\]
By \autoref{clm.AlmostCM}, $(\prod \underline{g}_\beta)^{-1/p^\infty}T_{\underline{\beta}}$ is a $(p\prod \underline{g}_\beta)^{1/p^\infty}$-almost integral perfectoid, $(p\prod \underline{g}_\beta)^{1/p^\infty}$-almost big Cohen-Macaulay (with respect to $\underline{x}$) $R$-algebra, and we have a commutative diagram for every $\underline{\beta}$:
\[
\xymatrix{
A\ar[r]\ar[d] & S_\beta \ar[r] \ar[d] & R_\beta \ar[rd] \\
B_1\widehat{\otimes}_{A_{\infty, 0}}B_2 \ar[r] & \widetilde{B} \ar[r] & T_{\underline{\beta}} \ar[r] & ((\prod \underline{g}_\beta)^{-1/p^\infty}T_{\underline{\beta}})^\natural \ar[r] & (\prod \underline{g}_\beta)^{-1/p^\infty}T_{\underline{\beta}}
}
\]
where the existence of the map $R_\beta\to ((\prod \underline{g}_\beta)^{-1/p^\infty}T_{\underline{\beta}})^\natural$ follows from $(\prod \underline{g}_\beta)P_\beta =0$. We set $B_{\underline{\beta}}=((\prod \underline{g}_\beta)^{-1/p^\infty}T_{\underline{\beta}})^\natural$ and by the same reasoning as above, we know that $B_{\underline{\beta}}$ is an integral perfectoid, $(p\prod \underline{g}_\beta)^{1/p^\infty}$-almost big Cohen-Macaulay (with respect to $\underline{x}$), $B_1\otimes_{R_\beta}B_2$-algebra. Moreover, the constructions are compatible in the sense that whenever $\beta\leq \beta'$ and $\underline{\beta}\leq\underline{\beta}'$, we have the following commutative diagram:
\[\xymatrix{
R_\beta \ar[r] \ar[d] & R_{\beta'} \ar[d] \\
B_{\underline{\beta}} \ar[r] & B_{\underline{\beta}'}
}
\]
Taking the $p$-adic completion of the direct limit, we obtain a map of integral perfectoid ${R^+}$-algebras $$\widehat{R^+}\to \widehat{\varinjlim_{\underline{\beta}} B_{\underline{\beta}}}.$$ Moreover, we know that this map factors through $B_1$ and $B_2$ since $B_1\otimes_{R_\beta}B_2$ maps $R_\beta$-linearly to $B_{\underline{\beta}}$ for every $\underline{\beta}$. Thus we have $B_1\otimes_{\widehat{{R^+}}}B_2\to \widehat{\varinjlim_{\underline{\beta}} B_{\underline{\beta}}}$.

Therefore it remains to prove that $\widehat{\varinjlim_{\underline{\beta}} B_{\underline{\beta}}}$ can be mapped to an integral perfectoid big Cohen-Macaulay ${R^+}$-algebra $B$. But since $B_{\underline{\beta}}$ is $(p\prod \underline{g}_\beta)^{1/p^\infty}$-almost big Cohen-Macaulay with respect to $\underline{x}=p, x_1,\dots,x_{d-1}$, $B_{\underline{\beta}}^\flat$ is $((p\prod \underline{g}_\beta)^\flat)^{1/p^\infty}$-almost big Cohen-Macaulay with respect to $\underline{x}^\flat=p^\flat, x_1^\flat,\dots, x_{d-1}^\flat$. This implies that each $B_{\underline{\beta}}^\flat$ maps to a balanced big Cohen-Macaulay algebra, and hence so does the direct limit $\varinjlim_{\underline{\beta}} B_{\underline{\beta}}^\flat$ by \cite[Lemma 3.2]{DietzBCMSeeds}.

Since we are in characteristic $p>0$, $\varinjlim_{\underline{\beta}} B_{\underline{\beta}}^\flat$ maps to a perfect (balanced) big Cohen-Macaulay algebra $C$ that is $(p^\flat, x_1^\flat,\dots, x_{d-1}^\flat)$-adically complete by \cite[Proposition 3.7]{DietzBCMSeeds}. Thus we have an induced map of perfect, $p^\flat$-adically complete rings of characteristic $p>0$: $$\widehat{\varinjlim_{\underline{\beta}} B_{\underline{\beta}}^\flat}^{p^\flat}\to C.$$ Untilting, we get: $$\widehat{\varinjlim_{\underline{\beta}} B_{\underline{\beta}}} \to C^\sharp \to B=\widehat{C^\sharp}^{\m}$$
Since $C$ is perfect, $p^\flat$-adically complete, and big Cohen-Macaulay with respect to $\underline{x}^\flat$, $C^\sharp$ is integral perfectoid and big Cohen-Macaulay with respect to $\underline{x}$. Therefore $B$ is balanced big Cohen-Macaulay by \cite[Corollary 8.5.3]{BrunsHerzog}. Since completion with respect to $\m$ is the same as completion with respect to $(p, x_1,\dots,x_{d-1})$ and the latter is a regular sequence on ${C^\sharp}$, we also have that $B$ is integral perfectoid by \cite[Proposition 2.2.1]{AndreWeaklyFunctorialBigCM}.
\end{proof}

In order to generalize the above result to an arbitrary set of integral perfectoid big Cohen-Macaulay algebras, we introduced the following definition which is a perfectoid version of \cite[Definition 3.1]{DietzBCMSeeds}.
\begin{definition}
Let $(R,\m)$ be a complete local domain of mixed characteristic $(0,p)$. An $R$-algebra $S$ is called a {\it perfectoid seed} if $S$ is integral perfectoid and it maps to an integral perfectoid (balanced) big Cohen-Macaulay $R$-algebra.
\end{definition}

Here is a perfectoid version of \cite[Lemma 3.2]{DietzBCMSeeds}

\begin{lemma}
\label{lem.LimPerfectoidSeed}
Let $(R,\m)$ be a complete local domain of mixed characteristic $(0,p)$. If $\{S_\lambda\}$ is a direct system of perfectoid seeds, then $\widehat{\varinjlim S_\lambda}$ is also a perfectoid seed.
\end{lemma}
\begin{proof}
Fix a system of parameters $\underline{x}=p,x_1,\dots,x_{d-1}$ of $R$. Let $S_\lambda\to B_\lambda$ be such that $B_\lambda$ is an integral perfectoid big Cohen-Macaulay algebra (but $\{B_\lambda\}$ may not form a directed system as we are not assuming there are maps between the $B_\lambda$'s). Then $p^\flat, x_1^\flat,\dots,x_{d-1}^\flat$ is a regular sequence on $B_\lambda^\flat$ because $B_\lambda/p\cong B_\lambda^\flat/p^\flat$. Since $S_\lambda^\flat\to B_\lambda^\flat$, $S_\lambda^\flat$ is a seed in characteristic $p>0$.\footnote{Here we can just work over the ring $\mathbb{F}_p[p^\flat,x_1^\flat,\dots,x_{d-1}^\flat]$.} By \cite[Lemma 3.2]{DietzBCMSeeds}, $\varinjlim S_\lambda^\flat$ is also a seed in characteristic $p>0$. Therefore $\widehat{\varinjlim S_\lambda^\flat}^{p^\flat}$ maps to a perfect and $p^\flat$-adically complete (balanced) big Cohen-Macaulay algebra $B$ by \cite[Proposition 3.7]{DietzBCMSeeds}. Untilting, we have (using \cite[Remark 6.2.9]{BhattLectureNotesPerfectoidSpaces})
\[
\widehat{\varinjlim S_\lambda} \to B^\sharp\to \widehat{B^\sharp}^{\m}
 \]
where $\widehat{B^\sharp}^{\m}$ is an integral perfectoid balanced big Cohen-Macaulay algebra by \cite[Corollary 8.5.3]{BrunsHerzog} and \cite[Proposition 2.2.1]{AndreWeaklyFunctorialBigCM} as before.  This finishes the proof.
\end{proof}

The following is the main result of this section, which is the mixed characteristic perfectoid analog of \cite[Theorem 8.10]{DietzBCMSeeds}. We will use this result in \autoref{sec.BCMParamTest} and \autoref{sec.BCMTestIdeals}.

\begin{theorem}
\label{thm.DominateAnySetOfBCM}
Let $(R, \fram)$ be a complete local domain of mixed characteristic $(0, p)$. Let $\{ B_{\gamma} \}_{\gamma \in \Gamma}$ be any set of integral perfectoid big Cohen-Macaulay $R$-algebras (resp. $R^+$-algebras). Then there is an integral perfectoid big Cohen-Macaulay $R^+$-algebra $B$ such that the map $R\to B$ (resp. $R^+\to B$) factors through all the maps $R\to B_{\gamma}$ (resp. $R^+\to B_\gamma$).
\end{theorem}
\begin{proof}
By \autoref{theorem: AndreShimomoto} every integral perfectoid big Cohen-Macaulay $R$-algebra $B_\gamma$ maps to an integral perfectoid big Cohen-Macaulay $R^+$-algebra. Therefore it is enough to prove the statement when all $B_\gamma$ are $R^+$-algebras. Now, for any finite subset $\Lambda=\{\gamma_1,\dots,\gamma_n\}\subseteq \Gamma$, we set $B_\Lambda=B_{\gamma_1}\widehat{\otimes}_{\widehat{R^+}}B_{\gamma_2}\widehat{\otimes}_{\widehat{R^+}}\cdots\widehat{\otimes}_{\widehat{R^+}}B_{\gamma_n}$ and set $\widetilde{B_\Lambda}=B_\Lambda[1/p]^\circ$. By \cite[Proposition 6.18]{ScholzePerfectoidspaces}, $\widetilde{{B_\Lambda}}$ is integral perfectoid. By repeatedly applying \autoref{lemma: seed}, $B_\Lambda$ maps to an integral perfectoid big Cohen-Macaulay $R^+$-algebra $C$. This induces $$\widetilde{B_\Lambda}\to \widetilde{C}:=C[1/p]^\circ.$$ Since $C$ is integral perfectoid, $\widetilde{C}$ is also integral perfectoid and is $p^{1/p^\infty}$-almost isomorphic to $C$ (in particular it is $p^{1/p^\infty}$-almost big Cohen-Macaulay). By \autoref{theorem: AndreShimomoto}, $\widetilde{C}$ maps to an integral perfectoid big Cohen-Macaulay algebra. Thus each $\widetilde{B_\Lambda}$ is a perfectoid seed. But clearly we have
$$\{\widetilde{B_\Lambda}\}_{\Lambda\subseteq \Gamma, |\Lambda|<\infty}$$
is a {\it directed set} of $R^+$-algebras (where the transition maps are the obvious ones).
Therefore by \autoref{lem.LimPerfectoidSeed}, $\widehat{\varinjlim \widetilde{B_\Lambda}}$ is also a perfectoid seed; so it maps to an integral perfectoid big Cohen-Macaulay algebra $B$. Since $R^+\to B$ factors through all $B_\gamma$, $B$ satisfies our conclusion. This finishes the proof.
\end{proof}

\begin{remark}
In the proof of \autoref{lemma: seed} and \autoref{thm.DominateAnySetOfBCM}, it is important that our big Cohen-Macaulay algebras are integral perfectoid. It is natural to ask whether any big Cohen-Macaulay algebra $B$ can be mapped to an integral perfectoid one. We do not know the answer. As suggested by the referee, we point out that by \cite[Tag 0DCR]{stacks-project}, we can always map $B$ to an absolutely integrally closed big Cohen-Macaulay algebra $C$, but it is not clear that the $p$-adic completion of $C$ is integral perfectoid.
\end{remark}

\begin{remark}
Recently, we have used the ``perfectoidization" functor of Bhatt and Scholze \cite{BhattScholzepPrismaticCohomology} to reprove and slightly generalize the results in \cite{AndreWeaklyFunctorialBigCM} and most of the results in this section, see \cite[Appendix A]{MSTWWAdjoint}. {\it Very recently}, Bhatt has announced that $\widehat{R^+}$ itself is an integral perfectoid big Cohen-Macaulay algebra in mixed characteristic \cite{BhattCohenMacaulaynessAbsoluteIntegralClosures}.  This big Cohen-Macaulay algebra is sufficient for most applications in this paper, but at the moment, results such as \autoref{lem.CanonicalBconstruction}, \autoref{prop.tau=tauB} or \autoref{prop.testidealsmallperturb} need the larger $B$s we consider in this paper.
\end{remark}

\section{Perfectoid big Cohen-Macaulay parameter test submodules}
\label{sec.BCMParamTest}

In this section we define perfectoid big Cohen-Macaulay parameter test submodule and prove many properties of it in analogy with the parameter test submodule in characteristic $p>0$ as well as the Grauert-Riemenschneider multiplier submodule  in characteristic $0$.

We recall that for a local ring $(R,\m)$, $\omega_R^\vee=\Hom_R(\omega_R, E)=H_\m^d(R)$ (and if $R$ is complete we also have $H_\m^d(R)^\vee=\omega_R$). So we can view $H_\m^d(R)$ as maps from $\omega_R$ to $E$. For a submodule $N$ of $H_\m^d(R)$, we define $\Ann_{\omega_R}N=\{z\in\omega_R \hspace{0.5em}| \hspace{0.5em} \eta(z)=0 \text{ for all } \eta\in N\}$.

\begin{definition}
\label{def.tauOmega}
Let $(R,\m)$ be an excellent local domain of dimension $d$ with a normalized dualizing complex $\omega_R^{\mydot}$ and canonical module $\omega_R$.  For every big Cohen-Macaulay $R$-algebra $B$ we define $0^B_{H_\m^d(R)}=\ker\big(H_\m^d(R)\to H_\m^d(B)\big)$. Furthermore, if $R$ has mixed characteristic $(0,p)$, we define
\[
\begin{array}{rl}
0^\scr{B}_{H_\m^d(R)}= \{ & \eta\in H_\m^d(R) \hspace{0.5em} |  \hspace{0.5em} \exists B \text{ an integral perfectoid big Cohen-Macaulay $R$-algebra} \\
      & \text{such that } \eta\in0^B_{H_\m^d(R)}\}.
\end{array}
\]
This is a submodule of $H_\m^d(R)$: if $\eta_1=0$ in $H_\m^d(B_1)$ and $\eta_2=0$ in $H_\m^d(B_2)$, take $B$ such that $B_1\otimes_R B_2\to B$ by \autoref{thm.DominateAnySetOfBCM}, then the image of both $\eta_1$ and $\eta_2$ are $0$ in $H_\m^d(B)$ so $\eta_1+\eta_2\in 0^\scr{B}_{H_\m^d(R)}$. We further define
\[
\mytau_B(\omega_R)=\Ann_{\omega_R}0^B_{H_\m^d(R)}\subseteq \omega_R, \text{ and } \mytau_{\scr{B}}(\omega_R)=\Ann_{\omega_R}0^\scr{B}_{H_\m^d(R)} \subseteq \omega_R.
\]
We call $\mytau_B(\omega_R)$ (resp. $\mytau_{\scr{B}}(\omega_R)$) the  \BCM{} test submodule of $\omega_R$ with respect to $B$ (resp. the perfectoid \BCM{} test submodule of $\omega_R$). We note that if $R$ is complete, then we have $$\mytau_B(\omega_R)=\left(H_\m^d(R)/0^B_{H_\m^d(R)} \right)^\vee \text{ and } \mytau_{\scr{B}}(\omega_R)=\left(H_\m^d(R)/0^\scr{B}_{H_\m^d(R)} \right)^\vee.$$
\end{definition}

\begin{remark}
\label{remark:completion}
\begin{enumerate}
  \item It is clear from \autoref{def.tauOmega} that in mixed characteristic, $0^B_{H_\m^d(R)}\subseteq 0^\scr{B}_{H_\m^d(R)}$ and $\mytau_{\scr{B}}(\omega_R)\subseteq \mytau_{B}(\omega_R)$ for all integral perfectoid big Cohen-Macaulay $R$-algebras $B$. We have $0^B_{H_\m^d(R)}=0^{\widehat{B}}_{H_\m^d(\widehat{R})}$ as in \autoref{rem.BCMrationalcompletion}. However, it is {\it not} clear that $\mytau_B(\omega_R)\otimes_R\widehat{R}=\mytau_{\widehat{B}}(\omega_{\widehat{R}})$. Because of this, we will mostly work with complete local rings.
  \item Since every integral perfectiod big Cohen-Macaulay $R$-algebra maps to an integral perfectoid big Cohen-Macaulay $R^+$-algebra by \autoref{theorem: AndreShimomoto}, we have
\[
\begin{array}{rl}
0^\scr{B}_{H_\m^d(R)}= \{ & \eta\in H_\m^d(R) \hspace{0.5em} |  \hspace{0.5em} \exists B \text{ an integral perfectoid big Cohen-Macaulay $R^+$-algebra} \\
      & \text{such that } \eta\in0^B_{H_\m^d(R)}\}.
\end{array}
\]
\end{enumerate}
\end{remark}

We start by proving that in characteristic $p>0$, $\mytau_B(\omega_R)$ is the same as the parameter test submodule for large enough $B$. This should be well-known to experts.

\begin{proposition}
\label{prop: big CM test ideal char p}
Let $R$ be a complete local domain of dimension $d$ and characteristic $p>0$. Then for every big Cohen-Macaulay $R$-algebra $B$, $\mytau_B(\omega_R)\supseteq \tau(\omega_R)$. Moreover, if $B$ is an $R^+$-algebra, then $\mytau_B(\omega_R)=\tau(\omega_R)$.
\end{proposition}
\begin{proof}
By local duality, it is enough to show that for every big Cohen-Macaulay $R$-algebra $B$, $0^B_{H_\m^d(R)}\subseteq 0^*_{H_\m^d(R)}$. However, $0^B_{H_\m^d(R)}=\ker(H_\m^d(R)\to H_\m^d(B))$ is a proper $F$-stable submodule of $H_\m^d(R)$ (to see it is proper, pick any system of parameters $x_1,\dots,x_d$ of $R$, the class $[\frac{1}{x_1x_2\cdots x_d}]$ is not zero in $H_\m^d(B)$ since $x_1,\dots,x_d$ is a regular sequence on $B$). But since $R$ is a complete local domain, $0^*_{H_\m^d(R)}$ is the unique largest proper $F$-stable submodule of $H_\m^d(R)$, see \cite[Proposition 2.5]{SmithFRatImpliesRat} or \cite[Proposition 2.11]{EnescuHochsterTheFrobeniusStructureOfLocalCohomology}. Thus $0^B_{H_\m^d(R)}\subseteq 0^*_{H_\m^d(R)}$.

Finally, when $B$ is an $R^+$-algebra, $0^B_{H_\m^d(R)} \supseteq 0^+_{H_\m^d(R)}=0^*_{H_\m^d(R)}$ by \cite[Theorem 5.1]{SmithTightClosureParameter} and so $\mytau_B(\omega_R)\subseteq \tau(\omega_R)$.
\end{proof}

\begin{remark}
The conclusion of \autoref{prop: big CM test ideal char p} is not true in general if we do not assume $R$ is a domain: for example take $R=k\llbracket x,y\rrbracket /(xy)$ and $B=k\llbracket x,y\rrbracket /(x)$, then $B$ is a big Cohen-Macaulay $R$-algebra but $0^B_{H_\m^1(R)}\nsubseteq 0^*_{H_\m^1(R)}$.
\end{remark}

Next we point out that, in mixed characteristic, $\mytau_\scr{B}(\omega_R)=\mytau_B(\omega_R)$ for all sufficiently large choices of integral perfectoid $B$. Therefore this suggests that $\mytau_\scr{B}(\omega_R)$ might be a good replacement of $\tau(\omega_R)$ in mixed characteristic. In fact, by the axiom of global choice (or by working in a fixed Grothendieck universe), we can pick an integral perfectoid big Cohen-Macaulay $R$-algebra $B_\eta$ for each element $\eta\in 0^\scr{B}_{H_\m^d(R)}$. Since $0^\scr{B}_{H_\m^d(R)}$ is a set, by \autoref{thm.DominateAnySetOfBCM} there exists an integral perfectoid big Cohen-Macaulay $R^+$-algebra $B$ that dominates every $B_\eta$. It follows that $0^\scr{B}_{H_\m^d(R)}=0^{B}_{H_\m^d(R)}$ and so $\mytau_\scr{B}(\omega_R)=\mytau_B(\omega_R)$.  

Our next goal is to prove a stronger result, which will be the key ingredient to show that our perfectoid parameter test submodule is stable under small perturbation.

\begin{definition}
\label{def.canonicalB}
Let $(R,\m)$ be a complete local domain of mixed characteristic $(0,p)$ and of dimension $d$. Let $B$ denote an integral perfectoid big Cohen-Macaulay $R^+$-algebra. We define
\[
\widetilde{0}^{B}_{H_\m^d(R)}=\{\eta\in H_\m^d(R) \hspace{0.5em} |  \hspace{0.5em} \exists \hspace{0.2em} 0\neq g\in R^+ \text{ such that } g^{1/p^\infty}\eta=0 \text{ in } H_\m^d(B)\}.
\]
\[
\widetilde{0}^{\scr{B}}_{H_\m^d(R)}=\{\eta\in H_\m^d(R) \hspace{0.5em} |  \hspace{0.5em} \exists \hspace{0.2em} 0\neq g\in R^+ \text{ and $\exists B$ such that } g^{1/p^\infty}\eta=0 \text{ in } H_\m^d(B)\}.
\footnote{Again, $\widetilde{0}^{\scr{B}}_{H_\m^d(R)}$ is a submodule: if $g^{1/p^\infty}\eta_1=0$ in $H_\m^d(B_1)$ and $h^{1/p^\infty}\eta_2=0$ in $H_\m^d(B_2)$, then take $B$ such that $B_1\otimes_{R^+}B_2\to B$ by \autoref{thm.DominateAnySetOfBCM}; we have $(gh)^{1/p^\infty}(\eta_1+\eta_2)=0$ in $H_\m^d(B)$ and hence $\eta_1+\eta_2\in \widetilde{0}^{\scr{B}}_{H_\m^d(R)}$.}
\]
\end{definition}



\begin{lemma}
\label{lem.CanonicalBconstruction}
Let $(R,\m)$ be a complete local domain of mixed characteristic $(0,p)$ and let $B$ be an integral perfectoid big Cohen-Macaulay $R^+$-algebra. Suppose $\eta\in H_\m^d(B)$ satisfies $g^{1/p^\infty}\eta=0$ in $H_\m^d(B)$ for some $0\neq g\in R^+$. Then there exists an integral perfectoid big Cohen-Macaulay $B$-algebra $B'$ such that $\eta=0$ in $H_\m^d(B')$.
\end{lemma}
\begin{proof}
We fix a system of parameters $p,x_1\cdots,x_{d-1}$ of $R$. We first show that $p^t\eta=0$ in $H_\m^d(B)$ if and only if we can write $\eta=\frac{z}{p^tx_1^w\cdots x_{d-1}^w}$ for some $z\in B$. Suppose $\eta=\frac{z'}{p^wx_1^w\cdots x_{d-1}^w}$ for some $z'\in B$ and some $w\gg0$. Since $B$ is big Cohen-Macaulay and $p^t\eta=0$, we have $p^tz'\in (p^w,x_1^w,\dots,x_{d-1}^w)B$ and thus $z'\in (p^{w-t}, x_1^w,\dots,x_{d-1}^w)B$. This means $z'=p^{w-t}z+x_1^wy_1+\cdots +x_{d-1}^wy_{d-1}$, which implies $\eta=\frac{z}{p^tx_1^w\cdots x_{d-1}^w}$. The other direction is obvious.

We first prove the case that $p\eta=0$. By the discussion above and replacing $x_i$ by $x_i^w$, we can write $\eta=\frac{z}{px_1\cdots x_{d-1}}$ for some $z\in B$. Since $g^{1/p^\infty}\eta=0$ in $H_\m^d(B)$ and $B$ is big Cohen-Macaulay, we have
$$g^{1/p^\infty}z\in (p, x_1,\dots,x_{d-1})B.$$
This is equivalent to saying that
$$(g^\flat)^{1/p^\infty}z^\flat \in (p^\flat, x_1^\flat,\dots,x_{d-1}^\flat)B^\flat$$
since $B/p\cong B^\flat/p^\flat$. Here $g^\flat=\langle g, g^{1/p}, g^{1/p^2}, \dots \rangle$ and similarly for $x_i^\flat$ and $z^\flat$ ($z$ may not have a compatible system of $p$-power roots in $B$ since we are not assuming $z$ is in $R^+$, but we can choose an arbitrary expression of $z^\flat$ and replace $z$ by $(z^\flat)^\sharp$, this doesn't affect the containment because the ideal contains $p$).

Next we prove that there is a big Cohen-Macaulay $B^\flat$-algebra (in characteristic $p$) that forces $z^\flat$ in $(p^\flat, x_1^\flat,\dots,x_{d-1}^\flat)$. We consider the following sequence:
\begin{equation}
\label{equation--partial algebra modification}
B^\flat\to T_0=\Big(\frac{B^\flat[Y_0,\dots, Y_{d-1}]}{z^\flat-p^\flat Y_0-x_1^\flat Y_1-\cdots-x_{d-1}^\flat Y_{d-1}}\Big)_{\leq N}\to T_1\to T_2\to\cdots\to T_r
\end{equation}
where $T_{i+1}$ is a partial algebra modification of $T_i$ with respect to $p^\flat, x_1^\flat,\dots,x_{d-1}^\flat$ for $i\geq 0$ in the sense of \cite[Section 4]{HochsterBigCohen-Macaulayalgebrasindimensionthree}. Following \cite[Theorem 4.2]{HochsterBigCohen-Macaulayalgebrasindimensionthree} or \cite[Section 11]{HochsterSolidClosure}, it is enough to show that there is no such sequence where the image of $1$ in $T_r$ is in $(p^\flat,x_1^\flat,\dots,x_{d-1}^\flat) T_r$.

Now suppose $1\in (p^\flat,x_1^\flat,\dots,x_{d-1}^\flat) T_r$ in \autoref{equation--partial algebra modification}, and suppose $(g^\flat)^{1/p^e}z^\flat=a_0p^\flat +a_1x_1^\flat+\cdots +a_{d-1}x_{d-1}^\flat$. We look at the following commutative diagram:
\[
\xymatrix@C=15pt{
B^\flat\ar[r]\ar[d] & T_0 \ar[d]\ar[r]  & T_1\ar[r]\ar[d] &\cdots \ar[r]&T_r\ar[d] \\
B^\flat\ar[r]& B^\flat\cdot \frac{1}{((g^\flat)^{1/p^e})^N}\ar[r] & B^\flat\cdot \frac{1}{((g^\flat)^{1/p^e})^{N_1}} \ar[r]  &\cdots \ar[r]& B^\flat\cdot \frac{1}{((g^\flat)^{1/p^e})^{N_r}}  
}\]
We briefly explain why such a diagram exists. The first vertical map is the identity, the second vertical map sends $Y_i$ to $a_i/{(g^\flat)^{1/p^e}}$, one can check that this is well defined. The other vertical maps exist because $B^\flat$ is a big Cohen-Macaulay algebra with respect to $p^\flat, x_1^\flat,\dots,x_{d-1}^\flat$ (so any relation will be trivialized in $B^\flat$) and hence we can define these maps as in \cite[Proof of Lemma 5.1]{HochsterBigCohen-Macaulayalgebrasindimensionthree}. The key point is that, the numbers $N_1,\dots,N_r$ depend only on $N$ and \autoref{equation--partial algebra modification} but {\it not} on $e$ by \cite[Lemma 5.1]{HochsterBigCohen-Macaulayalgebrasindimensionthree}.

Tracing the image of $1\in B^\flat$ in the above commutative diagram in two different ways, we find that $1\in B^\flat\cdot\frac{1}{((g^\flat)^{1/p^e})^{N_r}}$ is inside $(p^\flat,x_1^\flat,\dots,x_{d-1}^\flat) B^\flat\cdot \frac{1}{((g^\flat)^{1/p^e})^{N_r}}$. This means $((g^\flat)^{1/p^e})^{N_r}\in(p^\flat,x_1^\flat,\dots,x_{d-1}^\flat) B^\flat$ for all $e$. Since $N_r$ does not depend on $e$, this implies $$(g^\flat)^{1/p^e}\in (p^\flat,x_1^\flat,\dots,x_{d-1}^\flat) B^\flat$$ for all $e$. Since $g^\flat$ can be expressed as $\langle g, g^{1/p},g^{1/p^2},\dots\rangle$, $(g^\flat)^{1/p^e}=(g^{1/p^e})^\flat$ and this is equivalent to saying that
$$g^{1/p^e}\in (p, x_1,\dots,x_{d-1})B$$
for all $e$ since $B/p=B^\flat/p^\flat$. Since $B$ is a big Cohen-Macaulay $R^+$-algebra, it is a solid $S$-algebra for every complete local domain $S$ that is module-finite over $R$ by \cite[Corollary 2.4]{HochsterSolidClosure}. Pick such an $S$ such that $g\in S$, we have $$g\in (p,x_1,\dots,x_{d-1})^{p^e}B\cap S\subseteq (\m^{p^e}S)^\bigstar$$ for every $e$, where $I^\bigstar$ denotes the solid closure, which is always contained inside the integral closure by \cite[Theorem 5.10]{HochsterSolidClosure}. But then we have $$0\neq g\in \cap_e \overline{\m^{p^e}S}=0$$ which is a contradiction.

At this point we have proved there exists a big Cohen-Macaulay $B^\flat$-algebra $C$ such that $z^\flat\in (p^\flat, x_1^\flat,\dots,x_{d-1}^\flat)C$. Since we are in characteristic $p$, we can further assume that $C$ is perfect and $p^\flat$-adically complete by \cite[Proposition 3.7]{DietzBCMSeeds}. Untilt, we have $B=(B^\flat)^\sharp\to C^\sharp$ such that
$$z\in (p, x_1,\dots,x_{d-1}) C^\sharp$$
because $C^\sharp/p=C/p^\flat$. Now $B'=\widehat{C^\sharp}$, the $\m$-adic completion of $C^\sharp$, is an integral perfectoid (balanced) big Cohen-Macaulay $B$-algebra by \cite[Proposition 2.2.1]{AndreWeaklyFunctorialBigCM} and \cite[Corollary 8.5.3]{BrunsHerzog}, and we have $z\in (p, x_1,\dots,x_{d-1})B'$, i.e., $\eta=0$ in $H_\m^d(B')$. This finishes the proof when $p\eta=0$.

Finally let us prove the general case. Suppose $p^t\eta=0$ for some $t\geq 1$. The $t=1$ case we just proved implies that there exists an integral perfectoid big Cohen-Macaulay $B$-algebra $B_1$ such that $p^{t-1}\eta=0$ in $H_\m^d(B_1)$ (and we still have $g^{1/p^\infty}\eta=0$ in $H_\m^d(B_1)$). We can repeat the above process to find an integral perfectoid big Cohen-Macaulay $B_1$-algebra $B_2$ such that $p^{t-2}\eta=0$ in $H_\m^d(B_2)$, keep going and eventually we get an integral perfectoid big Cohen-Macaulay $B$-algebra $B_t:=B'$ such that $\eta=0$ in $H_\m^d(B')$.
\end{proof}

\begin{proposition}
\label{prop.tau=tauB}
Let $(R,\m)$ be a complete local domain of mixed characteristic $(0,p)$. Then there exists an integral perfectoid big Cohen-Macaulay $R^+$-algebra $B$ such that $$0^{B}_{H_\m^d(R)}=\widetilde{0}^{B}_{H_\m^d(R)}=0^\scr{B}_{H_\m^d(R)}=\widetilde{0}^\scr{B}_{H_\m^d(R)}.$$ As a consequence, $\mytau_\scr{B}(\omega_R)=\mytau_B(\omega_R)$ (it follows that the same is true for all $B'$ such that $B\to B'$). 
\end{proposition}
\begin{proof}
Since every integral perfectoid big Cohen-Macaulay $R$-algebra $B$ can be mapped to an integral perfectoid big Cohen-Macaulay $R^+$-algebra by \autoref{theorem: AndreShimomoto}, we have $0^{\scr{B}}_{H_\m^d(R)}\subseteq \widetilde{0}^\scr{B}_{H_\m^d(R)}$. Therefore it is enough to find $B$ such that $\widetilde{0}^\scr{B}_{H_\m^d(R)}\subseteq 0^{B}_{H_\m^d(R)}$. Now by \autoref{lem.CanonicalBconstruction} and the axiom of global choice (or if we restrict ourselves to a fixed Grothendieck universe), for every $\eta\in \widetilde{0}^\scr{B}_{H_\m^d(R)}$ we can choose an integral perfectoid big Cohen-Macaulay $R^+$-algebra $B_\eta$ such that $\eta\in 0^{B_{\eta}}_{H_\m^d(R)}$. Since $\{B_\eta\}_{\eta\in \widetilde{0}^\scr{B}_{H_\m^d(R)}}$ is a set of integral perfectoid big Cohen-Macaulay $R^+$-algebras, by \autoref{thm.DominateAnySetOfBCM} there exists an integral perfectoid big Cohen-Macaulay $R^+$-algebra $B$ such that $B_\eta\to B$ for all $\eta$. This implies $\eta\in 0^B_{H_\m^d(R)}$ for all $\eta\in \widetilde{0}^\scr{B}_{H_\m^d(R)}$ and so $\widetilde{0}^\scr{B}_{H_\m^d(R)}\subseteq 0^B_{H_\m^d(R)}$ as desired. 
\end{proof}

We have the following transformation rule of the {\BCM} parameter test submodules under finite maps.

\begin{lemma}
\label{lem.FiniteMapsTransformationForParamTestIdeals}
Suppose that $(R, \fram) \subseteq (S, \frn)$ is a finite extension of complete local domains and that $B$ is a big Cohen-Macaulay $S$-algebra (and hence also a big Cohen-Macaulay $R$-algebra). Then $\Tr(\mytau_B(\omega_S)) = \mytau_B(\omega_R)$ (here $\Tr$ is the Grothendieck trace which coincides with the field trace of $R \subseteq S$ if the map is generically separable). 
\end{lemma}
\begin{proof}
Observe that we have a factorization:
\[
H^d_{\fram}(R) \xrightarrow{\phi} H^d_{\n}(S) \to H^d_{\n}(B)
\]
and that $\phi$ is the Matlis dual of $\Tr : \omega_S \to \omega_R$.  We have that
\[
H^d_{\fram}(R) / 0^B_{H_\m^d(R)} \hookrightarrow H^d_{\n}(S) / 0^B_{H_\n^d(S)}
\]
is injective. The result follows by Matlis duality. 
\end{proof}

Now we prove our restriction type theorem for the {\BCM} test submodule.

\begin{theorem}
\label{theorem: big CM restriction}
Let $(R,\m)$ be a complete local domain of dimension $d$ and let $B$ be a big Cohen-Macaulay $R$-algebra. Then for every nonzerodivisor $x\in R$, the image of $\mytau_B(\omega_R)$ under the natural map $\omega_R\to \omega_R/x\omega_R\to\omega_{R/xR}$ equals $\mytau_{B/xB}(\omega_{R/xR})$.

In particular, if $C$ is any big Cohen-Macaulay $B/xB$-algebra, then we have that the image of $\mytau_B(\omega_R)$ under the natural map $\omega_R\to \omega_R/x\omega_R\to\omega_{R/xR}$ contains $\mytau_{C}(\omega_{R/xR})$.
\end{theorem}
\begin{proof}
We consider the commutative diagram
\[
\xymatrix{
0 \ar[r] & R \ar[r]^{\cdot x} \ar[d] & R \ar[d] \ar[r] & R/xR \ar[r]\ar[d] & 0\\
0\ar[r] & B\ar[r]^{\cdot x} & B\ar[r] & B/xB \ar[r] & 0
}
\]
which induces
\[
\xymatrix{
{} & H_\m^{d-1}(R/xR) \ar[r] \ar[d]^\beta & H_\m^d(R) \ar[d]^{\alpha} \ar[r]^{\cdot x} & H_\m^d(R) \ar[r]\ar[d] & 0\\
0\ar[r] & H_\m^{d-1}(B/xB) \ar[r] & H_\m^d(B) \ar[r]^{\cdot x} & H_\m^d(B) \ar[r] & 0.
}
\]
Chasing this diagram we know that $\im(\beta) \hookrightarrow \im(\alpha)$.  But the Matlis dual of $\im(\beta)$ is $\mytau_{B/xB}(\omega_{R/xR})$ and the Matlis dual of $\im(\alpha)$ is $\mytau_B(\omega_R)$.  The result follows.
\end{proof}

\begin{corollary}
\label{cor.big CM restriction to char p}
Let $R$ be a complete local domain of mixed characteristic $(0,p)$. If $0\neq x\in R$ is such that $R/xR$ is reduced and of characteristic $p$, then for every integral perfectoid big Cohen-Macaulay $R^+$-algebra $B$, we have $\mytau_{B/xB}(\omega_{R/xR})\supseteq \tau(\omega_{R/xR})$.\footnote{Here $R/xR$ is complete but not necessarily a domain, we can still define $\tau(\omega_{R/xR})$ as in \cite{SmithTestIdeals}, which is the Matlis dual of $0^*_{H_\m^{d-1}(R/xR)}$.}
\end{corollary}
\begin{proof}
By \autoref{theorem: Andre}, for every minimal prime $Q$ of $(x)$, the surjective ring map $R\to R/xR \to  R/Q$ fits into a commutative diagram:
\[\xymatrix{
R \ar[r] \ar[d]& R/xR\ar[r] \ar[d] & R/Q \ar[d] \\
B\ar[r]& B/xB\ar[r] & B(Q)
}\]
where $B(Q)$ is a big Cohen-Macaulay $(R/Q)^+$-algebra. This induces a commutative diagram:
\[\xymatrix{
H_\m^{d-1}(R/xR)\ar[r] \ar[d] & H_\m^{d-1}(R/Q) \ar[d] \\
H_\m^{d-1}(B/xB)\ar[r] & H_\m^{d-1}(B(Q))
}\]
Chasing this diagram, it is easy to see that the image of $0^{B/xB}_{H_\m^{d-1}(R/xR)}$ in $H_\m^{d-1}(R/Q)$  is contained in $0^{B(Q)}_{H_\m^{d-1}(R/Q)}=0^*_{H_\m^{d-1}(R/Q)}$ by \autoref{prop: big CM test ideal char p}. Since this is true for every minimal prime $Q$ of $R/xR$ and $R/xR$ is reduced, it follows that $0^{B/xB}_{H_\m^{d-1}(R/xR)}\subseteq 0^*_{H_\m^{d-1}(R/xR)}$ and hence $\mytau_{B/xB}(\omega_{R/xR})\supseteq \tau(\omega_{R/xR})$ by applying Matlis duality.
\end{proof}

We now compare our parameter test module with the Grauert-Riemenschneider multiplier submodule producing a generalization of \autoref{prop: big CM rational implies pseudo-rational}. Let $(R,\m)$ be an excellent local domain and let $\pi$: $X\to \Spec R$ be a projective birational map with $E$ the pre-image of $\{\m\}$. Let $K$ be the kernel of the natural map $H_\m^d(R)\to H_E^d(X, \cO_X) = \myH^d(\myR \Gamma_m(\myR \pi_* \cO_X))$. Then by local and Grothendieck duality we know that $\Ann_{\omega_R}K$ can be identified with $\pi_*\omega_X$, see \autoref{defProp.Char0}.

\begin{proposition}
\label{prop: big CM test ideal birational}
Let $(R,\m)$ be a complete local domain of dimension $d$, and let $\pi$: $X\to \Spec R$ be a proper birational map. Then there exists a big Cohen-Macaulay $R^+$-algebra $C$ (that can be assumed to be integral perfectoid in mixed characteristic) such that $\pi_*\omega_X\supseteq \mytau_C(\omega_R)$. In particular, in mixed characteristic we have $\mytau_\scr{B}(\omega_R)\subseteq \pi_*\omega_X$ for all such $\pi$.
\end{proposition}
\begin{proof}
Using Chow's lemma, it is harmless to assume that $\pi$ is projective.  Since $X\to \Spec R$ is projective and birational, $X=\Proj R[Jt]$ for some ideal $J\subseteq R$. We let $S=R[Jt]$, which is an excellent local domain. Let $\n\subseteq S$ denote the maximal ideal $\m+Jt$. Let
\[
K = \ker \big( H_\m^d(R)\to H_E^d(X, \cO_X) \big)
\]
with $E$ the pre-image of $\{\m\}$. By the Sancho-de-Salas exact sequence \cite{SanchodeSalasBlowingupmorphismswithCohenMacaulayassociatedgradedrings}, we have
\[
[H_{\n}^d(S)]_0\to H_\m^d(R)\to H_E^d(X, \cO_X).
\]
This sequence implies that
\begin{equation}
\label{equation: K in the image}
K \subseteq \Image\big(H_{\n}^d(S)\to H_\m^d(R)\big).
\end{equation}

Next we consider the surjective map $\widehat{S_\n}\to R$. Since $R$ is a domain this map factors through $\widehat{S_\n}/P$ for some minimal prime $P$ of $\widehat{S_\n}$. We note that $\dim(\widehat{S_\n}/P)=\dim R+1$ and hence $R$ is obtained from $\widehat{S_\n}/P$ by killing a height one prime $Q$. We now apply \cite[Theorem 3.9]{HochsterHunekeApplicationsofBigCM} in equal characteristic, and apply \autoref{theorem: Andre} in mixed characteristic to obtain a commutative diagram:
\[\xymatrix{
\widehat{S_\n}/P \ar[r] \ar[d] & R \ar[d] \\
B \ar[r] & C
}\]
where $B$, $C$ are big Cohen-Macaulay algebras over $(\widehat{S_\n}/P)^+$ and $R^+$ respectively (and in mixed characteristic $C$ is also integral perfectoid). This induces a commutative diagram of local cohomology:
\[\xymatrix{
H_\n^d(S)\cong H_\n^d(\widehat{S_\n}) \ar[r] & H_\n^d(\widehat{S_\n}/P) \ar[r] \ar[d] & H_\m^d(R) \ar[d] \\
{} & 0=H_\n^d(B) \ar[r] & H_\m^d(C)
}\]
By \autoref{equation: K in the image}, $K$ is in the image of $H_{\n}^d(S)\to H_\m^d(R)$. Thus chasing the diagram we find that $K\subseteq 0^C_{H_\m^d(R)}$. By local duality we find that $$\pi_*\omega_X=\Ann_{\omega_R}K\supseteq \mytau_C(\omega_R).$$ This finishes the proof.
\end{proof}

Combining the earlier results in the section, we obtain the following.
\begin{theorem}
\label{theorem: main}
Let $(R,\m)$ be an excellent analytically irreducible local domain of mixed characteristic $(0,p)$. Suppose $x\in R$ is a nonzerodivisor such that $R/xR$ is reduced and of characteristic $p$. Let $\pi$: $X\to \Spec R$ be a projective birational map. Then the image of $\pi_*\omega_X$ in $\omega_{R/xR}$ under the natural map $\pi_*\omega_X\to \omega_R\to\omega_R/x\omega_R\to\omega_{R/xR}$ contains $\tau(\omega_{R/xR})$. In particular, if $R/xR$ is $F$-rational, then $R$ is pseudo-rational.
\end{theorem}
\begin{proof}
First of all we will replace $R$ by its completion $\widehat{R}$ and replace $X$ by $X\times _{\Spec R}\Spec \widehat{R}$. Since $R$ is excellent and $R/xR$ is reduced, $\widehat{R}/x\widehat{R}$ is still reduced. It is clear that under this flat base change $\pi_*\omega_X$ and $\tau(\omega_{R/xR})$ will be changed to $\pi_*\omega_X\otimes\widehat{R}$ and $\tau(\omega_{R/xR})\otimes \widehat{R}$. Therefore without loss of generality, we can assume $R$ is a complete local domain.

Applying \autoref{prop: big CM test ideal birational}, we know that there exists an integral perfectoid big Cohen-Macaulay $R^+$-algebra $B$ such that $\pi_*\omega_X\supseteq \mytau_{B}(\omega_R)$. By \autoref{theorem: big CM restriction}, we know that the image of $\mytau_{B}(\omega_R)$ under the natural map $\mytau_{B}(\omega_R)\to \omega_R\to\omega_R/x\omega_R\to\omega_{R/xR}$ contains $\mytau_{B/xB}(\omega_{R/xR})$. But then by \autoref{cor.big CM restriction to char p}, we know that $\mytau_{B/xB}(\omega_{R/xR})\supseteq \tau(\omega_{R/xR})$.

The last assertion is already proved in \autoref{theorem: F-ratPseudorat}, but we pointed out that it follows directly from the more general statement. So suppose $R/xR$ is $F$-rational, then $R/xR$ is normal and Cohen-Macaulay. We know that $R$ is normal and Cohen-Macaulay (because these two properties deform). Therefore to show $R$ is pseudo-rational, it suffices to show that $\pi_*\omega_X=\omega_R$ for every projective birational map $\pi$: $X\to\Spec R$. But we already know that the image of $\pi_*\omega_X$ in $\omega_R/x\omega_R=\omega_{R/xR}$ contains $\tau(\omega_{R/xR})=\omega_{R/xR}$ by $F$-rationality. Thus $\pi_*\omega_X=\omega_R$ as desired.
\end{proof}

\subsection{Big Cohen-Macaulay parameter test submodules and the singular locus}

In characteristic $p>0$, the parameter test submodule $\tau(\omega_R)$ basically measures how far $R$ is from being $F$-rational. In particular, for every $h$ such that $R_h$ is regular, $\tau(\omega_R)_h=\tau(\omega_{R_h})=R_h$ and hence a fixed power of $h$, $h^N$, multiplies $\omega_R$ into the parameter test submodule. It is natural to ask whether analogous results hold in mixed characteristic. In this subsection we will partially answer this question.

\begin{theorem}[Uniform annihilation]
\label{theorem: uniform annihilation}
Let $(A,\m_A)\to (R,\m)$ be a module-finite extension such that $A$ is a complete regular local ring of mixed characteristic $(0,p)$ and $R$ is a complete local domain. Suppose $h\in A$ is such that $A_h\to R_h$ is finite \'{e}tale. Then there exists an integer $N$ such that $h^N 0^{\scr{B}}_{H_\m^d(R)}=0$, or dually $h^N \omega_R \subseteq \mytau_\scr{B}(\omega_R)$.
\end{theorem}

To prove this theorem we need a simple lemma on Galois theory of rings.
\begin{lemma}
\label{lemma: Galois}
Let $A\to R$ be a finite extension of Noetherian normal domains. Suppose the induced map on the fraction field $L_A\to L_R$ is Galois with Galois group $\scr{G}$. If $I\subseteq R$ is a $\scr{G}$-invariant ideal of $R$ such that $\Tr_{L_R/L_A}(I)=A$, then $I=R$.
\end{lemma}
\begin{proof}
Since $L_A\to L_R$ is Galois, we know that $$\Tr_{L_R/L_A}(x)=\sum_{\sigma\in\scr{G}}\sigma(x)$$ for every $x\in L_R$. Since $A$ is normal, we know that $\Tr_{L_R/L_A}(R)\subseteq A$. Now suppose $z\in I$, since $I$ is $\scr{G}$-invariant it follows directly from the above formula that $\Tr_{L_R/L_A}(z)\in I$. Therefore $$\Tr_{L_R/L_A}(I)\subseteq A\cap I,$$ thus $\Tr_{L_R/L_A}(I)=A$ implies $I=R$.
\end{proof}

\begin{proof}[Proof of \autoref{theorem: uniform annihilation}]
  We pick a (sufficiently large) integral perfectoid big Cohen-Macaulay $R^+$-algebra $B$ that satisfies the conclusion of \autoref{prop.tau=tauB}.

First we show that we can reduce to the case that $R$ is normal.  Indeed, let $R^{\textnormal{N}}$ be the normalization of $R$, then $A_h \to R^{\textnormal{N}}_h$ is still \'etale and $R_h \to R^{\textnormal{N}}_h$ is an isomorphism.  Furthermore, by \autoref{lem.FiniteMapsTransformationForParamTestIdeals}, $\Tr(\tau_B(\omega_{R^{\textnormal{N}}})) = \tau_B(\omega_R)$ where in this case $\Tr : \omega_{R^{\textnormal{N}}} \hookrightarrow \omega_R$ is an inclusion which becomes an isomorphism after inverting $h$.  It is then clear we may assume that $R$ normal.

 It is enough to prove that there exists $N$ such that $h^N0^B_{H_\m^d(R)}=0$.

Let $L_{A}$ and $L_{R}$ be the fraction fields of $A$ and ${R}$ respectively. We let $L'$ be the Galois closure of $L_R$ over $L_A$ inside an algebraic closure $\overline{L}_A$ of $L_{A}$. We write $\scr{G}=\text{Gal}(\overline{L}_A/L_A)$, and write $\scr{G}'=\text{Gal}(L'/L_A)$. Note that $\scr{G}'=\scr{G}/\scr{H}$ for some normal subgroup $\scr{H}$ of $\scr{G}$. We further let $R'$ be the integral closure of ${R}$ in $L'$. Then $A$ is the ring of invariants of $R'$ under $\scr{G}'$ (and is the ring of invariants of $A^+=R^+$ under $\scr{G}$), and $R'$ is the ring of invariants of $A^+=R^+$ under $\scr{H}$. More importantly, since $A_h\to {R}_h$ is finite \'{e}tale, $A_h\to R'_h$ is also finite \'{e}tale. 
Since the $\scr{G}'$-action on $R'$ induces a $\scr{G}'$-action on $H_\m^d(R')$, we set
\[
W=\sum_{\sigma'\in\scr{G}'}\sigma'(0^B_{H_\m^d(R')}).
\]

\begin{claim}
\label{clm.ImageOfTildeWInAisZero}
We have $H_\m^d(A)\cap W=0$ in $H_\m^d(R')$. Consequently, the composite map
\begin{equation}
\label{equation: trace surjection}
\Ann_{\omega_{R'}}W\hookrightarrow \omega_{R'}\xrightarrow{\Tr} \omega_A
\end{equation}
is a surjection.
\end{claim}
\begin{proof}[Proof of Claim]
For every $\sigma'\in \scr{G}'$, we pick a lift $\sigma\in\scr{G}$ of $\sigma'$. Let $B_\sigma$ denote the integral perfectoid big Cohen-Macaulay $R^+$-algebra $B$ with its $R^+$-algebra structure coming from the composition $R^+\xrightarrow{\sigma} R^+\to B$. Since every $\eta\in H_\m^d(R')$ is fixed by $\scr{H}$, we have $\sigma'(\eta)=\sigma(\eta)$ as elements in $H_\m^d(R^+)$ and hence their images in $H_\m^d(B)$ are also the same. It follows that
\begin{equation}
\label{eq.SigmaW''EqualsKilledInBsigma}
\sigma'(0^B_{H_\m^d(R')})=0^{B_{\sigma^{-1}}}_{H_\m^d(R')}.
\end{equation}
Now applying \autoref{thm.DominateAnySetOfBCM}, we can find an integral perfectoid big Cohen-Macaulay $R^+$-algebra $C$ such that there exists an $R^+$-linear map $B_\sigma\to C$ for every $\sigma$. Since $A$ is regular and $C$ is a big Cohen-Macaulay algebra, $C$ is faithfully flat over $A$. In particular, we know that
\begin{equation}
\label{equation: A is pure inside}
0^C_{H_\m^d(A)}=\{\eta\in H_\m^d(A)\;\mid\;\eta=0 \text{ in }H_\m^d(C)\}=0.
\end{equation}
But since $B_\sigma$ maps to $C$ for every $\sigma$ (which is a fixed lift of $\sigma'$), it follows from \autoref{eq.SigmaW''EqualsKilledInBsigma} that
\[
\sigma'(0^B_{H_\m^d(R')})=0^{B_{\sigma^{-1}}}_{H_\m^d(R')}\subseteq 0^C_{H_\m^d(R')}
\]
for every $\sigma'$ and thus $W=\sum_{\sigma'\in\scr{G}'}\sigma'(0^B_{H_\m^d(R')})\subseteq 0^C_{H_\m^d(R')}$. Therefore it must intersect with $H_\m^d(A)$ trivially by \autoref{equation: A is pure inside}. This completes the proof of \autoref{clm.ImageOfTildeWInAisZero} by Matlis duality.
\end{proof}

Since $\omega_{R'}$ can be canonically viewed as $\Hom_A(R', A)$ by \cite[Theorem 3.3.7]{BrunsHerzog}, the $\scr{G}'$-action on $R'$ induces a natural $\scr{G}'$-action on $\omega_{R'}$. Since $W$ is $\scr{G}'$-invariant by construction, $\Ann_{\omega_{R'}}W\subseteq \omega_{R'}$ is a $\scr{G}'$-submodule. Now localizing \autoref{equation: trace surjection} at $h$, since $A_h\to R'_h$ is finite \'{e}tale, we can identify $(\omega_A)_h$ and $(\omega_{R'})_h$ with $A_h$ and $R'_h$ respectively. Thus we have:
\[\xymatrix{
(\Ann_{\omega_{R'}}W)_h \ar@{^{(}->}[r]\ar[d]^\cong & (\omega_{R'})_h \ar[r]^{\Tr} \ar[d]^\cong &  (\omega_A)_h \ar[d]^\cong \\
I \ar@{^{(}->}[r] & R'_h \ar[r]^{\Tr} & A_h
}
\]
where we identify $(\Ann_{\omega_{R'}}W)_h$ as a $\scr{G}'$-invariant ideal $I\subseteq R'_h$. Since the composite map $I\to R'_h\to A_h$ is surjective by \autoref{equation: trace surjection}, \autoref{lemma: Galois} (since the map is finite \'{e}tale after inverting $h$, the trace map is the same as the field trace map up to multiplication by a unit) yields $I=R'_h$, i.e., $(\Ann_{\omega_{R'}}W)_h=(\omega_{R'})_h$. This implies $(W^\vee)_h \cong (\omega_{R'}/\Ann_{\omega_{R'}}W)_h=0$. From this we know that the finitely generated module $W^\vee$ is annihilated by a power of $h$, thus so is $W$. Therefore $0^B_{H_\m^d(R')}\subseteq W$ is also annihilated by a power of $h$, and thus so is $(0^B_{H_\m^d(R')})^\vee$. This implies $(0^B_{H_\m^d(R')})^\vee_h=0$. Finally, we consider the commutative diagram with exact rows:
\[\xymatrix{
0 \ar[r] & 0^B_{H_\m^d(R)} \ar[r] \ar[d] & H_\m^d({R}) \ar[r]\ar[d] & H_\m^d(B) \ar[d]^= \\
0 \ar[r] & 0^B_{H_\m^d(R')}\ar[r] & H_\m^d(R') \ar[r] & H_\m^d(B)
}
\]
Taking the Matlis dual and localizing at $h$, the above diagram induces:
\[\xymatrix{
(0^B_{H_\m^d(R)})^\vee_h  & (\omega_{{R}})_h\cong\omega_{{R}_h}\ar@{->>}[l]  & H_\m^d(B)^\vee_h  \ar[l]\\
0=(0^B_{H_\m^d(R')})^\vee_h\ar[u] & (\omega_{R'})_h \cong\omega_{R'_h}\ar@{->>}[u]\ar@{->>}[l] & H_\m^d(B)^\vee_h \ar[u]^=\ar[l]
}
\]
where the middle surjectivity is because ${R}_h$ is regular (since it is finite \'{e}tale over $A_h$), hence $R_h\to R'_h$ splits.  Chasing this diagram it is easy to see that $(0^B_{H_\m^d(R)})^\vee_h=0$, thus the finitely generated module $(0^B_{H_\m^d(R)})^\vee$ is annihilated by $h^N$ for some $N\gg0$. Therefore $0^B_{H_\m^d(R)}$ is also annihilated by $h^N$.
\end{proof}

We end this section by proving that the ideal generated by all those $h$ satisfying \autoref{theorem: uniform annihilation} contains a power of the defining ideal of the singular locus of $R/p$. This gives us a weak analog of \cite[Theorem 3.4]{HochsterHunekeTightClosureAndStrongFRegularity}. We need a slight generalization of Flenner's local Bertini theorem \cite{FlennerLocalBertini}, which we expect is known to experts.

\begin{lemma}[Flenner]
\label{lem.FlennerBertini}
Let $(R,\m,k)$ be a complete local ring of equal characteristic with $k$ infinite. Let $I,I_0,\dots,I_m$ be ideals of $R$ such that $I\nsubseteq I_j$ for all $0\leq j\leq m$. Then there exists $x\in I$ such that
\begin{enumerate}
  \item $x\notin P^{(2)}$ for all $P\in\Spec R-V(I)$, and
  \item $x\notin I_j$ for all $0\leq j\leq m$.
\end{enumerate}
\end{lemma}
\begin{proof}
By Cohen's structure theorem, we can write $R=k\llbracket z_1,\dots, z_t\rrbracket/J$. Let $\widehat{\Omega}_{R/k}$ be the complete module of differentials of $R$ over $k$ (so it is generated by $dz_1,\dots,dz_t$). Fix a generating set $\{f_1,\dots,f_s\}$ of $I$ and consider the following set:
\[
\{g_1,\dots,g_n\}:= \{f_1,f_1z_1,\dots,f_1z_t, f_2, f_2z_1,\dots,f_2z_t,\dots,f_s,f_sz_1,\dots,f_sz_t\}.
\]
Then clearly $g_1,\dots,g_n$ is still a generating set of $I$. Since the submodule of $\widehat{\Omega}_{R/k}$ generated by $dg_1,\dots,dg_n$ contains $df_i$ and $d(f_iz_j)=z_jdf_i+f_idz_j$ for all $i,j$, it contains $f_idz_j$ for all $i,j$. If we localize at $P\in \Spec R-V(I)$, then at least one of the $f_i$ becomes a unit and thus $dg_1,\dots,dg_n$ generate $(\widehat{\Omega}_{R/k})_P$ for all $P\in \Spec R-V(I)$. Now we apply \cite[Satz 1.7]{FlennerLocalBertini} to $M=\widehat{\Omega}_{R/k}$ and $\{m_1,\dots,m_n\}=\{dg_1,\dots,dg_n\}$\footnote{Note that the equation in \cite[Satz 1.7]{FlennerLocalBertini} is satisfied by \cite[Lemma 2.6]{FlennerLocalBertini} and that $dg_1,\dots,dg_n$ generate $(\widehat{\Omega}_{R/k})_P$.}, we know that there exists a nonzero polynomial $G\in k[x_1,\dots,x_n]$ such that if $G(\lambda_1,\dots,\lambda_n)\neq 0$, then $d(\lambda_1g_1+\cdots+\lambda_ng_n)=\lambda_1dg_1+\cdots+\lambda_ndg_n$ is part of a minimal set of generators for $(\widehat{\Omega}_{R/k})_P$ for all $P\in \Spec R-V(I)$.

Since $k$ is infinite and $I\nsubseteq I_j$ for all $j$, a general linear combination $x=\lambda_1g_1+\cdots+\lambda_ng_n\in I$ (where $\lambda_j\in k$) is not contained in $I_0,\dots,I_m$. But a general choice of $\lambda_1,\dots,\lambda_n$ also satisfies $G(\lambda_1,\dots,\lambda_n)\neq 0$ and thus $dx=d(\lambda_1g_1+\cdots+\lambda_ng_n)$ is part of a minimal set of generators for $(\widehat{\Omega}_{R/k})_P$ by the above discussion. Therefore by \cite[Lemma 2.2]{FlennerLocalBertini}, this implies $x$ is not contained in $P^{(2)}$ for all $P\in\Spec R-V(I)$.
\end{proof}

\begin{lemma}
\label{lem.cutR_0}
Let $(R,\m,k)$ be a complete local ring of equal characteristic with $k$ infinite. Suppose $R$ is equidimensional and regular in codimension $0$. Let $Q\in \Spec R$ be a height $h$ prime such that $R_Q$ is regular. Then we can find a sequence of parameters $x_1,\dots,x_h$ in $Q$ such that $R/(x_1,\dots,x_h)$ is regular in codimension $0$.
\end{lemma}
\begin{proof}
We prove by induction on $h$. The case $h=0$ is trivial. So we assume $h\geq 1$ and let $J$ be the defining ideal of the singular locus of $R$, which has height at least one by assumption. Let $P_1,\dots,P_m$ be the height one minimal primes of $J$ together with the minimal primes of $R$ (if $J$ has height at least two then we only consider minimal primes of $R$). Note that none of the $P_i$ contains $Q$ since $Q$ has height at least one and $R_Q$ is regular. We apply \autoref{lem.FlennerBertini} with $I=Q$, $I_0=Q^{(2)}$, $I_j=P_j$ to find the element $x_1=x\in Q$. Then $R/(x)$ is equidimensional and $(R/(x))_Q$ is regular as $x\notin Q^{(2)}$. To complete the induction we only need to show $R/(x)$ is still regular in codimension $0$. Let $Q_1,\dots,Q_n$ be (height one) minimal primes of $(x)$ in $R$. If $R_{Q_i}$ is not regular, then $Q_i$ contains both $J$ and $(x)$. But then $Q_i$ contains $J+(x)$ which has height at least two as $x$ is not in any minimal prime of $J$, which is a contradiction. Thus $R_{Q_i}$ is regular for all $i$. But since $x$ is not in $Q_i^{(2)}$ by construction, $(R/(x))_{Q_i}$ is regular and hence $R/(x)$ is regular in codimension $0$.
\end{proof}

Combining the previous lemma with the Cohen-Gabber theorem, we have the following.

\begin{lemma}
\label{lemma: Cohen-Gabber}
Let $(R,\m,k)$ be a complete local domain of mixed characteristic $(0,p)$ with $k$ infinite, such that $R/p$ is regular in codimension $0$. Suppose $Q\in \Spec R$ such that $p\in Q$ and $(R/p)_Q$ is regular. Then there exists a complete and unramified regular local ring $A$ and an element $h\in A$, $h \notin Q$, such that $A\to R$ is a module-finite extension and that $A_h\to R_h$ is finite \'{e}tale.
\end{lemma}
\begin{proof}
We apply \autoref{lem.cutR_0} to $R/p$ to find a parameter sequence $x_1,\dots,x_s$ such that $Q$ is a minimal prime of $(p,x_1,\dots,x_s)$ and $R/(p,x_1,\dots,x_s)$ is regular in codimension $0$. In other words, $p,x_1,\dots,x_s$ is a regular system of parameters in $R_{Q_i}$ for all minimal primes $Q=Q_1,Q_2,\dots,Q_n$ of $(p,x_1,\dots,x_s)$.

By the Cohen-Gabber theorem (see \cite[Theorem 1.1]{KuranoShimomotoCohenGabber}), there exists a choice of coefficient field $k$ of $R/p$ and $\overline{y}_1,\dots,\overline{y}_{t}\in R/(p,x_1,\dots,x_s)$ such that $k\llbracket \overline{y}_1,\dots,\overline{y}_{t}\rrbracket \to R/(p,x_1,\dots,x_s)$ is module-finite and $\text{Frac}(k\llbracket \overline{y}_1,\dots,\overline{y}_{t}\rrbracket)\to \text{Frac}(R/Q_i)$ is finite \'{e}tale for all $i$. We pick lifts $y_1,\dots,y_{t}$ in $R$ such that $p, x_1,\dots,x_s, y_1,\dots,y_t$ is a system of parameters on $R$.

By Cohen's structure theorem we have $A=C_k\llbracket x_1,\dots,x_s, y_1,\dots, y_{t}\rrbracket \to R$ is a module-finite extension. We claim that this $A\to R$ is the desired extension where we set $h$ to be the discriminant of the map. To see this, we note that the primes of $R$ that lie over the prime $(p,x_1,\dots,x_s)\subseteq A$ are precisely $Q=Q_1,Q_2,\dots,Q_n$. We claim that the map $A_{(p,x_1,\dots,x_s)}\to R_{Q_i}$ is essentially \'{e}tale for all $i$. This is because the map is flat, $(p, x_1,\dots,x_s)R_{Q_i}=Q_iR_{Q_i}$ (since $p,x_1,\dots,x_s$ is a regular system of parameters in $R_{Q_i}$), and the residue field extension is finite separable: $A_{(p,x_1,\dots,x_s)}/(p,x_1,\dots,x_s)A_{(p,x_1,\dots,x_s)}$ is the fraction field of $k\llbracket \overline{y}_1,\dots,\overline{y}_{t}\rrbracket $ while $R_{Q_i}/Q_iR_{Q_i}$ is the fraction field of $R/Q_i$, so this follows from our construction. Since the map $A_{(p,x_1,\dots,x_s)}\to R_{Q_i}$ is essentially \'{e}tale for every $i$, the discriminant $h\in A$ is not contained in $(p,x_1,\dots,x_s)$, and thus not in $Q$.
\end{proof}

\begin{remark}
It should be pointed out that the conclusion of \autoref{lemma: Cohen-Gabber} is not true if $R/p$ is not regular in codimension $0$. Consider $R=C_k\llbracket x,y\rrbracket/(p-xy^2)$ and $Q=(x)$, we have $p\in Q$ and $(R/p)_Q$ is a field. But there does not exist a complete and unramified regular local ring $A$ with $A\to R$ module-finite with an element $h\in A-Q$ such that $A_h\to R_h$ is finite \'etale. If such $A$ exists, then as $(y)\cap A=(p)$ and $h\notin (p)$, $A_{(p)}\to R_{(y)}$ is unramified, which is a contradiction since $p\in (y^2)$ in $R$.
\end{remark}

\begin{theorem}
\label{thm.GenSingularlocus}
Let $(R,\m,k)$ be a complete local domain of mixed characteristic $(0,p)$, such that $R/p$ is reduced.  Then there exists an ideal $J\subseteq R$ such that $\sqrt{J(R/p)}$ defines the singular locus of $R/p$, and such that $J$ multiplies $\omega_R$ into $\mytau_\scr{B}(\omega_R)$.
\end{theorem}

We begin with a proof with $k$ infinite.

\begin{proof}[Proof of \autoref{thm.GenSingularlocus} with infinite residue field]
By \autoref{lemma: Cohen-Gabber}, for every $Q\in \Spec R$ such that $p\in Q$ and $(R/p)_Q$ is regular, we can find $A\to R$ and $h\notin Q$ that satisfies the hypothesis of \autoref{theorem: uniform annihilation}. It follows that there exists a sequence of elements $h_1,\dots,h_s$ of $R$ such that $(\bar{h}_1,\dots,\bar{h}_s)$ generates the defining ideal of the singular locus of $\Spec (R/p)$ up to radical and such that there exists $N_1,\dots, N_s$ such that $h_i^{N_i}\omega_R\subseteq \mytau_\scr{B}(\omega_R)$ by \autoref{theorem: uniform annihilation}. Now take $J=(h_1^{N_1},\dots,h_s^{N_s})$ and the result follows.
\end{proof}

We also provide another proof of \autoref{thm.GenSingularlocus} that does not require $k$ to be infinite.

\begin{proof}[Proof of \autoref{thm.GenSingularlocus} without infinite residue field]
Assume $\dim R = d$.
For every $Q \in \Spec R$ where $(R/p)_Q$ is regular, it follows that $R_Q$ itself is regular and hence Gorenstein (so $R_Q \cong \omega_{R_Q}$).  Thus we have the factorization $\omega_{R_Q} \twoheadrightarrow \omega_{R_Q}/p \xrightarrow{\sim} \omega_{R_Q/p}$.
Since $(R/p)_Q$ is regular, there exists $c \in R \setminus Q$ so that $(R/p)_c$ is regular.  Thus $c^n \in R^{\circ}$ is a big test element in the sense of \cite{HochsterFoundations} for some $n > 0$.  It follows that $0 = c^n \cdot 0^*_{H^{d-1}_{\fram}(R/p)}$ where $0^*_{H^{d-1}_{\fram}(R/p)}$ is the tight closure.  Dually, this means that $c^n \omega_{R/p} \subseteq \tau(\omega_{R/p})$ and in particular that $(\tau(\omega_{R/p}))_Q = \omega_{R_Q/p}$.
By \autoref{theorem: big CM restriction} and \autoref{cor.big CM restriction to char p}, we have that the modulo-$p$ image of $\mytau_{\scr{B}}(\omega_R)$ in $\omega_{R/p}$ contains $\tau(\omega_{R/p})$.  Hence, localizing that map, the composition
\[
(\mytau_{\scr{B}}(\omega_R))_Q \subseteq \omega_{R_Q} \to \omega_{R_Q}/p \cong \omega_{R_Q/p}
\]
is surjective and so by Nakayama's lemma, $(\mytau_{\scr{B}}(\omega_R))_Q = \omega_{R_Q} \cong R_Q$.  It follows $\mytau_{\scr{B}}(\omega_R) : \omega_R$ is not contained in $Q$, and the proof is complete.
\end{proof}

\begin{remark}
We hope that the formation of $\mytau_{\scr{B}}(\omega_R)$ commutes with localization, in particular that $\mytau_{\scr{B}}(\omega_R)_Q=\mytau_{\scr{B}}(\omega_{R_Q})$  for all $Q\in \Spec R$. If this is true, then a power of the defining ideal of the singular locus will multiply $\omega_R$ into $\mytau_{\scr{B}}(\omega_R)$.
\end{remark}

\section{Perfectoid big Cohen-Macaulay test ideals of pairs}
\label{sec.BCMTestIdeals}

In this section we develop a theory of \BCM{} test ideals for a pair $(R, \Delta \geq 0)$ and also \BCM{} parameter test submodules for a pair $(\omega_R, \Gamma)$.  We fix a setting to work in, usually we consider the mixed characteristic case, but the setting that follows makes sense in all characteristics.

\begin{setting}
\label{set.PerfectoidTestIdealTransformationRules}
Let $(R,\m)$ be a complete normal local domain.

\begin{itemize}
\item[(i)] The divisor $\Gamma \geq 0$ will always denote an effective $\bQ$-Cartier divisor on $\Spec R$.
\item[(ii)]  The divisor $\Delta \geq 0$, will always denote an effective $\bQ$-divisor on $\Spec R$ such that $K_R + \Delta$ is $\bQ$-Cartier.
\item[(iii)]   If $\Delta \geq 0$ is defined, we fix an embedding $R \subseteq \omega_R \subseteq K(R)$ (and hence we also fix an effective canonical divisor $K_R$).  Since we are working locally, such an effective canonical divisor always exists.
\end{itemize}
Because of the assumptions above, the divisor $K_R + \Delta$ is an effective $\bQ$-Cartier divisor.  Therefore, in many situations, we will set $\Gamma = K_R + \Delta$.

We will use $B$ to denote a big Cohen-Macaulay ${R}^+$-algebra. By hypothesis $n \Gamma$ is Cartier for some integer $n > 0$ and so we can write $n \Gamma = \Div_R(f)$ where $f \in R$ (we have that $f \in R$ and not just in $K(R)$ by our effectivity hypotheses).  Since $B$ is an ${R}^+$-algebra, $f^{1/n}$ makes sense in $B$.  Furthermore, the different choices of $f^{1/n}$ differ by roots of unity, which are themselves units.
\end{setting}

\begin{definition}
\label{def.tauOmegaDelta}
With notation as in \autoref{set.PerfectoidTestIdealTransformationRules}, we define
\[
0^{B, \Gamma}_{H^d_{\m}(R)} = \ker\Big( H^d_{\fram}(R) \xrightarrow{f^{1/n}} H^d_{\fram}(B) \Big).
\]
Moreover, if $R$ has mixed characteristic $(0,p)$, we define
\[
\begin{array}{rl}
0^{\scr{B},\Gamma}_{H_\m^d(R)}=\{\eta\in H_\m^d(R) \hspace{0.5em} |  \hspace{0.5em}&  \exists B \text{ integral perfectoid big Cohen-Macaulay $R^+$-algebra}\\
& \text{ such that } f^{1/n}\eta=0 \text{ in } H_\m^d(B)\}.
\end{array}
\]
We then define
\[
\mytau_B(\omega_R, \Gamma)=\Ann_{\omega_R} 0^{B, \Gamma}_{H^d_{\m}(R)} \text{ and } \mytau_\scr{B}(\omega_R, \Gamma)=\Ann_{\omega_R}0^{\scr{B}, \Gamma}_{H_\m^d(R)}.
\]
We call $\mytau_B(\omega_R, \Gamma)$ the {\BCM{} parameter test submodule of $(\omega_R, \Gamma)$ with respect to $B$} and $\mytau_\scr{B}(\omega_R, \Gamma)$ the perfectoid \BCM{} parameter test submodule of $(\omega_R, \Gamma)$.
\end{definition}


\begin{remark}
\begin{enumerate}
  \item It is easy to see that these definitions are independent of the choice of $f$ or $f^{1/n}$ since any two choices only differ by a unit and so the choice does not impact $0^{B, \Gamma}_{H^d_{\m}(R)}$ and $0^{\scr{B},\Gamma}_{H_\m^d(R)}$.
  \item When $\Gamma=0$, $\mytau_B(\omega_R, 0)=\mytau_B(\omega_R)$ and $\mytau_\scr{B}(\omega_R, 0)=\mytau_\scr{B}(\omega_R)$ are the same as the ones given in \autoref{def.tauOmega} (see also \autoref{remark:completion}).
  \item In mixed characteristic, we have $0^{B, \Gamma}_{H^d_{\m}(R)} \subseteq 0^{\scr{B},\Gamma}_{H_\m^d(R)} \text{ and } \mytau_B(\omega_R, \Gamma)\supseteq \mytau_\scr{B}(\omega_R, \Gamma). $
\end{enumerate}
\end{remark}

We now prove the pair variant of \autoref{prop.tau=tauB}. This implies that our perfectoid \BCM{} parameter test submodule is stable under small perturbation, and thus  everything we will prove later in this section about $\mytau_B$ also holds for $\mytau_\scr{B}$.

\begin{proposition}
\label{prop.tausmallperturb}
With notation as in \autoref{set.PerfectoidTestIdealTransformationRules} and \autoref{def.tauOmegaDelta}, also suppose $(R,\m)$ has mixed characteristic $(0,p)$. Then there exists an integral perfectoid big Cohen-Macaulay $R^+$-algebra $B$ such that for any $0\neq g\in R$ and any rational number $0 < \varepsilon\ll 1$,
\[
0^{B,\Gamma}_{H_\m^d(R)}=0^{B,\Gamma+\varepsilon\Div_R(g)}_{H_\m^d(R)}=0^{\scr{B},\Gamma}_{H_\m^d(R)}=0^{\scr{B},\Gamma+\varepsilon\Div_R(g)}_{H_\m^d(R)}.
\]
It follows that
\[
\mytau_B(\omega_R, \Gamma)=\mytau_B(\omega_R, \Gamma+\varepsilon\Div_R(g))=\mytau_\scr{B}(\omega_R, \Gamma)=\mytau_\scr{B}(\omega_R, \Gamma+\varepsilon\Div_R(g)).
\]
\end{proposition}
\begin{proof}
It suffices to prove the result for a fixed $g$. Notice first that $0^{B, \Gamma+\varepsilon\Div_R(g)}_{H^d_{\fram}(R)}$ stabilizes for $1 \gg \varepsilon > 0$ by the Artinian property of $H^d_{\fram}(R)$ (and likewise for the version with $\scr{B}$), and so working with an arbitrary sufficiently small $\varepsilon$ does make sense.
It is enough to find an integral perfectoid big Cohen-Macaulay $R^+$-algebra $B$ such that $0^{\scr{B},\Gamma+\varepsilon\Div_R(g)}_{H_\m^d(R)}\subseteq 0^{B,\Gamma}_{H_\m^d(R)}$ for all $\varepsilon\ll 1$. For every $\eta\in 0^{\scr{B},\Gamma+\varepsilon\Div_R(g)}_{H_\m^d(R)}$, we have $f^{1/n}g^{\varepsilon}\eta=0$ in $H_\m^d(C)$ for some integral perfectoid big Cohen-Macaulay $R^+$-algebra $C$ and for all $\varepsilon\ll 1$. Therefore by \autoref{lem.CanonicalBconstruction} there exists an integral perfectoid big Cohen-Macaulay $R^+$-algebra $C_\eta$ such that $f^{1/n}\eta=0$ in $H_\m^d(C_\eta)$. Now by the axiom of global choice and \autoref{thm.DominateAnySetOfBCM}, there exists an integral perfectoid big Cohen-Macaulay $R^+$-algebra $B$ such that $C_\eta\to B$ for all $\eta$. It follows that $f^{1/n}\eta=0$ in $H_\m^d(B)$ for all $\eta\in 0^{\scr{B},\Gamma+\varepsilon\Div_R(g)}_{H_\m^d(R)}$ and hence $0^{\scr{B},\Gamma+\varepsilon\Div_R(g)}_{H_\m^d(R)}\subseteq 0^{B,\Gamma}_{H_\m^d(R)}$ as desired.
\end{proof}

\begin{remark}[More general big Cohen-Macaulay algebras]
\label{rem.GeneralBForTestIdeals}
Note that in order for our definition to make sense we did not really need $B$ to be an ${R}^+$ algebra.  We only required $B$ to be a big Cohen-Macaulay $R$-algebra containing an element $f^{1/n}$ whose $n$th power is the image of $f$ inside of $B$.  The advantage of working with $R^+$ algebras is two-fold.  First, no matter which $\Delta$ you pick, if $K_R + \Delta$ is $\bQ$-Cartier, you can find a corresponding $f$.  Second, if one works with $R^+$ algebras, then we may pick our $f^{1/n} \in R^+$ and any such choice differs by a unit, an $n$th root of unity and hence it does not matter which $f^{1/n}$ you pick in the definition of the test ideal.  For a more general big Cohen-Macaulay algebra however, it may be hard to guarantee whether any two $n$th roots of $f$ differ by a unit, and so the choice of $f^{1/n}$ actually matters in the definition.  One can address it for instance by requiring that $B$ is an ${R[f^{1/n}]}^{\textnormal{N}}$ algebra, where ${R[f^{1/n}]}^{\textnormal{N}}\subseteq R^+$ denotes the normalization of $R[f^{1/n}]$.
\end{remark}

Next we prove the following very basic version of Skoda's theorem.

\begin{lemma}
\label{lem.skodatype}
With notation as in \autoref{set.PerfectoidTestIdealTransformationRules}, assume additionally that $\Gamma' = \Gamma + \Div_R(h)$ for some $0\neq h \in R$.  Then
\[
\mytau_B(\omega_R, \Gamma') = h \cdot \mytau_B(\omega_R, \Gamma) 
\]
\end{lemma}
\begin{proof}
We notice that $0_{H^d_{\fram}(R)}^{B, \Gamma'}$ is the kernel of
\[
H^d_{\m}(R)  \xrightarrow{h \cdot f^{1/n}} H^d_{\fram}(B)
\]
which can be factored as
\[
H^d_{\m}(R) \xrightarrow{h} H^d_{\m}(R) \xrightarrow{f^{1/n}} H^d_{\m}(B).
\]
Thus
\[
0^{B, \Gamma'}_{H^d_{\m}(R)} = 0^{B, \Gamma}_{H^d_{\m}(R)} :_{H^d_{\m}(R)} h.
\]
The result follows.
\end{proof}

Next we switch to $\Gamma = K_R + \Delta$ as in \autoref{set.PerfectoidTestIdealTransformationRules}.  We prove that the definition is independence of the choice of $K_R$.

\begin{lemma}
Fix an embedding $R \subseteq \omega_R \subseteq K(R)$ and hence a choice of canonical divisor $K_R$.  Suppose we have another embedding $\omega_R \cong \omega_R' \subseteq K(R)$ (whose image also contains $R$, and also defines an effective $K_R'$).  Then
\[
\mytau_B(\omega_R, K_R + \Delta) = \mytau_B(\omega_R', K_R' + \Delta)
\]
as fractional ideals.
\end{lemma}
\begin{proof}
Without loss of generality (choosing a bigger $K_R'$ if necessary) we may assume that $K_R  + \Div_R(h) = K_R'$ where $h \in R$, then $\omega_R = h \cdot \omega_R'$.  Therefore, for any $\bQ$-Cartier $\Gamma \geq 0$, we have $\mytau_B(\omega_R, \Gamma) = h \cdot \mytau_B(\omega_R', \Gamma)$.   On the other hand, it follows from \autoref{lem.skodatype} that $h \cdot \mytau_B(\omega_R', K_R + \Delta) = \mytau_B(\omega_R', K_R' + \Delta)$.  Combining these observations proves the lemma.
\end{proof}

Because of the previous lemma, the object $\mytau_B(\omega_R, K_R + \Delta)$ is independent of the choice of embedding $\omega_R \subseteq K(R)$ (which is essentially the same as the choice of a canonical divisor $K_R$).  Therefore, since the intersection $\bigcap_{K_R \geq 0} R(K_R) = R$ (since $R$ is normal), $\mytau_B(\omega_R, K_R + \Delta)$ is actually an ideal of $R$ (as pointed out by the referee).  However, we give an alternative proof below.

\begin{lemma}
\label{lem.TestIdealIsAnIdeal}
With notation as in \autoref{set.PerfectoidTestIdealTransformationRules}, $\mytau_B(\omega_R, K_R + \Delta) \subseteq \omega_R$ is in fact a subset and hence an ideal of $R$. The same is true for  $\mytau_\scr{B}(\omega_R, K_R + \Delta)$.
\end{lemma}
\begin{proof}
Let $T \subseteq {R}^+$ denote the normalization of ${R}[f^{1/n}] \subseteq {R}^+$.
By Matlis duality and the fact that $R \to B$ factors through $T$, it suffices to show that $f^{1/n} \cdot \omega_T \xrightarrow{\Tr} \omega_R$ has image contained in $R$. However, this follows because $f^{1/n} \cdot \omega_T = \omega_T(-\pi^* (K_R + \Delta)) \subseteq \omega_{T/R}$ and $\Tr(\omega_{T/R})\subseteq R$. The second conclusion follows from the first and \autoref{prop.tausmallperturb}.
\end{proof}

Based on the above lemmas, we make the following definition.

\begin{definition}
\label{def.tauRDelta}
With notation as in \autoref{set.PerfectoidTestIdealTransformationRules}, we define
$$\mytau_B(R, \Delta)=\mytau_B(\omega_R, K_R + \Delta) \subseteq R ,$$
and in mixed characteristic we define
$$  \mytau_\scr{B}(R, \Delta)=\mytau_\scr{B}(\omega_R, K_R + \Delta) \subseteq R.$$
We call $\mytau_B(R, \Delta)$ the \emph{\BCM{} test ideal of $(R, \Delta)$ with respect to $B$} and $\mytau_\scr{B}(R, \Delta)$ the \emph{perfectoid \BCM{} test ideal of $(R,\Delta)$}.  We say that $(R, \Delta)$ is \emph{big Cohen-Macaulay-regular with respect to $B$} (or simply {\BCMReg{B}}) if $\mytau_B(R, \Delta) = R$, and we say $R$ is \BCMReg{B} if $(R,0)$ is \BCMReg{B}.\footnote{So in particular if $R$ is \BCMReg{B} then $R$ is $\mathbb{Q}$-Gorenstein. } In mixed characteristic we say $(R, \Delta)$ is \emph{perfectoid \BCMReg{}} if $\mytau_\scr{B}(R, \Delta)=R$.
\end{definition}


The following result summarizes basic properties of \BCM{} test ideals and perfectoid \BCMReg{} singularities.

\begin{proposition}
\label{prop.testidealsmallperturb}
With notation as in \autoref{set.PerfectoidTestIdealTransformationRules} and \autoref{def.tauRDelta}, also suppose $(R,\m)$ has mixed characteristic $(0,p)$. Then there exists an integral perfectoid big Cohen-Macaulay $R^+$-algebra $B$ such that for any $0\neq g\in R$ and any rational number $0 < \varepsilon\ll 1$,
$$\mytau_B(R, \Delta)=\mytau_B(R, \Delta+\varepsilon\Div_R(g))=\mytau_\scr{B}(R, \Delta)=\mytau_\scr{B}(R, \Delta+\varepsilon\Div_R(g)).$$
In particular, $R$ is perfectoid \BCMReg{} if and only if $R$ is \BCMReg{B} for one (large enough) integral perfectoid big Cohen-Macaulay $R^+$-algebra $B$ (and equivalently $R$ is \BCMReg{B} for every integral perfectoid big Cohen-Macaulay $R^+$-algebra $B$).
\end{proposition}
\begin{proof}
This follows immediately from \autoref{def.tauRDelta} and \autoref{prop.tausmallperturb}.
\end{proof}

\begin{lemma}
\label{lem.ComparisionDifferentDiv}
With notation as in \autoref{set.PerfectoidTestIdealTransformationRules} and \autoref{def.tauRDelta}, assume we have $0 \leq \Gamma \leq \Gamma'$, both $\bQ$-Cartier $\bQ$-divisors. Then we have
\[
\mytau_B(\omega_R, \Gamma) \supseteq \mytau_B(\omega_R, \Gamma') \text{ and so } \mytau_B(\omega_R, \Gamma) \supseteq f \mytau_B(\omega_R).
\]
In particular if $\Delta \leq \Delta'$ with $\Gamma := \Delta + K_R$ and $\Gamma' := \Delta' + K_R$, as above, then
 \[
 \mytau_B(R, \Delta) \supseteq \mytau_B(R, \Delta') 
 \]
In particular, in mixed characteristic, when $B$ is integral perfectoid, both $\mytau_B(\omega_R, \Gamma)$ and $\mytau_B(R, \Delta)$ are nonzero. The same conclusions hold for $\mytau_\scr{B}$.
\end{lemma}
\begin{proof}
The first part of the first line follows simply by writing $\Div_R(f) = n\Gamma, \Div_{R}(f') = n\Gamma'$ for a common sufficiently large $n$, and noting that $f$ divides $f'$ in $R^+$.  For the second part of the first line simply use \autoref{lem.skodatype}.  The second line follows from the first by \autoref{def.tauRDelta}.
To show that they are nonzero in mixed characteristic, choose a regular local subring $A \subseteq R$ by the Cohen-Structure theorem, and then apply \autoref{theorem: uniform annihilation}. We know that $\mytau_B(\omega_R)$ is nonzero, therefore so is $\mytau_B(\omega, \Gamma)\supseteq f\mytau_B(\omega_R)$. The result for $\mytau_\scr{B}$ follows from the result for $\mytau_B$ and \autoref{prop.testidealsmallperturb}.
\end{proof}

\subsection{Comparison between \BCMReg{B} and \BCMRat{B} singularities} In this subsection we relate \BCMReg{B} and \BCMRat{B} singularities, and in particular we will prove that \BCMReg{B} rings are Cohen-Macaulay.

\begin{theorem}
\label{thm.BCMRegularImpliesCM}
Suppose that $(R, \fram)$ is a complete normal local domain and that $\Delta \geq 0$ is an effective $\bQ$-divisor such that $K_R + \Delta$ is $\bQ$-Cartier. If $(R, \Delta)$ is \BCMReg{B} for some big Cohen-Macaulay $R^+$-algebra $B$, then $R \to B$ is pure and so $R$ is \BCMRat{B} and in particular, $R$ is Cohen-Macaulay.

It follows that if $(R,\m)$ has mixed characteristic $(0,p)$ and $R$ is perfectoid \BCMReg{}, then $R$ is \BCMRat{B} for every integral perfectoid big Cohen-Macaulay $R$-algebra $B$.
\end{theorem}


\begin{proof}
Assume the notations of \autoref{set.PerfectoidTestIdealTransformationRules} and in particular fix $R \subseteq \omega_R$.  We suppose $n > 0$ is the index of $K_R + \Delta$, and write $\Div_R(f) = n(K_R + \Delta)$ for some $f \in R$, we fix $f^{1/n} \in R^+$. Consider the following commutative diagram:
\[
\xymatrix{
R \ar[r] \ar@{^{(}->}[d] & B \ar[d] \ar@/^3pc/[dd]^{\cdot f^{1/n}}\\
\omega_R \ar[r] & B \otimes \omega_R \ar[d] \\
& B.
}
\]
Taking local cohomology we obtain
\[
\xymatrix{
H^d_{\fram}(R) \ar[r] \ar[d] & H^d_{\fram}(B) \ar[d] \ar@/^4pc/[dd]^{\cdot f^{1/n}}\\
H^d_{\fram}(\omega_R) \ar[r]_-{\psi} & H^d_{\fram}(B \otimes \omega_R) \ar[d]_{\mu} \\
& H^d_{\fram}(B).
}
\]
Note that $H^d_{\fram}(R) \to H^d_{\fram}(\omega_R)$ is Matlis dual to $R \subseteq \omega_R$.
Since $(R, \Delta)$ is \BCMReg{B}, we have $R = \mytau_B(R,\Delta)=\mytau_B(\omega_R, K_R+\Delta)\subseteq R$. By Matlis duality we know that if $K = \ker\big( H^d_{\fram}(R) \to H^d_{\fram}(B)\big)$, then we may identify the maps
 \[
 \xymatrix{
 H^d_{\fram}(R) \ar@{=}[d] \ar[r] & H^d_{\fram}(R)/K \ar[d]^{\sim}\\
 H^d_{\fram}(R) \ar[r] & H^d_{\fram}(\omega_R).
 }
 \]
Therefore, we know that $\mu \circ \psi$ is an injection and hence $\psi$ is injective.

Now we observe that, since $H_\m^d(B\otimes\omega_R)\cong B\otimes H_\m^d(\omega_R)$ (which follows from the \v{C}ech complex description of local cohomology \cite[page 130]{BrunsHerzog} and the right exactness of the tensor product) and $H_\m^d(\omega_R)\cong E$ (which is the injective hull of the residue field of $R$), $\psi$ can be identified with the canonical map $E\to E\otimes_R B$. Thus the Matlis dual $\psi^{\vee}$ of $\psi$ is
\[
R \xleftarrow{\psi^{\vee}} \Hom_R(E \otimes_R B, E) \cong \Hom_R(B, \Hom_R(E, E)) \cong \Hom_R(B, R).
\]
It is not difficult to see that this map $\Hom_R(B, R) \to R$ is evaluation at $1$.  Hence the map $R \to B$ splits.  This completes the proof of the first statement.

Finally, if $R$ is perfectoid \BCMReg{}, then the first statement shows that $R\to B$ is pure for every integral perfectoid big Cohen-Macaulay $R^+$-algebra $B$. But since every integral perfectoid big Cohen-Macaulay $R$-algebra maps to an integral perfectoid big Cohen-Macaulay $R^+$-algebra by \autoref{theorem: AndreShimomoto}, we know that $R\to B$ is pure for every integral perfectoid big Cohen-Macaulay algebra $B$. Hence $R$ is \BCMRat{B}.
\end{proof}


It is quite natural to ask whether the converse of \autoref{thm.BCMRegularImpliesCM} holds.

\begin{question}
Suppose $R$ is a complete local normal domain . Further suppose that $R \to B$ is pure for some (or sufficiently large) big Cohen-Macaulay $R^+$-algebra $B$. Does there exist an effective $\bQ$-divisor $\Delta$ such that $K_R + \Delta$ is $\bQ$-Cartier and $(R, \Delta)$ is \BCMReg{B}?
\end{question}

In fact, an affirmative answer to the above question in characteristic $p>0$ will imply that weakly $F$-regular and strongly $F$-regular are equivalent. Thus we expect this question is very difficult in general. We note that in characteristic $p > 0$, some related results were shown in \cite{SchwedeSmithLogFanoVsGloballyFRegular} (also see \cite{DeFernexHaconSingsOnNormal} in characteristic zero).

The next proposition shows that this question has an affirmative answer when $R$ is $\bQ$-Gorenstein.

\begin{proposition}
\label{prop: PureimpliesBCMregforQGor}
Suppose that $(R, \fram)$ is a normal complete domain and is $\bQ$-Gorenstein with index $n$.  Suppose that $B$ is a big Cohen-Macaulay $R^+$-algebra and that $R \to B$ is pure. Then $(R, 0)$ is \BCMReg{B}.
\end{proposition}
\begin{proof}
Write $n K_R = \Div_R(f)$ and choose $f^{1/n} \in R^+$.  We form the same diagram as in the proof of \autoref{thm.BCMRegularImpliesCM}.
\[
\xymatrix{
H^d_{\fram}(R) \ar[r] \ar[d] & H^d_{\fram}(B) \ar[d] \ar@/^4pc/[dd]^{\cdot f^{1/n}}\\
H^d_{\fram}(\omega_R) \ar[r]_-{\psi} & H^d_{\fram}(B \otimes \omega_R) \ar[d]_{\mu} \\
& H^d_{\fram}(B).
}
\]
To show that $(R, 0)$ is \BCMReg{B} we must show that $\mu \circ \psi$ is injective since $H^d_{\fram}(R) \twoheadrightarrow H^d_{\fram}(\omega_R)$ is the Matlis dual of $R \subseteq \omega_R$ (see the proof of \autoref{thm.BCMRegularImpliesCM}).  Since $R\to B$ is pure, $\psi$ is injective. Therefore, it suffices to show that $\mu : H_\m^d(B\otimes \omega_R)\to H_\m^d(B)$ is injective.

Let $S$ be the normalization of $R[f^{1/n}]$. Since $n K_R = \Div(f)$, $S(\Div_S(f^{1/n}))$ is the reflexification of $S\otimes\omega_R$. Thus $H_\m^d(S\otimes\omega_R)\to H_\m^d(S(\Div_S(f^{1/n})))$ is an isomorphism. This implies the second map in $$\cdot f^{1/n}: H_\m^d(S)\to H_\m^d(S\otimes\omega_R)\to H_\m^d(S)$$ is an isomorphism. Therefore since $B$ is an $S$-algebra, after base change to $B$, we know that $\mu$: $H_\m^d(B\otimes\omega_R)\to H_\m^d(B)$ is an isomorphism.
\end{proof}

\begin{corollary}
Suppose $(R,\m)$ is a complete normal Gorenstein local domain. If $R$ is \BCMRat{B} for some big Cohen-Macaulay $R^+$-algebra $B$, then $R$ is \BCMReg{B}. 
\end{corollary}
\begin{proof}
Since $R$ is Gorenstein, $H_\m^d(R)\cong E$ is the injective hull of $R/\m$. Therefore the injectivity of $H_\m^d(R)\to H_\m^d(B)$ implies that $E\to E\otimes B$ is injective. Applying Matlis duality shows that $R\to B$ is split (see the proof of \autoref{thm.BCMRegularImpliesCM}). But then $R$ is \BCMReg{B} by \autoref{prop: PureimpliesBCMregforQGor}. 
\end{proof}

\subsection{Transformation rule under finite maps}
\label{sec.TransformationRulesUnderFiniteMaps}
We can now state a transformation rule for \BCM{} test ideals under finite maps.  This is completely analogous to the main results of \cite{SchwedeTuckerTestIdealFiniteMaps}.  First we recall some setup.

\begin{setup}[\cf \cite{SchwedeTuckerTestIdealFiniteMaps}]
\label{setup.RamificationInGeneral}
For a separable finite extension $R \subseteq S$ of normal domains we have that the field trace $\Tr$ sends $S$ into $R$.  Furthermore, if we fix the ramification divisor $\Ram$ and set $K_S := \phi^* K_R + \Ram$, then the field trace also restricts to
\[
\Tr : \omega_S := S(K_S) \to R(K_R) := \omega_R
\]
which is also identified with the Grothendieck trace and hence is Matlis dual to $H^d_{\fram}(R) \to H^d_{\fram}(S)$.

More generally for an arbitrary finite extension $R \subseteq S$ (if $R \subseteq S$ is inseparable), we consider Grothendieck trace $\Tr \in \Hom_R(\omega_S, \omega_R) \cong S$ generating the set as an $S$-module, and then choose embeddings $R \subseteq \omega_R$ and $S \subseteq \omega_S$ so that $\Tr(S) \subseteq R$.  In this case, $\Tr \in \Hom_R(S, R) = \omega_{S/R}$ corresponds to a divisor which we also call $\Ram$. Note that if $S = R$ is quasi-Gorenstein and the inclusion is an iterate of Frobenius, then we may identify $R$ and $\omega_R$ in which case the corresponding $\Ram = 0$.
\end{setup}

We generalize \autoref{lem.FiniteMapsTransformationForParamTestIdeals} to the setting of pairs.

\begin{theorem}
\label{theorem: finite mapPairsCanonical}
Suppose that $R \subseteq S$ is a finite extension of complete normal local domains with induced $\phi : \Spec S \to \Spec R$.  Then for any $\bQ$-Cartier $\bQ$-divisor $\Gamma \geq 0$ as in \autoref{set.PerfectoidTestIdealTransformationRules} (where we implicitly choose a map $S \hookrightarrow R^+ \to B$ and so also view $B$ as an $S$-algebra) and $\Tr \in \Hom_R(\omega_S, \omega_R)$ the Grothendieck trace, we have
\begin{equation}
\label{eq.TransRuleOmega}
\Tr(\mytau_B(\omega_S, \phi^* \Gamma)) = \mytau_B(\omega_R, \Gamma).
\end{equation}
Further assume that $\Delta$ is a $\bQ$-divisor on $\Spec R$ satisfying the conditions of \autoref{set.PerfectoidTestIdealTransformationRules} and such that $\phi^* \Delta \geq \Ram$, the ramification divisor corresponding to the fixed $\Tr \in \Hom_R(S, R)$ as in \autoref{setup.RamificationInGeneral}.  Then:
\begin{equation}
\label{eq.TransRuleR2}
\mytau_B(R, \Delta) = \Tr( \mytau_B(S, \phi^* \Delta - \Ram) ).
\end{equation}
\end{theorem}
\begin{proof}
Note \autoref{eq.TransRuleR2} is simply a special case of \autoref{eq.TransRuleOmega} based on our choice of $\mytau(R, \Delta)$ and so we only need to prove the first statement.
We first notice that $n\phi^* \Gamma = \Div_S(f)$.  Consider the factorization
\[
H^d_{\m}(R) \to H^d_{\m}(S) \xrightarrow{\cdot f^{1/n}} H^d_{\m}(B).
\]
The kernel of the composition is those elements of $H^d_{\m}(R)$ that map to the kernel of $\cdot f^{1/n}$.  The result follows by Matlis duality.
\end{proof}

\begin{remark}[The non-complete case]
\label{rem.NonCompleteFiniteTransformationRules}
Suppose $R \subseteq S$ is a finite extension of excellent normal domains with $(R, \fram)$ local and $S$ semi-local with maximal ideals $\frn_i$.  Completing this extension we obtain a map $\widehat{R} \to \widehat{S} = \prod_i \widehat{S}_i := \prod_i \widehat{S_{\frn_i}}$.  If one fixes $B$, a big Cohen-Macaulay $\widehat{R}^+$-algebra, then we can choose an embedding $\widehat{S}_i \subseteq \widehat{R}^+$ (which is determined only up to an element of the Galois group of $\widehat{S_i}/\widehat{R}$).  Based on such an embedding, we obtain that $B$ is an $\widehat{S}_i$-algebra for each $i$.

In this way, if $\Delta$ is a $\bQ$-divisor on $S$ such that $K_S + \Delta$ is $\bQ$-Cartier, we can define
\[
\mytau_B(S, \Delta) := S \cap \left( \bigcap_i \mytau_B(\widehat{S}_i, \Delta_i) \right)
\]
where $\Delta_i$ is the pullback of $\Delta$ to $\Spec S_i$.  In view of this definition and the fact that the trace map behaves well under completion, we expect that
\[
\label{eq.NonCompleteTransformationRule}
\mytau_B(R, \Delta) \overset{?}{=} \Tr( \mytau_B(S, \phi^* \Delta - \Ram) )
\]
holds even without the complete hypothesis.  However this is not clear to us since we do not know if the formation of $\mytau_B(R, \Delta)$ commutes with completion.
\end{remark}

We now obtain two corollaries analogous to those obtained in \cite{SchwedeTuckerTestIdealFiniteMaps}.

\begin{corollary}
\label{cor.BCMRegularImpliesPurityOfWildness}
Suppose $R \subseteq S$ is a finite extension of excellent normal domains with $R$ local and with induced $\phi : \Spec S \to \Spec R$.  Suppose $(R, \Delta)$ is a pair so that $(\widehat{R}, \widehat{\Delta})$ is \BCMReg{B} for some $B$ as in \autoref{set.PerfectoidTestIdealTransformationRules}, and $\phi^* \widehat \Delta \geq \widehat \Ram_{\phi} = \Ram_{\widehat{\phi}}$ (for instance if $R \subseteq S$ is \'etale-in-codimension 1).  Then $\Tr : S \to R$ is surjective and hence if $R \subseteq S$ is generically Galois, $R \subseteq S$ is tamely ramified everywhere in any of the senses of \cite{KerzSchmidtOnDifferentNotionsOfTameness}.
\end{corollary}
\begin{proof}
We work using the \emph{notation} of \autoref{rem.NonCompleteFiniteTransformationRules}, note each $\widehat{S}_i$ is a finite extension of $\widehat{R}$, and so we may assume that $R$ is complete and that $S$ is local.  The result follows then since $R = \Tr( \mytau_B(S, \phi^* \Delta - \Ram) ) \subseteq \Tr(S)$.
\end{proof}

\begin{corollary}
Suppose $(R, \fram) \subseteq (S, \frn)$ is a finite and \'etale-in-codimension 1 extension of complete local normal domains with induced $\phi : \Spec S \to \Spec R$.  Suppose further that $(R, \Delta)$ is \BCMReg{B} for some $B$ as in \autoref{set.PerfectoidTestIdealTransformationRules}.  Then $(S, \phi^* \Delta)$ is also \BCMReg{B}.
\end{corollary}
\begin{proof}
We notice that $\Tr : S \to R$ is surjective by \autoref{cor.BCMRegularImpliesPurityOfWildness}.  Furthermore, if we have $\mytau_B(S, \phi^* \Delta) \neq S$, then $\mytau_B(S, \phi^* \Delta) \subseteq \frn$.  It follows from \autoref{theorem: finite mapPairsCanonical} that
$$\mytau_B(R, \Delta)=\Tr(\mytau_B(S, \phi^* \Delta - \Ram))=\Tr(\mytau_B(S, \phi^* \Delta))\subseteq \Tr(\frn) \subseteq \fram$$ which contradicts that $(R, \Delta)$ is \BCMReg{B}.
\end{proof}

\subsection{Comparison with multiplier ideals and test ideals} In this subsection we prove that our perfectoid \BCM{} test ideal is always contained in the multiplier ideal and agrees with the test ideal in characteristic $p > 0$.
\begin{theorem}
\label{thm.TauBCMvsMult}
Suppose $R$ is a complete normal local domain and $\Delta \geq 0$ is a $\bQ$-divisor such that $K_R + \Delta$ is $\bQ$-Cartier.  If $\pi : Y \to X = \Spec R$ is a proper birational map with $Y$ normal,
then there exists an integral perfectoid big Cohen-Macaulay $R^+$-algebra $B$ such that
\[
\mytau_B(R, \Delta) \subseteq \pi_* \cO_Y(\lceil K_Y - \pi^* (K_X + \Delta) \rceil).
\]
Note in the case that $\pi$ is a resolution of singularities, the right side of the displayed equation is the multiplier ideal. It follows that
$$\mytau_\scr{B}(R, \Delta)\subseteq \pi_* \cO_Y(\lceil K_Y - \pi^* (K_X + \Delta) \rceil)$$ for all such $\pi$.
\end{theorem}
Compare the following proof also to \cite[Theorem 8.1]{BlickleSchwedeTuckerTestAlterations}.
\begin{proof}
Write $\Gamma = K_R + \Delta$ as in \autoref{set.PerfectoidTestIdealTransformationRules}.
It suffices to consider the case where $\pi$ is projective, the blowup of $J \subseteq R$, by Chow's lemma.
Let $T$ be the normalization of $\Spec R[f^{1/n}]$ and let $Z := \overline{Y \times_X T}$ denote the normalization of the blowup of $J \cdot \cO_T$.  Let $\pi_{Z/X} : Z \to X$ and $\pi_{Z/T} : Z \to T$ denote the induced maps and notice that $\pi_{Z/X}$ factors through $\pi$.

By \autoref{prop: big CM test ideal birational}, there exists an integral perfectoid big Cohen-Macaulay $R^+$-algebra $B$ such that
\[
\mytau_B(\omega_T) \subseteq \Gamma(Z, \omega_Z).
\]
Multiplying both sides by $f^{1/n}$ and applying $\Tr$, by \autoref{lem.skodatype} and \autoref{theorem: finite mapPairsCanonical}, we have
\[
\mytau_B(R, \Delta) = \mytau_B(\omega_R, \Gamma) \subseteq \Tr \left( \Gamma(Z, \cO_Z(K_Z - \pi_{Z/X}^*(K_X + \Delta))) \right).
\]
Consider the induced diagram:
\[
\xymatrix@R=25pt@C=45pt{
Z \ar[r]^{\eta} \ar[d]_{\phi} \ar[dr]^{\pi_{Z/X}} & T \ar[d]^{\psi} \\
Y \ar[r]_{\pi} & X
}
\]
The proof of \cite[Theorem 8.1]{BlickleSchwedeTuckerTestAlterations} is a proof of the containment:
\[
\Tr \left( \Gamma(Z, \cO_Z(K_Z - \eta^*(\psi^*\Delta - \Ram_{\psi} + K_T)) \right) \subseteq \Gamma(Y, \cO_Y(\lceil K_Y - \pi^*(K_Y + \Delta)\rceil)).
\]
But since $\Ram_{\psi} = K_T - \psi^*K_X$, and $\pi_{Z/X} = \psi \circ \eta$ we have immediately that
\[
\Tr \left( \Gamma(Z, \cO_Z(K_Z - \pi_{Z/X}^*(K_X + \Delta))) \right) \subseteq \Gamma(Y,\cO_Y(\lceil K_Y - \pi^* (K_X + \Delta) \rceil)).
\]
This completes the proof.
\end{proof}

\begin{corollary}
\label{cor.BCMRegImpliesKLT}
If $(X = \Spec R, \Delta)$ is a pair with $R$ an excellent normal local domain such that $(\widehat{X} = \Spec \widehat{R}, \widehat{\Delta})$ is perfectoid \BCMReg{}, then $(X, \Delta)$ is \emph{KLT}.
\end{corollary}
\begin{proof}
This follows immediately from \autoref{thm.TauBCMvsMult} by noting that being KLT can be checked after completion.
\end{proof}

It follows directly from our \autoref{defprop.CharpSings} and \autoref{def.tauRDelta} that our \BCM{} test ideal is the same as the test ideal in equal characteristic $p>0$ (and hence is the same as the usual test ideal when $R$ is $F$-finite by \autoref{prop.DefFregularSameClassical}).
\begin{corollary}
\label{cor.TauBCMVsTestIdeal}
With notation as in \autoref{set.PerfectoidTestIdealTransformationRules}, suppose $(R, \fram)$ is a complete local normal domain of characteristic $p > 0$.  Then $\mytau_B(R, \Delta)=\tau(R, \Delta)$ for every big Cohen-Macaulay $R^+$-algebra $B$. In particular, $R$ is \BCMReg{B} (for some $B$ or for every $B$) if and only if $R$ is strongly $F$-regular in characteristic $p>0$.
\end{corollary}

\subsection{More general restriction theorems}

Our goal in this subsection is to prove a general restriction theorem for \BCM{} test ideals.
We first include lemmas on the construction of canonical divisors.  These are obvious to experts but we do not know a reference in this generality.  The first is a slight divisorial variation of the usual prime avoidance lemma.

\begin{lemma}
\label{lem.CanonicalDivisorCanAvoid}
Suppose $R$ is a normal domain and $\omega$ is a rank one reflexive module. Further suppose that $D_1, \ldots, D_n$ are a list of distinct prime divisors.  Then there exists an effective divisor $D$ with no common components with the $D_i$ and such that $R(D)\cong \omega$.
\end{lemma}
\begin{proof}
We proceed by induction on $n$.  In the case that $n = 1$, simply choose an element $z \in \omega \setminus \omega(-D_1)$.  Then the effective divisor corresponding to the section $z \in \omega$ does not have $D_1$ as a component.  For the inductive case, suppose we have a section $z \in \omega$ such that
\[
z \notin \omega(-D_i)
\]
for $i = 1, \dots, n-1$.  If $z \notin \omega(-D_n)$, we are done, so suppose $z \in \omega(-D_n)$.  Next notice that $\omega(-D_1 - \dots -D_{n-1}) \not \subseteq \omega(-D_n)$ since the $D_i$ are distinct, we choose
\[
y \in \omega(-D_1 - \dots -D_{n-1}) \setminus \omega(-D_n).
\]
Finally we consider $y + z$, this cannot be in $\omega(-D_i)$ for $i = 1, \dots, n$ by construction.  The divisor corresponding to the section $y + z \in \omega$ is effective and has no common components with the $D_i$.
\end{proof}

\begin{lemma}
\label{lem.setupCanonical}
Suppose that $(R, \m)$ is a normal local ring of mixed characteristic $(0, p)$ with a canonical module $\omega_R$ and an element $0 \neq v \in R$.  Additionally fix $\Delta \geq 0$ a $\bQ$-divisor on $\Spec R$ such that $\Delta$ and $\Div(v)$ have no common components, and suppose that $K_R + \Delta$ is $\bQ$-Cartier.  Then there exists a choice of canonical divisor $K_R \geq 0$ (corresponding to an embedding $R \subseteq \omega_R$) satisfying the following conditions.
\begin{enumerate}
\item The cokernel $\omega_R / R$ is unmixed (all associated primes are minimal primes of height $1$ in $R$).
\label{lem.setupCanonical.UnmixedCokernel}
\item $v$ is a regular element on $\omega_R / R$.
\label{lem.setupCanonical.VarpiRegularElt}
\item If additionally, the index of $K_R + \Delta$ is not divisible by $p > 0$, then we may form the normalization of a local cyclic index-1 cover $R \subseteq R'$ that is \'etale over the generic points of $V(v)$ and such that if $\pi : \Spec R' \to \Spec R$ is the induced map, then $\pi^* (K_R + \Delta)$ is Cartier.
\label{lem.setupCanonical.FiniteExtension}
\end{enumerate}
\end{lemma}
\begin{proof}
First use \autoref{lem.CanonicalDivisorCanAvoid} to choose an effective canonical divisor $K_R$ with no common components with any component of $\Delta$ or $\Div(v)$.  The choice of $K_R$ fixes an embedding $R \subseteq \omega_R = R(K_R)$ and so we have
\[
0 \to R \to \omega_R \to \omega_R/R \to 0.
\]
We claim that $\omega_R/R$ has no associated primes $Q$ of height $\geq 2$ in $R$.  Indeed if $Q$ is a height 2 prime,
\[
0 = H^0_Q(\omega_R) \to H^0_Q(\omega_R/R) \to H^1_Q(R) = 0
\]
where the first vanishing is since $\omega_R$ is torsion free and $H^1_Q(R) = 0$ since $R$ is S2.  It follows that $H^0_Q(\omega_R/R) = 0$ and so \eqref{lem.setupCanonical.UnmixedCokernel} holds. Then \eqref{lem.setupCanonical.VarpiRegularElt} follows since $\Supp \omega_R/R$ is $\Supp K_R$, thus any associated prime of $\omega_R/R$ which contains $v$ also has height at least 2 in $R$ (and there are no such primes by \eqref{lem.setupCanonical.UnmixedCokernel}).

Finally, we prove \eqref{lem.setupCanonical.FiniteExtension}.  We fix $R''$ to be a ramified cyclic cover of $R$ along the $\bQ$-Cartier $\bQ$-divisor $K_R + \Delta$, see for instance \cite[Section 2.3]{KollarKovacsSingularitiesBook}.  Further let $R'$ denote the normalization of $R''$.  Since $K_R + \Delta$ has no common components with $V(v)$, we see that $R \subseteq R''$ is \'etale where claimed, and so $R''$ is normal over the generic points of $V(v)$, and thus $R \subseteq R'$ is \'etale over the generic points of $V(v)$ as well.  Since the pullback of $K_R + \Delta$ to $\Spec R''$ is Cartier, it remains Cartier after pullback to $R'$.  This proves \eqref{lem.setupCanonical.FiniteExtension}.
\end{proof}

Next we generalize \autoref{theorem: big CM restriction} to the context of pairs.

\begin{theorem}
\label{thm.BCMTestIdeal1minusEpsilon}
Let $(R,\m)$ be a complete normal local domain of dimension $d$ and let $B$ be a big Cohen-Macaulay $R^+$-algebra. Then for every nonzero element $x\in R$ and every rational $1 > \varepsilon > 0$, the image of $\mytau_B(\omega_R, (1-\varepsilon)\Div_R(x))$ under the natural map $\omega_R\to \omega_R/x\omega_R\to\omega_{R/xR}$ equals $\mytau_{B/x^{\varepsilon}B}(\omega_{R/xR})$.

In particular, if $C$ is a big Cohen-Macaulay $R/xR$-algebra which is also canonically a $B/x^{\varepsilon}B$-algebra, then we have that the image of $\mytau_B(\omega_R, (1-\varepsilon)\Div_R(x))$ under the natural map $\omega_R\to \omega_R/x\omega_R\to\omega_{R/xR}$ contains $\mytau_{C}(\omega_{R/xR})$.
\end{theorem}
\begin{proof}

We first note that $B/x^{\varepsilon}B$ is canonically an algebra over $R/xR$ via the map $R/xR\to B/xB\to B/x^{\varepsilon}B$. The remainder of the proof is essentially the same as that of \autoref{theorem: big CM restriction} but we start instead with the diagram
\[
\xymatrix{
0 \ar[r] & R \ar[r]^{\cdot x} \ar[d]_{\cdot x^{1-\varepsilon}} & R \ar[d] \ar[r] & R/xR \ar[r]\ar[d] & 0\\
0\ar[r] & B\ar[r]^{\cdot x^{\varepsilon}} & B\ar[r] & B/x^{\varepsilon}B \ar[r] & 0
}
\]
which induces
\[
\xymatrix{
{} & H_\m^{d-1}(R/xR) \ar[r] \ar[d]_{\beta} & H_\m^d(R) \ar[d]_{\alpha = (\cdot x^{1-\varepsilon}) \circ \gamma} \ar[r]^{\cdot x} & H_\m^d(R) \ar[r]\ar[d]^{\gamma} & 0\\
0\ar[r] & H_\m^{d-1}(B/x^{\varepsilon}B) \ar[r] & H_\m^d(B) \ar[r]^{\cdot x^{\varepsilon}} & H_\m^d(B) \ar[r] & 0.
}
\]
We know that $\im(\beta) \hookrightarrow \im(\alpha)$.  But the Matlis dual of $\im(\beta)$ is $\mytau_{B/x^{\varepsilon}B}(\omega_{R/xR})$ and the Matlis dual of $\im(\alpha)$ is $\mytau_B(\omega_R, (1-\varepsilon)\Div_R(x))$.  The result follows.
\end{proof}


Using the above and a cyclic cover, we obtain the following rather general restriction theorem.

\begin{theorem}
\label{thm.TestRestriction}
Suppose that $R$ is a complete normal local domain of mixed characteristic $(0, p)$, and that $0\neq h \in R$ is such that $R/hR$ is also a normal domain.  Additionally fix $\Delta$ a $\bQ$-divisor on $R$ such that $K_R + \Delta$ is $\bQ$-Cartier with index not divisible by $p$ and that $V(h)$ and $\Delta$ have no common components.  Then for every integral perfectoid big Cohen-Macaulay $R^+$-algebra $B$ and any $1 > \varepsilon > 0$, there exists a big Cohen-Macaulay $({R}/h{R})^+$-algebra $C$ (that is integral perfectoid if $R/hR$ has mixed characteristic) together with a compatible map $B \to C$ so that:
\[
\mytau_{C}(R/hR, \Delta|_{R/hR}) \subseteq \mytau_B(R, \Delta + (1-\varepsilon)\Div_R(h)) \cdot (R/hR).
\]
It follows that if $R/hR$ has mixed characteristic $(0,p)$, we have
\[
\mytau_{\scr{B}}(R/hR, \Delta|_{R/hR}) \subseteq \mytau_\scr{B}(R, \Delta + (1-\varepsilon)\Div_R(h)) \cdot (R/hR);
\]
and if $R/hR$ has equal characteristic $p>0$, then we have
\[
\tau(R/hR, \Delta|_{R/hR}) \subseteq \mytau_\scr{B}(R, \Delta + (1-\varepsilon)\Div_R(h)) \cdot (R/hR).
\]
In particular, if $(R/hR, \Delta|_{R/hR})$ is perfectoid \BCMReg{} in mixed characteristic or strongly $F$-regular in equal characteristic $p>0$, then $(R, \Delta + (1-\varepsilon)\Div_R(h))$ is perfectoid \BCMReg{} for every $1 > \varepsilon > 0$.
\end{theorem}
\begin{proof}
The statements on $\mytau_\scr{B}$, $\tau$, perfectoid \BCM{}-regularity and strongly $F$-regularity follow immediately from the statement on $\mytau_B$ by considering \autoref{prop.testidealsmallperturb} and \autoref{cor.TauBCMVsTestIdeal}. Therefore we only need to prove the containment involving $\mytau_C$ and $\mytau_B$.

We fix some rational $\varepsilon > 0$ for the rest of the proof.
We first describe $\Delta|_{R/hR}$.  Since $R/hR$ is normal, we know that each point of $V(h) \subseteq \Spec R$, which has height $\leq 2$ as an ideal of $R$, is regular.  Hence each point in $\Supp(\Delta) \cap V(h)$ that has codimension $1$ in $V(h)$ is regular in $\Spec R$.  It follows that one may define $\Delta|_{R/hR}$ in codimension $1$ and so we can define it everywhere.  Finally, set $K_R + \Delta = {1 \over n}\Div_R(f)$ where $f$ has no common components with $V(h)$ (we can do this by \autoref{lem.CanonicalDivisorCanAvoid} and \autoref{lem.setupCanonical}).

\begin{claim}
If $\overline{f} \in R/hR$ denotes the image of $f$, then we may choose $K_{R/hR}$ so that
\[
(K_R + \Delta)|_{V(h)} = {1 \over n}\Div_R(f)|_{V(h)} =  {1 \over n} \Div_{R/hR}(\overline{f}) = K_{R/hR} + \Delta|_{V(h)}.
\]
\end{claim}
\begin{proof}[Proof of Claim]
This is obvious for regular schemes, set $K_{R/hR} = K_R|_{V(h)}$ and use the fact that we are working locally.
However, the computation can be done at codimension $1$ points of $V(h)$, which correspond to codimension-2 regular points of $\Spec R$.  In particular, we can reduce to the regular case.
This proves the claim.
\end{proof}

Fix a normalization of a ramified cyclic cover $R \subseteq R'$ with respect to $K_R + \Delta$ and $h =: v$ as in \autoref{lem.setupCanonical}.  Because $R'$ is normal and $R' \supseteq R$ is \'etale over $V(h)$, we see that $R'/hR'$ is generically reduced and all the irreducible components of $\Spec(R'/hR')$ are of dimension $\dim R - 1$.  Because $R'$ is S2, we see that $R'/hR'$ is S1, and so $R'/hR'$ is reduced.  Finally, observe that $B$ is also an $R'$ algebra since we may embed $R' \subseteq R^+$.  Now using the weakly functorial big Cohen-Macaulay algebras in the form of \autoref{rem.functorialityOfNonReducedAndre}, we fix $C$ a big Cohen-Macaulay $(R'/hR')^+$-algebra (that is integral perfectoid in mixed characteristic) which admits a map from $B$ and so satisfies the role of $C$ in the diagram in \autoref{theorem: Andre}.  Notice that since $C$ is an $(R'/hR')^+$-algebra, when we view $C$ as an $R^+=R'^+$-algebra, $h^{\varepsilon} C = 0$ because $h^{\varepsilon}\cdot (R'/hR')^+=0$. Thus $C$ is a $B/h^{\varepsilon} B$-algebra as well. 

Next notice that we have an induced finite generically \'etale map $R/hR \hookrightarrow R'/hR'$ and so $C$ is also a big Cohen-Macaulay $R/hR$-algebra.  If $S$ is the normalization of $R'/hR'$, then it is a product of complete normal local domains $S = \prod S_i$, each of which is a finite extension of $R/hR$.  By our construction from \autoref{rem.functorialityOfNonReducedAndre}, $C$ is equal to a product of $S_i^+$-algebras.  By \autoref{lem.FiniteMapsTransformationForParamTestIdeals} and \autoref{theorem: finite mapPairsCanonical} 
we obtain the following diagram:
\[
\xymatrix{
\mytau_B(\omega_R, (1-\varepsilon)\Div_{R}(h)) \ar@/_2pc/[dd]_{\phi} \ar@{^{(}->}[d] &  \ar@{->>}[l]_{\Tr} \mytau_B(\omega_{R'}, (1-\varepsilon)\Div_{R'}(h)) \ar@{^{(}->}[d] \ar@/^2pc/[dd]^{\psi} \\
\omega_R \ar[d] & \ar[l] \omega_{R'} \ar[d] \\
\omega_{R/hR} & \ar[l] \omega_{R'/hR'} & \ar[l] \omega_S\\
\mytau_{C}(\omega_{R/hR}) \ar@{_{(}->}[u]^{\mu} & \ar@{_{(}->}[u]^{\nu} \ar@{->>}[l]_{\Tr} \mytau_{C}(\omega_{R'/hR'}) & \ar@{->>}[l] \mytau_C(\omega_S) \ar@{_{(}->}[u]_{\rho} \\
}
\]
where by \autoref{thm.BCMTestIdeal1minusEpsilon} the maps $\phi$ and $\psi$ contain the images of $\mu$ and $\nu$ respectively.

We will now consider the map $\Tr : f^{1/n} \cdot \mytau_C(\omega_S) \to \mytau_C(\omega_{R/hR})$.
By construction, note that $\mytau_C(\omega_{R/hR}) = \sum_{i = 1}^n \mytau_{C_i}(\omega_{R/hR})$ and likewise $\mytau_{C}(R/hR, \Delta|_{V(h)}) = \sum_{i = 1}^n \mytau_{C_i}(R/hR, \Delta|_{V(h)})$.  Next observe that by \autoref{lem.skodatype}, $f^{1/n} \cdot \mytau_{C_i}(\omega_{S_i}) = \mytau_{C_i}(\omega_{S_i}, \psi_i^* (K_R + \Delta)|_{V(h)})$ where $\psi_i: \Spec S_i \to \Spec R/hR$ is the induced map.  But now by \autoref{theorem: finite mapPairsCanonical}
\[
\Tr\big(f^{1/n} \cdot \mytau_{C_i}(\omega_{S_i})\big) = \mytau_{C_i}(\omega_{R/hR}, (K_R + \Delta)|_{V(h)}) = \mytau_{C_i}(R/hR, \Delta|_{V(h)}).
\]
Summing over all $S_i$ we see that $\Tr : f^{1/n} \cdot \mytau_C(\omega_S) \twoheadrightarrow \mytau_{C}(R/hR, \Delta|_{V(h)})$ is surjective. Since this factors through $f^{1/n} \cdot \mytau_C(\omega_S) \twoheadrightarrow f^{1/n} \cdot \mytau_C(\omega_{R'/hR'})$, hence we have that
\[
\Tr (f^{1/n} \cdot \mytau_C(\omega_{R'/hR'})) = \mytau_{C}(R/hR, \Delta|_{V(h)}).
\]

Multiplying the middle column of the large diagram above by $f^{1/n}$ and using again \autoref{lem.skodatype} and \autoref{theorem: finite mapPairsCanonical} we may factor the rows of the above diagram as follows:
\[
\xymatrix{
\mytau_B(\omega_R, (1-\varepsilon)\Div_{R}(h)) \ar@/_3pc/[dd]_{\phi} \ar@{^{(}->}[d] & \ar@{^{(}->}[d] \ar@{_{(}->}[l] \mytau_B(R, \Delta+(1-\varepsilon)\Div_{R}(h)) & \ar@{->>}[l]_-{\Tr} f^{1/n} \cdot \mytau_B(\omega_{R'}, (1-\varepsilon)\Div_{R'}(h)) \ar@{^{(}->}[d] \ar@/^3pc/[dd]^{\psi} \\
\omega_R \ar[d] & \ar@{_{(}->}[l] R \ar[d] & \ar[l] f^{1/n} \cdot \omega_{R'} \ar[d] \\
\omega_{R/hR} & \ar@{_{(}->}[l] R/hR & \ar[l] f^{1/n} \cdot \omega_{R'/hR'} \\
\mytau_{C}(\omega_{R/hR}) \ar@{_{(}->}[u]^{\mu} & \ar@{_{(}->}[u] \ar@{_{(}->}[l] \mytau_{C}(R/hR, \Delta|_{V(h)}) & \ar@{_{(}->}[u]^{\nu} \ar@{->>}[l]_-{\Tr} f^{1/n} \cdot \mytau_{C}(\omega_{R'/hR'}) \\
}
\]
Since the image of $\psi$ contains the image of $\nu$, it follows from the diagram that the image of $\mytau_B(R, \Delta+(1-\varepsilon)\Div_{R}(h))$ in $R/hR$ contains $\mytau_{C}(R/hR, \Delta|_{V(h)})$.  This proves the theorem.
\end{proof}


Combining \autoref{thm.TestRestriction} and \autoref{cor.BCMRegImpliesKLT}, we immediately obtain the following result.
\begin{corollary}
\label{cor.FregDeforms}
Use the notation of \autoref{thm.TestRestriction} but instead of assuming that $R$ is complete assume that $R$ is only excellent. Assume additionally that $R/hR$ has equal characteristic $p > 0$ (e.g., $h=p$). If the pair $(R/hR, \Delta|_{R/hR})$ is strongly $F$-regular, then the pair $(R, \Delta+(1-\varepsilon)\Div_{R}(h))$ is \emph{KLT} for every $\varepsilon > 0$ and hence $(R, \Delta + \Div_R(h))$ is log canonical.
\end{corollary}
\begin{proof}
Since strongly $F$-regularity is preserved by passing to completion, we know that $(\widehat{R}/h\widehat{R}, \widehat{\Delta}|_{\widehat{R}/h\widehat{R}})$ is strongly $F$-regular. By \autoref{thm.TestRestriction}, $(\widehat{R}, \widehat{\Delta}+(1-\varepsilon)\Div_{\widehat{R}}(h))$ is perfectoid \BCMReg{} and hence $(R, \Delta+(1-\varepsilon)\Div_{R}(h))$ is KLT by \autoref{cor.BCMRegImpliesKLT}.
\end{proof}

\begin{remark}
One should also expect more general inversion of adjunction / restriction theorem statements where the subscheme we are restricting to is a Weil divisor and not necessarily Cartier.  For example, if $D \subseteq \Spec R$ is a normal prime $\bQ$-Cartier divisor and $K_R + \Delta$ is also $\bQ$-Cartier, then if $(D, \mathrm{Diff}_D (K_R + \Delta))$ is perfectoid \BCMReg{} in mixed characteristic or strongly $F$-regular in characteristic $p>0$, we should expect that $(R, \Delta+(1-\varepsilon) D)$ is perfectoid \BCMReg{}.  In fact, one would even expect that $(R, \Delta+D)$ is an analog of PLT in mixed characteristic, compare with \cite[Theorem 5.50]{KollarMori}.  Likewise one would expect sharper restriction theorems analogous to \cite[Theorem 4.4]{TakagiPLTAdjoint} or \cite[Theorem 9.5.16]{LazarsfeldPositivity2}. 
\end{remark}

In view of this remark, we prove one more result in this direction in the situation when there is no Shokurov's different \cite{ShokurovThreeDimensionalLogFlips}.

\begin{proposition}
Suppose $(R, \fram)$ is a complete S2 local domain of dimension $d$ and $I \cong \omega_R$ is a proper canonical ideal (note if $R$ is normal and $D = \Div_R(I)$ is an anti-canonical divisor, then $K_R + D = \Div_R(f)$, thus this fits into the framework of pairs we discussed above).  Suppose that $R/I$ is \BCMRat{C} for some big Cohen-Macaulay $R/I$-algebra $C$ admitting a compatible map from a big Cohen-Macaulay $R$-algebra $B$.  Then $R \to B$ splits.  In particular, if $R$ is $\bQ$-Gorenstein of mixed characteristic $(0,p)$ and $R/I$ is \BCMRat{}, then $R$ is \emph{KLT}.
\end{proposition}
\begin{proof}
We begin with a claim.
\begin{claim}
$H^d_{\fram}(I \otimes_R B) \cong H^d_{\fram}(I \cdot B)$ via the canonical map.
\end{claim}
\begin{proof}[Proof of claim]
It suffices to proves that for every finitely generated $R$-module $N \subseteq B$ we have that $H^d_{\fram}(I \otimes_R N) \to H^d_{\fram}(I \cdot N)$ is an isomorphism.  We notice that if we write $0 \to K \to I \otimes N \to I \cdot N \to 0$, then $K$ is supported on a set of codimension $\geq 1$ (since $I$ is free in codimension 0).    Hence $H^d_{\fram}(K) = 0$.
This proves the claim.
\end{proof}
Consider the diagram
\[
\xymatrix{
0 \ar[r] & I \ar[d] \ar[r] & R \ar[d]\ar[r] &  R/I \ar[d]\ar[r] & 0\\
0 \ar[r] & I \cdot B \ar[r] & B \ar[r] & B/(I\cdot B) \ar[r] & 0.\\
}
\]
Taking local cohomology and using the claim we obtain
\[
\xymatrix{
& H^{d-1}_{\fram}(R/I) \ar[d] \ar[r] & H^d_{\fram}(I) \ar[d]\ar[r] & H^d_{\fram}(R) \ar[d]\\
0 \ar[r] & H^{d-1}_{\fram}(B/(I \cdot B)) \ar[d] \ar[r] & H^d_{\fram}(I \otimes B) \ar[r] & H^d_{\fram}(B). \\
& H^{d-1}_{\fram}(C)
}
\]
We next claim that $H^d_{\fram}(I) \to H^d_{\fram}(I \otimes B)$ is injective.
To see this, note that $H^d_{\fram}(I) \cong E$, the injective hull of the residue field, and hence it has a $1$-dimension socle.  By local duality, the Matlis dual of $H^d_{\fram}(I) \to H^d_{\fram}(R)$ is $R \leftarrow \omega_R$, and this is the embedding of $I \cong \omega_R$ as a proper ideal.  Thus $H^d_{\fram}(I) \to H^d_{\fram}(R)$ is not injective and so a socle generator $z$ is mapped to zero in $H^d_{\fram}(R)$ and so has a pre-image $y \in H^{d-1}_{\fram}(R/I)$.  Since $H^{d-1}_{\fram}(R/I) \to H^{d-1}_{\fram}(C)$ is injective, $y$ has nonzero image also in $H^{d-1}_{\fram}(B/(I \cdot B))$.  In conclusion, $z$ has nonzero image in $H^d_{\fram}(I \tensor B)$ and so
\[
H^d_{\fram}(I) \hookrightarrow H^d_{\fram}(I \otimes B)
\]
is injective as claimed. Since $H_\m^d(I)\cong E$ is the injective hull of the residue field of $R$ and $H_\m^d(I\otimes B)\cong H_\m^d(I)\otimes B\cong  E\otimes B$ (see \cite[page 130]{BrunsHerzog}), the above injection says that $E\to E\otimes B$ is injective. Then, exactly as in the proof of \autoref{thm.BCMRegularImpliesCM}, taking the Matlis dual tells us that $R\to B$ is split, as desired.  The final statement follows from \autoref{prop: PureimpliesBCMregforQGor} and \autoref{cor.BCMRegImpliesKLT}.
\end{proof}

We conclude the section with a corollary which follows immediately from combining \autoref{cor.BCMRegularImpliesPurityOfWildness} and \autoref{thm.TestRestriction}.

\begin{corollary}[Purity of the tamely ramified locus]
Suppose $R \subseteq S$ is a finite generically Galois extension of complete normal domains of mixed characteristic $(0,p)$ such that $R$ is $\bQ$-Gorenstein.  Suppose $0\neq f \in R$ is an element so that $R/fR$ is perfectoid \BCMReg{} in mixed characteristic or strongly $F$-regular in equal characteristic $p>0$. Suppose also that $R[f^{-1}] \subseteq S[f^{-1}]$ is \'etale and $R \subseteq S$ is tamely ramified in codimension $1$.  Then $\Tr : S \to R$ is surjective and so $R \subseteq S$ is tame everywhere.
\end{corollary}
\begin{proof}
By \autoref{thm.TestRestriction}, $(R, (1-\varepsilon) \Div_R(f))$ is perfectoid \BCMReg{} for all $1 > \varepsilon > 0$.  Since $R \subseteq S$ is tamely ramified, if $\pi : \Spec S \to \Spec R$ is the induced map, then $\pi^*(1-\varepsilon)\Div_R(f) \geq \Ram_{\pi}$.  The result follows by \autoref{cor.BCMRegularImpliesPurityOfWildness}.
\end{proof}

\section{Application: $F$-rational and strongly $F$-regular singularities in families}
\label{sec.ApplicationToArithmeticFamilies}

Our goal in this section is to study the following question:
\begin{center}
\emph{How do $F$-rationality and $F$-regularity vary as one varies the characteristic?}
\end{center}
We begin with a simple case which will illustrate the larger point.


\begin{setting}
\label{set.DedekindDomainSetting}
Suppose a Dedekind domain $D$ is a quasi-finite extension of $\bZ$.  In particular, $D$ is a localization of a finite extension of $\bZ$.  We fix $K$ to be the fraction field, a field of characteristic zero.  Let $\phi : X \to U := \Spec D$ be a flat family essentially of finite type.  When talking about a $\bQ$-divisor $\Delta$ on $X$, we always assume that $\Delta$ is effective and has no vertical components (in other words, every component of $\Delta$ dominates $D$).  For any point $p \in D$, we use $X_p$ (resp. $\Delta_p$) to denote the fiber.
\end{setting}

Little is lost if you simply assume that $D = \bZ$ (or a localization of $\bZ$).

\begin{theorem}
\label{thm.ProjectiveFamilyFratModPImpliesChar0}
With notation as in \autoref{set.DedekindDomainSetting}, let $\phi: X\to U := \Spec D$ be a proper flat family.  Suppose $X_p$ is $F$-rational for some closed point $p \in U$.  Then $X_{K}$ has rational singularities.  Furthermore $X_q$ is $F$-rational for a Zariski dense and open set $V$ of closed points $q \in U$.
\end{theorem}
\begin{proof}
Let $Y$ be the non-pseudo-rational locus of $X$. Since we do not have resolution of singularities, it is not clear to us that this set is closed.   Hence, we let $\overline{Y}$ denote the closure of $Y$ and we will eventually show that $\overline{Y}\cap X_p=\emptyset$. So suppose not, then we pick $z\in\overline{Y}\cap X_p$ viewed as a point of $X$, and let $W$ be an irreducible component of $\overline{Y}$ that contains $z$.
\begin{claim}
We may assume that $\phi(W) = U$
\end{claim}
\begin{proof}[Proof of Claim]
Note that $\phi(W)$ is closed since $W$ is closed and $\phi$ is proper.  If $\phi(W) \subsetneq U$ then $\phi(W)$ is discrete, which means $W\subseteq X_{p_1}\cup X_{p_2}\cup\cdots\cup X_{p_n}$ for some points $p_1=p, p_2,\dots,p_n \in U$. Since $W$ is irreducible and $X_{p_i}$ are disjoint, we must have $W \subseteq X_p$ (since $z\in W$ and $z\in X_p$). But for every $y\in X_p$, $\cO_{X,y}/p$ is a local ring of $X_p$ and thus is $F$-rational by hypothesis. Hence $\cO_{X,y}$ is pseudo-rational for every $y\in  X_p$ by \autoref{theorem: F-ratPseudorat}. However, then we have $X_p \cap Y\supseteq W\cap Y\neq \emptyset$, a contradiction.  Thus we may assume $\phi(W) = U$.
\end{proof}

Now we fix $\eta$ to be the generic point of $W$.  We see that the local ring $\cO_{X, \eta}$ is pseudo-rational (because this is a localization of $\cO_{X, z}$) and hence rational since $\cO_{X, \eta}$ has characteristic $0$. The key point for us is that being rational is an open condition in characteristic zero since resolution of singularities exists.  We next choose $\Spec R$, an open neighborhood of $\eta \in X$, such that $\Spec (R \otimes_{D} K)$ has rational singularities.  By the main result of \cite{HaraRatImpliesFRat,MehtaSrinivasRatImpliesFRat}, for an open dense set $V$ of points $\mathfrak{q}$ of $U$, $R/\mathfrak{q}R$ has $F$-rational singularities.  It follows from \autoref{theorem: F-ratPseudorat} that all points in the \emph{open} set $\phi^{-1} V \cap \Spec R$ have pseudo-rational singularities.  In particular, $\overline{Y} \cap (\phi^{-1} V \cap \Spec R) = \emptyset$ and so $\overline{Y}$ is a \emph{proper} closed subset of $X$.
However, we notice that $\eta \in \overline{Y}$, $\eta \in \Spec R$ and $\eta \in \phi^{-1}(V)$ (since $\eta$ lies over $(0) \subseteq D$), but then $\eta \in \overline{Y} \cap (\phi^{-1} V \cap \Spec R)$ a contradiction.  Therefore we conclude that $X_p \cap \overline{Y} = \emptyset$.

Finally, since $\overline{Y}\cap X_p=\emptyset$, $\phi(\overline{Y})$ is a closed subset of $U$ that does not contain $p$. Thus $\phi(\overline{Y})$ does not contain the generic point of $U$ and hence $X_{K}$ is pseudo-rational, therefore has rational singularities.
\end{proof}

Without the properness hypothesis (or a suitable locality hypothesis, see \autoref{thm.LocalOpenFrationalOverZ} below), the above result is false.  Consider the following example:

\begin{example}
\label{ex.BadAffineExample}
Consider the diagram of rings
\[
\big\{ \bZ[x] \xrightarrow{x \mapsto \overline{x}} \bZ[x]/( px - 1 ) \xleftarrow{\overline{x} \mapsfrom y} \bZ[y] \big\}
\]
and let $R$ denote the pullback.  In other words
\[
\begin{array}{rl}
R =   & \big\{ (f(x), g(y)) \in \bZ[x] \times \bZ[y]\;\big|\; f(1/p) = g(1/p) \in \bZ[1/p] \big\}\\
\cong & \ker \Big( \bZ[x] \times \bZ[y] \xrightarrow{(f(x), g(y)) \mapsto f(1/p) - g(1/p)} \bZ[1/p] \Big)
\end{array}
\]
First observe that $R/p \cong \bF_p[x] \oplus \bF_p[y]$ and in particular, $R/p$ has $F$-rational singularities.  However, for any prime $t \neq p$,
\[
\begin{array}{rl}
R/t \cong & \big\{ (f(x), g(y)) \in \bF_t[x] \times \bF_t[y]\;\big|\; f(1/p) = g(1/p) \in \bF_t \big\}\\
\cong & \ker\Big(\bF_t[x] \times \bF_t[y] \xrightarrow{(f(x), g(y)) \mapsto f(1/p) - g(1/p)} \bF_t\Big).
\end{array}
\]
But this is a nodal singularity, and in particular not normal or $F$-rational.  In particular $R/p$ is $F$-rational but $R/t$ is not $F$-rational for any other prime $t \in \bZ$.
\end{example}

Of course, the same sort of example as the one above occurs in characteristic zero.  Consider an affine family over a smooth curve, $X \to C$.  One could picture the badly singular locus (for example, the locus with non-rational singularities) of $X$ as a hyperbola which has a vertical asymptote over the closed fiber $X_c$ (for some $c \in C$).  Then it can happen that the closed fiber is smooth but no other fiber is smooth.  Of course, this cannot happen if the family is projective or proper, since then the hyperbola describing the singular locus would intersect the closed fiber $X_c$ at infinity.
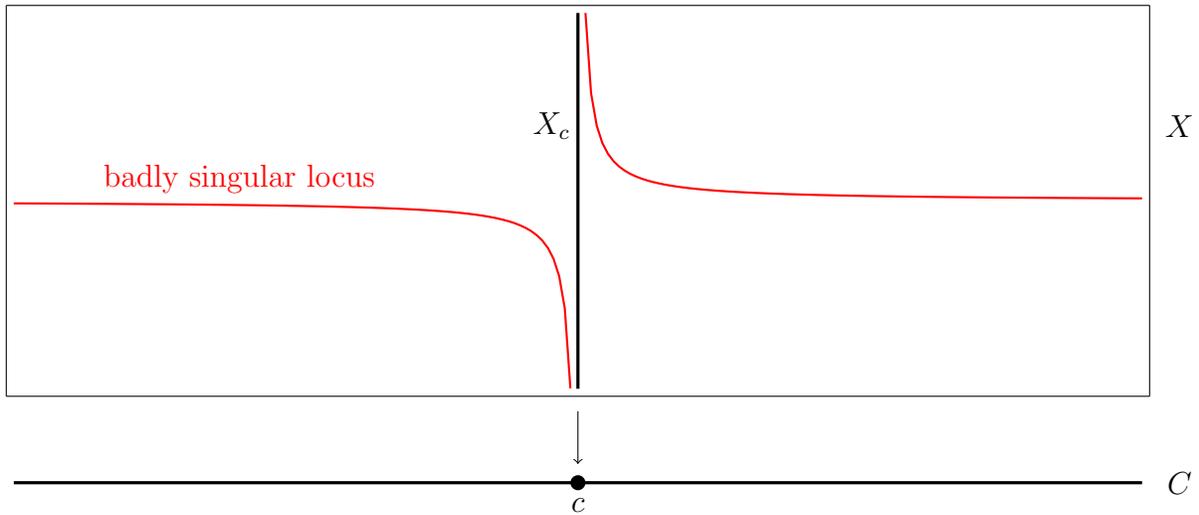
\begin{center}
\begin{tikzpicture}
\draw[scale=0.5,thick,domain=0.2:15,samples=100,range=-5:5smooth,variable=\x,red] plot ({\x},{1/\x});
\draw[scale=0.5,thick, domain=-15:-0.2,samples=100,range=-5:5smooth,variable=\x,red] plot ({\x},{1/\x});
\draw[very thick] (0,-2.5) -- (0, 2.5);
\draw (-7.6, 2.6) -- (7.6, 2.6);
\draw (-7.6, -2.6) -- (7.6, -2.6);
\draw (-7.6, 2.6) -- (-7.6, -2.6);
\draw (7.6, 2.6) -- (7.6, -2.6);
\draw[->] (0, -2.8) -- (0, -3.5);
\draw[very thick] (-7.5, -3.75) -- (7.5, -3.75);
\node at (8.0, -3.75) {$C$};
\node at (8.0, 1.0) {$X$};
\node at (0, -4.05) {$c$};
\fill[black] (0, -3.75) circle (0.1);
\node at (-0.35, 1) {$X_c$};
\node[red] at (-4.5, 0.3) {badly singular locus};
\end{tikzpicture}
\end{center}
There is another way to obtain the openness of the $F$-rational locus, and that is by working even more locally.

\begin{theorem}
\label{thm.LocalOpenFrationalOverZ}
Suppose that $R$ is of finite type and flat over $D$ as in \autoref{set.DedekindDomainSetting} and choose a prime ideal $Q \subseteq R$ with $(0) = Q \cap D$.  For each $t \in \Spec D$, write $\sqrt{t R + Q} = \bigcap_{i= 1}^{n_t} \bq_{t,i}$ a decomposition into minimal primes.  Then the following set is open in $\mSpec D$:
 \[
 W = \big\{ t \in \mSpec D \;\big|\; \{ \bq_{t, i} \}_{i=1}^{n_t} \text{ is nonempty and } (R/t)_{\bq_{t,i}} \text{ is $F$-rational for all $\bq_{t,i}$} \big \}.
 \]
\end{theorem}
Note a picture outlining the idea of the proof follows the proof.
\begin{proof}
Let $\rho : \Spec R \to \Spec D$ be the induced map. Note that $D\to R/Q$ is flat since $Q \cap D = (0)$. Thus $\rho(V(Q))$ is open and nonempty in $D$.  Hence the set
\[
\rho(V(Q)) \cap \mSpec(D) = \big\{ t \in \mSpec D\;\big|\; \{ \bq_{t, i} \}_{i=1}^{n_t} \text{ is non-empty } \}
\]
is nonempty and open in $\mSpec D$.  Since $Q \cap D = (0)$ we have $K \subseteq R_Q$ is of characteristic zero.  Further, we may assume that for some $t \in \mSpec D$, $(R/t)_{\bq_{t,i}}$ is $F$-rational for all $\bq_{t,i}$ since if not, then the set $W$ is empty.

Since $(R/t)_{\bq_{t,i}}$ is $F$-rational we see by \autoref{theorem: F-ratPseudorat} that $R_{\bq_{t,i}}$ is pseudo-rational, and so the further localization $R_Q$ has rational singularities (since it has characteristic $0$).  By inverting an element of $R$ not contained in $Q$, we can choose $R' \supseteq R$ a finitely generated $D$-algebra such that $\Spec (R' \otimes_{D} K)$ is an open neighborhood of $Q \in \Spec (R \otimes_D K)$ that has rational singularities (we abuse notation here and refer to more than one ideal as $Q$).  It follows from \cite{HaraRatImpliesFRat,MehtaSrinivasRatImpliesFRat} that for all but finitely many $t \in \Spec D$, $R'/t$ has $F$-rational singularities.

\begin{claim}
For all but finitely many $t$, $\{ \bq_{t, i} \}_{i=1}^{n_t} \subseteq \Spec R' \subseteq \Spec R$.
\end{claim}
\begin{proof}[Proof of Claim]
Note that $R' = R[h^{-1}]$ for some $h \in R \setminus Q$.  Observe that any set of infinitely many $\bq_{t_j, i}$s, is dense in $V(Q) \subseteq \Spec R$ (since it must vary over an infinite set of $t_j$s). Thus if $V(h)$ contains infinitely many $\bq_{t_j, i}$, then $V(h)$ contains $Q$, a contradiction. This proves the claim.
\end{proof}

Now that we have the claim, we simply observe that $(R'/t)_{\bq_{t,i}} = (R/t)_{\bq_{t,i}}$ for all $i = 1, \dots, n_t$ and all but finitely many $t$.
\end{proof}

\begin{remark}
The picture to keep in mind for the above theorem is the following:
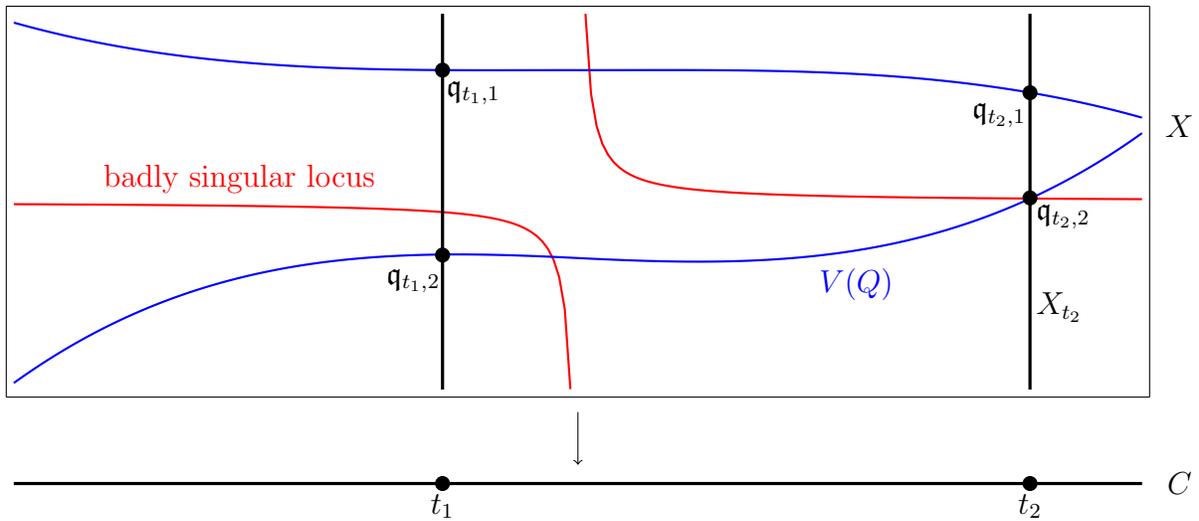
\begin{center}
\begin{tikzpicture}
\draw[scale=0.5,thick,domain=0.2:15,samples=100,range=-5:5smooth,variable=\x,red] plot ({\x},{1/\x});
\draw[scale=0.5,thick, domain=-15:-0.2,samples=100,range=-5:5smooth,variable=\x,red] plot ({\x},{1/\x});
\draw[scale=0.5,thick, domain=-15:15,samples=200, range=-5:5smooth, variable=\x,blue] plot ({\x}, {0.001*\x^3+0.003*\x^2-0.05*\x-1.5});
\draw[scale=0.5,thick, domain=-15:15,samples=200, range=-5:5smooth, variable=\x,blue] plot ({\x}, {-0.0005*\x^3+0.002*\x^2-0.001*\x+3.5});
\draw[very thick] (6.01,-2.5) -- (6.01, 2.5);
\draw[very thick] (-1.8,-2.5) -- (-1.8, 2.5);
\draw (-7.6, 2.6) -- (7.6, 2.6);
\draw (-7.6, -2.6) -- (7.6, -2.6);
\draw (-7.6, 2.6) -- (-7.6, -2.6);
\draw (7.6, 2.6) -- (7.6, -2.6);
\draw[->] (0, -2.8) -- (0, -3.5);
\draw[very thick] (-7.5, -3.75) -- (7.5, -3.75);
\node at (8.0, -3.75) {$C$};
\node at (8.0, 1.0) {$X$};
\node at (6.01, -4.05) {$t_2$};
\node at (-1.8, -4.05) {$t_1$};
\fill[black] (6.01, -3.75) circle (0.1);
\fill[black] (-1.8, -3.75) circle (0.1);
\fill[black] (6.01, 0.05) circle (0.1);
\fill[black] (6.01, 1.45) circle (0.1);
\fill[black] (-1.8, -0.71) circle (0.1);
\fill[black] (-1.8, 1.75) circle (0.1);
\node at (6.4, -1.4) {$X_{t_2}$};
\node at (6.45, -0.2) {$\bq_{t_2, 2}$};
\node at (5.6, 1.15) {$\bq_{t_2, 1}$};
\node at (-2.19, -1.05) {$\bq_{t_1, 2}$};
\node at (-1.39, 1.45) {$\bq_{t_1, 1}$};
\node[red] at (-4.5, 0.3) {badly singular locus};
\node[blue] at (3.7, -1.1) {$V(Q)$};
\end{tikzpicture}
\end{center}
Here a point $\bq_{t_2, 2} \in X_{t_2} = \Spec (R/t_2)$ is not $F$-rational but those $t_2 \in \Spec D$ are a closed subset that do not contain most points (like $t_1$).
\end{remark}

A slight rephrasing yields the following corollary.

\begin{corollary}
With notation as in \autoref{set.DedekindDomainSetting}, suppose that $(R_{K}, \fram_{K})$ is a local ring essentially of finite type over $K$.  Suppose that $(R_{D}, Q_{D} \in \Spec R_D)$ is a finite type and flat-over-$D$ model for $(R_{K}, \fram_K)$, so that $(R_D)_{Q_D} = (R_{D} \otimes_{D} K)_{Q_{D}} = R_{K}$.  Then the set of primes $t \in \mSpec D$ that satisfy the following condition
\begin{itemize}
\item{} $t R + Q_{D} \neq R_{D}$ and if $\bq_{t, i}$ is a minimal prime of $t R + Q_{D}$, then $R_{\bq_{t,i}}/t$ is $F$-rational.
\end{itemize}
is an open set of $\mSpec D$.
\end{corollary}

\subsection{$F$-regularity and log terminal singularities}

With minimal change, the results we have already obtained for $F$-rational/rational singularities also hold for $F$-regular/log terminal singularities.

\begin{theorem}
\label{thm.ProjectiveFamilyOneFRegImpliesAlmostall}
  Let $\phi: (X,\Delta \geq 0)\to U := \Spec D$ be a proper and flat family of pairs with notation as in \autoref{set.DedekindDomainSetting}. Suppose $K_X+\Delta$ is $\mathbb{Q}$-Cartier of index $N$ and suppose $(X_p, \Delta_p)$ is strongly $F$-regular for some closed point $p \in U$ whose residual characteristic does not divide $N$.  Then $(X_K, \Delta_K)$ is \emph{KLT}.  Furthermore $(X_q, \Delta_q)$ is \emph{KLT} for a Zariski dense and open set $V$ of closed points $q \in U$.
\end{theorem}
\begin{proof}
Let $Y$ be the non-KLT locus of $(X, \Delta)$.  The proof then follows essentially as in \autoref{thm.ProjectiveFamilyFratModPImpliesChar0} with the following modifications.
\begin{itemize}
\item{} We replace pseudo-rational/rational with KLT.
\item{} We replace \autoref{theorem: F-ratPseudorat} with \autoref{cor.FregDeforms} and notice that we must use the hypothesis that $\Delta$ has no vertical components to guarantee that it and anything we mod out by have no common components.
\item{} We replace \cite{HaraRatImpliesFRat,MehtaSrinivasRatImpliesFRat} by \cite{TakagiInterpretationOfMultiplierIdeals} (or see \cite[Proposition 6.10]{SchwedeRefinementsOfFRegularity} for the non-local case).
\end{itemize}
\end{proof}

Likewise we also obtain the following.

\begin{theorem}
\label{thm.LocalOpenFregularOverZ}
Suppose that $R$ is of finite type, flat, and generically geometrically normal over $D$ and $\Delta$ an effective $\bQ$-divisor as in \autoref{set.DedekindDomainSetting}.  Suppose that $K_R + \Delta$ is $\bQ$-Cartier with index $N$.   Finally choose a prime ideal $Q \subseteq R$ such that $Q \cap D = (0)$.  For each $t \in \Spec D$, write $\sqrt{t R + Q} = \bigcap_{i = 1}^{n_t} \bq_{t,i}$ a decomposition into minimal primes.  Then the set $W$ below is open in $\mSpec D$:
\[
\Big\{ t \in \mSpec D \;\Big|\; \mathrm{char}(D/t) \nmid N, \{ \bq_{t, i} \}_{i = 1}^{n_t} \neq \emptyset \text{ and each } ( (R/tR)_{\bq_{t,i}}, \Delta_t) \text{ is strongly $F$-regular} \Big\}.
\]
\end{theorem}
\begin{proof}
As in \autoref{thm.LocalOpenFrationalOverZ}, the set
\[
\rho(V(Q)) \cap \mSpec(D) = \big\{ t \in \mSpec D\;\big|\; \{ \bq_{t, i} \}_{i=1}^{n_t} \text{ is non-empty } \}
\]
is nonempty and open in $\mSpec D$.  We may further assume that there is some $t \in W$ as otherwise the statement is vacuous.

By \autoref{cor.FregDeforms} we know that $(R_Q, \Delta|_{R_Q})$ is KLT.  By inverting an element of $R$ not contained in $Q$, we can choose $R' \supseteq R$ a finitely generated $D$-algebra such that $\Spec (R' \otimes_{D} K)$ is an open neighborhood of $Q \in \Spec (R \otimes_D K)$ so that $(R', \Delta|_{R'})$ has KLT singularities (we abuse notation here and refer to more than one ideal as $Q$).  It follows from \cite[Theorem 5.2]{HaraRatImpliesFRat} or \cite[Corollary 3.4]{TakagiInterpretationOfMultiplierIdeals} that for all but finitely many $t \in \Spec D$, $(R'/t, \Delta_t)$ has strongly $F$-regular singularities.

The following claim from \autoref{thm.LocalOpenFrationalOverZ} is unchanged.

\begin{claim}
For all but finitely many $t$, $\{ \bq_{t, i} \} \subseteq \Spec R' \subseteq \Spec R$.
\end{claim}

Now that we have the claim, we simply see that $(R'/t)_{\bq_{t_i}} = (R/t)_{\bq_{t_i}}$ for all but finitely many $t$ and the result follows as in \autoref{thm.LocalOpenFrationalOverZ}.
\end{proof}

We make the following conjecture over higher dimensional bases.  We do not pursue this here however.

\begin{conjecture}
If one is working over a higher dimension base $D$ (instead of simply the spectrum of a Dedekind domain), then for a general finite type family $X \to \Spec D$, the locus of $t \in \mSpec D$ with $F$-rational $X_t$ is constructible.  Likewise if the family is flat and proper (or sufficiently local), then the same locus is open.
\end{conjecture}

\section{Discussion of algorithmic consequences}
\label{sec.AlgorithmicConsequences}

Suppose that $R$ is a ring of finite type over $\bQ$.  We would like to decide if $R$ has rational or log terminal singularities using a computer algebra system such as Macaulay2 \cite{M2}.  This appears to be quite difficult (not least because the most obvious strategy requires that one must first implement resolution of singularities in these environments, and that can be quite slow) and, at least to the authors' knowledge, has not been done outside of the case where the ring is a complete intersection (done via $D$-module techniques, see \cite{DmodulesSource}).  However, we do have methods for verifying that a ring is $F$-rational in a fixed characteristic $p > 0$.  And indeed, by \cite{SmithFRatImpliesRat}, it was known that if $(R_{\bZ})/p$ has $F$-rational singularities for some $p \gg 0$, then $R_{\bQ}$ has rational singularities as well.  The problem is that to determine if $p > 0$ is big enough, one had to already compute $\pi_* \omega_{\tld X_{\bZ}} \subseteq \omega_{X_{\bZ}}$ where $\pi: \tld X_{\bZ} \to X_{\bZ} = \Spec R_{\bZ}$ is a resolution of singularities.

The main results of this paper imply that one can use the following method to verify that a ring has rational singularities (although it cannot be used to show that a ring does not have rational singularities).  Note that related results for the log canonical threshold vs the $F$-pure threshold in a regular ambient ring were obtained in \cite{ZhuLCThresholdsInPositiveChar}.

\begin{algorithm}
With $R$ as above, choose a prime $Q \subseteq R$ such that you want to verify $R_Q$ has rational singularities:
\begin{description}
\item[Step 1]  Spread out $R$ to a domain $R_{\bZ} \subseteq R$ over $\bZ$ and also spread out $Q$ to a prime $Q_{\bZ}$ of $R_{\bZ}$.
\item[Step 2]  Choose a prime $p$ such that $Q_{\bZ} + (p) \neq R_{\bZ}$.
\item[Step 3]  Use the {\tt TestIdeals} package of Macaulay2 to check if $(R_{\bZ})/p$ has $F$-rational singularities.
\item[Step 4a]  If the answer to Step 3 is affirmative, then $R_Q$ has rational singularities.
\item[Step 4b]  If the answer to Step 3 is not affirmative, then return to Step 2 and choose a different prime.
\end{description}
\end{algorithm}

\begin{caveat}
Note that spreading out $R$ to a \emph{domain} already requires some checking, as one needs to verify that the given presentation $R_{\bZ}$ is in fact a domain.
\end{caveat}

As mentioned, this algorithm can verify that a ring has rational singularities but cannot show that a ring is not $F$-rational.  However, there are already a number of ways to do that.

\begin{itemize}
\item{} Show that $R$ is not Cohen-Macaulay.
\item{} Find a blowup $Y \to \Spec R$ such that $\pi_* \omega_Y \subsetneq \omega_R$ (this can be done without a resolution of singularities).
\end{itemize}

\bibliographystyle{skalpha}
\bibliography{MainBib}
\end{document}